\documentclass[a4paper, 12pt, reqno,final]{amsart}
\usepackage[foot]{amsaddr} 
\usepackage[utf8x]{inputenc}
\usepackage[T1]{fontenc}
\usepackage[english]{babel}
\usepackage{amsmath, amsfonts, amsthm, amssymb, amscd}
\usepackage[final]{graphicx}
\usepackage[below]{placeins}
\usepackage{subfig}
\usepackage{multirow}
\usepackage{mathtools}

\usepackage[hidelinks]{hyperref}
\hypersetup{
  pdfauthor = {Stefan Metzger},
  pdftitle = {SSAV}
}

\usepackage{showkeys} 
\usepackage{esint}
\usepackage{caption}
\usepackage{dsfont}

\usepackage{tikz}
\usetikzlibrary{patterns}
\usetikzlibrary{shapes.geometric}
\usepackage{yhmath}

\usepackage[]{todonotes}
\usepackage{diagbox}

\newtheorem{theorem}{Theorem}[section]
\newtheorem{lemma}[theorem]{Lemma}
\newtheorem*{lemma*}{Lemma}
\newtheorem{corollary}[theorem]{Corollary}

\newtheorem{remark}[theorem]{Remark}

\numberwithin{equation}{section}



\makeatletter
\newcommand{\labitem}[2]{%
\def\@itemlabel{\textbf{#1}}
\item
\def\@currentlabel{#1}\label{#2}}
\makeatother

\newcommand{\norm}[1]{\left\|{#1}\right\|}
\newcommand{\abs}[1]{\left|{#1}\right|}

\newcommand{\rkla}[1]{{\left(#1\right)}}
\newcommand{\trkla}[1]{{(#1)}}
\newcommand{\gkla}[1]{{\left\{#1\right\}}}
\newcommand{\tgkla}[1]{{\{#1\}}}
\newcommand{\skla}[1]{{\left\langle#1\right\rangle}}
\newcommand{\ekla}[1]{{\left[#1\right]}}
\newcommand{\tekla}[1]{{[#1]}}
\newcommand{\tabs}[1]{|{#1}|}

\newcommand{\bs}[1]{\boldsymbol{#1}}

\newcommand{\iOmega}{\int_{\Omega}}
\newcommand{\Om}{\mathcal{O}}

\newcommand{\iO}{\int_{\Om}}

\newcommand{\iGamma}{\int_\Gamma}

\newcommand{\para}[1]{\partial _{#1}}

\newcommand{\dx}{\, \mathrm{d}{x}}

\newcommand{\dt}{\, \mathrm{d}t}
\newcommand{\ds}{\, \mathrm{d}s}

\newcommand{\dG}{\,\mathrm{d}\Gamma}

\renewcommand{\div}{\operatorname{div}}
\newcommand{\nablaG}{\nabla_\Gamma}

\newcommand{\nn}{^{n}}

\newcommand{\no}{^{n-1}}

\newcommand{\tl}{^{\tau}}
\newcommand{\tp}{^{\tau,+}}
\newcommand{\tm}{^{\tau,-}}
\newcommand{\tpm}{^{\tau,(\pm)}}
\newcommand{\pml}{^{(\pm)}}

\newcommand{\h}{_{h}}
\newcommand{\hj}{_{h_j}}

\newcommand{\weak}{\rightharpoonup}
\newcommand{\weakstar}{\stackrel{*}{\rightharpoonup}}

\newcommand{\weaktop}{{\mathrm{weak}}}

\newcommand{\Uh}{U_{h}^{\Om}}
\newcommand{\Uhj}{U_{h_j}^{\Om}}
\newcommand{\Th}{\mathcal{T}_h}

\newcommand{\UhG}{U_{h}^{\Gamma}}
\newcommand{\UhGj}{U_{h_j}^{\Gamma}}
\newcommand{\ThG}{\mathcal{T}_h^{\Gamma}}

\newcommand{\Ihop}{\mathcal{I}_h}
\newcommand{\Ih}[1]{\Ihop\gkla{#1}}
\newcommand{\Ihjop}{\mathcal{I}_{h_j}}
\newcommand{\Ihj}[1]{\Ihjop\gkla{#1}}

\newcommand{\IhGop}{\mathcal{I}_h^\Gamma}
\newcommand{\IhG}[1]{\IhGop\gkla{#1}}
\newcommand{\IhGjop}{\mathcal{I}_{h_j}^\Gamma}
\newcommand{\IhGj}[1]{\IhGjop\gkla{#1}}

\newcommand{\restr}[2]{\ensuremath{
  \left.\kern-\nulldelimiterspace 
  #1 
  \vphantom{\big|} 
  \right|_{#2} 
  }}
  
\newcommand{\extend}[2]{\ensuremath{
  \left.\kern-\nulldelimiterspace 
  #1 
  \vphantom{\big|} 
  \right|^{#2} 
  }}
  
\newcommand{\diam}{\operatorname{diam}}

\newcommand{\g}[1]{\mathfrak{g}_{#1}}

\newcommand{\Zh}{\mathds{Z}\h}

\newcommand{\p}{{\mathfrak{p}}}

\newcommand{\expected}[1]{\mathds{E}\ekla{#1}}
\newcommand{\expectedt}[1]{\widetilde{\mathds{E}}\ekla{#1}}
\newcommand{\Prob}{\mathds{P}}

\DeclareMathOperator*{\esssup}{ess\,sup}

\newcommand{\Ito}{It\^{o}}

\newcommand{\sinc}[1]{\blacktriangle^{#1}\boldsymbol{\xi}^\tau}

\newcommand{\sinctilde}[1]{\blacktriangle^{#1}\widetilde{\boldsymbol{\xi}}^j}

\newcommand{\trace}[1]{\left.\ekla{#1}\right|_\Gamma}

\newcommand{\LapInv}{\mathfrak{D}}
\newcommand{\dom}{\mathrm{dom}\,}

\newcommand{\citeASAV}{Metzger2024_arxiv}

\begin{document}
\title[Contact line tension with thermal noise]{A convergent augmented SAV scheme for stochastic Cahn--Hilliard equations with dynamic boundary conditions describing contact line tension}
\date{\today}
\author[S.~Metzger]{Stefan Metzger}
\address[S.~Metzger]{Friedrich--Alexander Universität Erlangen--Nürnberg,~Cauerstraße 11,~91058~Erlangen,~Germany}
\email{stefan.metzger@fau.de}






%
\selectlanguage{english}

\allowdisplaybreaks

\begin{abstract}
We augment a thermodynamically consistent diffuse interface model for the description of line tension phenomena by multiplicative stochastic noise to capture the effects of thermal fluctuations and establish the existence of pathwise unique (stochastically) strong solutions.
By starting from a fully discrete linear finite element scheme, we do not only prove the well-posedness of the model, but also provide a practicable and convergent scheme for its numerical treatment.
Conceptually, our discrete scheme relies on a recently developed augmentation of the scalar auxiliary variable approach, which reduces the requirements on the time regularity of the solution.
By showing that fully discrete solutions to this scheme satisfy an energy estimate, we obtain first uniform regularity results.
Establishing Nikolskii estimates with respect to time, we are able to show convergence towards pathwise unique martingale solutions by applying Jakubowski's generalization of Skorokhod's theorem.
Finally, a generalization of the Gyöngy--Krylov characterization of convergence in probability provides convergence towards strong solutions and thereby completes the proof. 

\end{abstract}
\maketitle

\section{Introduction}
The description of the evolution of two immiscible fluids in a confined domain $\Om$ has been a significant research topic throughout the last centuries.
Thereby, the evolution of the three-phase contact line between the two fluids and the solid wall was of particular interest.
The investigation of mathematical formulas to predict the contact angles of a droplet wetting a solid substrate dates back to the beginning of the 19th century, when T.~Young \cite{Young1805} proposed the following famous formula for the equilibrium contact angle $\Theta$:
\begin{align}\label{eq:Young}
\tilde{\sigma}\cos\Theta=\gamma_{fs,1}-\gamma_{fs,2}\,.
\end{align}
Here, $\gamma_{fs,2}$ denotes the contact energy density between the wetting fluid and the solid substrate, $\gamma_{fs,1}$ denotes the contact energy density between the surrounding fluid (e.g.~air or vapor) and the substrate, and $\tilde{\sigma}$ describes the interfacial tension between the two fluids.
Relation \eqref{eq:Young} can be derived by minimizing the energy
\begin{align}
\mathcal{E}_1:=\int_{I_f}\tilde{\sigma}\dG+\int_{A}\trkla{\gamma_{fs,2}-\gamma_{fs,1}}\dG\,,
\end{align}
where $I_f$ denotes the fluid-fluid interface and $A\subset\partial\Om$ denotes the surface wetted by the droplet.
It was, however, already noted by J.~W.~Gibbs towards the end of the 19th century that contact line effects should be included (cf.~Chapter III in \cite{Gibbs1961}).
Yet Gibbs did not provide any mathematical formulation.
Since then, the influence of contact lines on the contact angle was investigated by many authors (see e.g.~\cite{Toshev1988, Widom95, SwainLipowsky98, AmirfazliNeumann2004, BlecuaLipowskyKierfeld06, Law_etal2017, ZhangWangNestler2023}, and the references therein).
In order to include contact line effects into the description, the energy $\mathcal{E}_1$ has to be augmented by an additional contact line integral, i.e.
\begin{align}
\mathcal{E}_2:=\int_{I_f}\tilde{\sigma}\dG+\int_A\trkla{\gamma_{fs,2}-\gamma_{fs,1}}\dG+\int_{I_f\cap\partial\Om}\tilde{\kappa}\ds\,.
\end{align}
Here, $I_f\cap\partial\Om$ is the three-phase contact line on the boundary $\partial\Om$ of the fluid domain and $\tilde{\kappa}$ denotes the line tension.
Although the line tension can be negative, we will restrict ourselves to the case $\tilde{\kappa}>0$.
Assuming that the droplet is spherical (or a spherical cap when attached to the substrate) and minimizing $\mathcal{E}_2$ leads to the following formula for the stationary contact line (cf.~\cite{Widom95}):
\begin{align}\label{eq:contactlinetension}
\tilde{\sigma}\cos\Theta=\trkla{\gamma_{fs,1}-\gamma_{fs,2}}+\frac{\tilde{\kappa}}{r}\,.
\end{align}
Formula \eqref{eq:contactlinetension} indicates that Young's original description \eqref{eq:Young} predicts the stationary contact angle sufficiently well, if the radius $r$ of the circular contact area $A$ is sufficiently large.
For smaller droplets with smaller contact areas, however, this effect is able to cause significant deviations from the results predicted by Young's formula (see e.g.~\cite{Widom95}).
As it was noted in \cite{SwainLipowsky98} these descriptions should also take thermally excited fluctuations into account, as they are able to change the contact line contour and thereby modify the contact angle.\\
In this publication, we analyze a diffuse interface model with dynamic boundary conditions describing contact line tension effects including thermal fluctuations.
The basic idea of a diffuse interface model is to replace surface (or line) Dirac functions of $I_f$ (or $I_f\cap\partial\Om$, respectively) by smooth approximations of the form $\tfrac12\varepsilon\tabs{\nabla\phi}^2+ \varepsilon^{-1}F\trkla{\phi}$, where $\phi$ is the phase-field parameter describing the two fluid phases, $F$ is a double-well potential with minima in (or close to) the values $\pm1$ indicating the pure phases, and $\varepsilon$ is a small parameter related to the width of the diffuse interface.
Typical choices for this double-well potential are the logarithmic double-well potential
\begin{align}
W_{\log}\trkla{\phi}:=\frac{\vartheta}2\trkla{1+\phi}\log\trkla{1+\phi}+\frac{\vartheta}2\trkla{1-\phi}\log\trkla{1-\phi}-\frac{\vartheta_c}2\phi^2
\end{align}
with $0<\vartheta<\vartheta_c$, the double obstacle potential
\begin{align}
W_{\operatorname{obst}}\trkla{\phi}:=\left\{\begin{array}{cc}
\vartheta\trkla{1-\phi^2}&\text{if~}\phi\in\tekla{-1,+1}\\
\infty&\text{else}
\end{array}\right.
\end{align}
with $\vartheta>0$, and the polynomial double-well potential $W_{\operatorname{pol}}\trkla{\phi}:=\tfrac14\trkla{\phi^2-1}^2$.
The logarithmic and double obstacle potentials are of great analytical interest, as they allow to confine the phase-field parameter to the physical meaningful interval $\tekla{-1,+1}$.
The numerical treatment of these potentials, however, is rather intricate (cf.~\cite{CopettiElliott92, Blowey1992, Blowey1996, Barrett1999, Barrett2001, Frank2020}) and most numerical schemes are based on the polynomial double-well potential $W_{\operatorname{pol}}$.
Assuming $F\trkla{\phi}\equiv W_{\operatorname{pol}}\trkla{\phi}$, we approximate $\mathcal{E}_2$ by
\begin{multline}
\widetilde{\mathcal{E}}\trkla{\phi}:=\iO\sigma\rkla{\frac{\varepsilon}2\abs{\nabla\phi}^2+\frac1\varepsilon F\trkla{\phi}}\dx+\int_{\partial\Om}\gamma_{fs}\trkla{\phi}\dG \\
+\int_{\partial\Om}\kappa\rkla{\frac\varepsilon2\abs{\nablaG\phi}^2+\frac1\varepsilon F\trkla{\phi}}\dG\,,
\end{multline}
with $\gamma_{fs}$ interpolating between $\gamma_{fs,1}$ and $\gamma_{fs,2}$, $\sigma$ and $\kappa$ being rescaled versions of $\tilde{\sigma}$ and $\tilde{\kappa}$, and $\nablaG$ denoting the surface gradient.
A deterministic evolution of the phase-field that minimizes $\widetilde{\mathcal{E}}$ and conserves $\iO\phi\dx$ can be described by
\begin{subequations}\label{eq:CHAC}
\begin{align}
\para{t}\phi=&\Delta\mu&&\text{in~}\Om\,,\\
\mu=&-\sigma\varepsilon\Delta\phi+\sigma\varepsilon^{-1} F^\prime\trkla{\phi}&&\text{in~}\Om\,,\\
\para{t}\phi =&-\gamma^\prime_{fs}\trkla{\phi}-\kappa\trkla{-\varepsilon\Delta_\Gamma\phi+\varepsilon^{-1} F^\prime\trkla{\phi}}-\sigma\varepsilon\nabla\phi\cdot\bs{n}&&\text{on~}\partial\Om\,,\\
\nabla\mu\cdot\bs{n}=&0&&\text{on~}\partial\Om\,,
\end{align}
\end{subequations}
i.e.~by a Cahn--Hilliard equation with Allen--Cahn-type dynamic boundary conditions.
For a more rigorous derivation of \eqref{eq:CHAC} and a discussion of the involved parameters, we refer the reader to \cite{ZhangWangNestler2023}.
Similar equations were derived in \cite{Qian2006}, where \eqref{eq:CHAC} with $\kappa=0$ was coupled with suitable Navier--Stokes equations to describe moving contact lines.
Such Allen--Cahn-type boundary conditions have been extensively studied.
A by no means exhausting list of contributions includes e.g.~\cite{Fischer1997, Fischer1998, Kenzler2001, RackeZheng2003, Wu2004, ChillFasangovaPruess2006, Gilardi2009, Miranville2010, Cherfils2011, Colli2015, Mininni2017}.
As in this publication we are not interested in the sharp interface limit $\varepsilon\searrow0$, we will simplify the representation of \eqref{eq:CHAC} by setting $\varepsilon=\sigma=\kappa=1$ and introduce $G\trkla{\phi}:=F\trkla{\phi}+\gamma_{fs}\trkla{\phi}$.\\
To include thermal fluctuations into the model, we consider a $\mathcal{Q}$-Wiener process $W=\trkla{W_t}_{t\in\tekla{0,T}}$ defined on a filtered probability space $\trkla{\Omega,\mathcal{A},\mathcal{F},\Prob}$.
The exact assumptions on the $\mathcal{Q}$-Wiener process are listed in Assumptions \ref{item:W1} and \ref{item:W2} in Section \ref{sec:preliminaries}.
They in particular imply that its trace $\trkla{\trace{W_t}}_{t\in\tekla{0,T}}$ on $\Gamma:=\partial\Omega$ is well-defined.
With this $\mathcal{Q}$-Wiener process we augment the system \eqref{eq:CHAC} by multiplicative It\^o noise:
\begin{subequations}\label{eq:model}
\begin{align}
\mathrm{d}\phi-\Delta\mu\dt&\,=\Phi\trkla{\phi}\mathrm{d}W&&\text{in~}\Om\,,\\
\mu=-\Delta\phi+&\,F^\prime\trkla{\phi}&&\text{in~}\Om\,,\\
\mathrm{d}\!\trace{\phi} + \trkla{-\Delta_\Gamma \trace{\phi}+G^\prime\trkla{\trace{\phi}}+ \nabla\phi\cdot\bs{n}}\dt&=\trace{\Phi\trkla{\phi}\mathrm{d}W}&&\text{on~}\partial\Om\,,\\
\nabla\mu\cdot\bs{n}=&\,0 &&\text{on~}\partial\Om\,
\end{align}
\end{subequations}
with $\trace{\,\cdot\,}$ denoting the trace operator, i.e.~we model the noise on the boundary $\Gamma=\partial\Om$ as the trace of the noise in the bulk.
The operator $\Phi$ maps the stochastic process $\phi$ into the space of Hilbert--Schmidt operators from $\mathcal{Q}^{1/2} L^2\trkla{\Om}$.
For a detailed definition, we refer the reader to Section \ref{sec:preliminaries} below.\\
To show the existence of solutions to \eqref{eq:model}, we will start from a fully discrete finite element scheme.
Hence, we will not only prove the well-posedness of \eqref{eq:model} but also establish convergence for our numerical scheme.
The considered scheme relies on an augmented version of the scalar auxiliary variable (SAV) approach.
Originally, the SAV approach was introduced in \cite{ShenXuYang2018} for deterministic PDEs describing gradient flows.
This approach has been applied to various deterministic problems (see e.g.~\cite{ShenXuYang19, ShenYang20} and the references therein) and many variations of this approch have been developed and tested (see e.g.~\cite{HouAzaiezXu19, LiuLi20, HuangShenYang2020, YangDong2020, ZhangShen2022,LamWang2023}).
As this approach provides linear schemes, it allows for a significant reduction in computation time.
Hence, an application to stochastic PDEs, where often multiple different paths need to be simulated, is tempting.
Yet, a straightforward application of the standard SAV scheme to stochastic PDEs is not always crowned with success.
Although there are positive results for the stochastic wave equation (see e.g.~\cite{CuiHongSun2022_arxiv}), for most SPDEs the poor time regularity of the solutions impedes convergence results.
As shown in \cite{\citeASAV} this is not only an artificial analytical problem that jeopardizes rigorous convergence proofs, but can also lead to wrong results in practical simulations.
To overcome these difficulties, the author proposed an augmented version of the SAV scheme in \cite{\citeASAV}, which extends the applicability of the SAV schemes to stochastic PDEs with less regular solutions.\\

The outline of this paper is as follows:
In Section \ref{sec:preliminaries}, we will introduce the relevant function spaces and interpolation operators.
We will also collect our assumptions on the data and important auxiliary results.
In Section \ref{sec:discretescheme}, we present the numerical scheme. The main convergence results of this publication can be found in Section \ref{sec:mainresults}:
Theorem \ref{thm:mainresult} provides convergence towards pathwise unique martingale solutions and Theorem \ref{thm:strongsolutions} provides for a given $\mathcal{Q}$-Wiener process the convergence of discrete solutions towards (stochastically) strong solutions.
The results are proven in the remaining sections:
The existence of solutions to the numerical scheme is established in Section \ref{sec:existence}.
Uniform regularity results for these solutions are established in Section \ref{sec:regularity}.
In Section \ref{sec:compactness}, we apply Jakubowski's theorem to identify weakly and strongly converging subsequences based on the prior established regularity results.
In Section \ref{sec:limit}, we pass to the limit $\trkla{h,\tau}\searrow0$ and establish the convergence towards (and existence of) martingale solutions.
Exploiting monotonicity arguments in Section \ref{sec:uniqueness}, we show that these martingale solutions are pathwise unique which concludes the proof of Theorem \ref{thm:mainresult}.
The proof of Theorem \ref{thm:strongsolutions} can be found in Section \ref{sec:strongsolutions}, where we apply a generalized version of the Gyöngy--Krylov characterization of convergence in probability to show that for a given $\mathcal{Q}$-Wiener process our discrete solutions converge towards strong solutions to \eqref{eq:model}.

\section{Notation and assumptions}\label{sec:preliminaries}
The spatial domain $\Om\subset\mathds{R}^d$ with $d\in\tgkla{2,3}$ is assumed to be bounded and convex with lipschitzian boundary $\Gamma:=\partial\Om$.
To avoid additional technicalities, we shall assume that $\Om$ is polygonal (or polyhedral, respectively).
We denote the space of $k$-times weakly differentiable functions with weak derivatives in $L^p\trkla{\Om}$ by $W^{k,p}\trkla{\Om}$.
For $p=2$, these spaces are Hilbert spaces, which we shall denote by $H^k\trkla{\Om}:=W^{k,2}\trkla{\Om}$.
The dual space of $H^1\trkla{\Om}$ will be denoted by $\trkla{H^1\trkla{\Om}}^\prime$.
The symbol $\skla{.,.}$ stands for the duality pairing between $\trkla{H^1\trkla{\Om}}^\prime$ and $H^1\trkla{\Om}$ which satisfies $\skla{u,v}=\trkla{u,v}_{L^2\trkla{\Om}}$ for $u\in L^2\trkla{\Om}$ and $v\in H^1\trkla{\Om}$.
For $\varphi\in\trkla{H^1\trkla{\Om}}^\prime$ we define a generalized mean value by
\begin{align}
\trkla{\phi}_{\Om}:=\tabs{\Om}^{-1}\skla{\varphi,1}\,.
\end{align}
We shall use a similar notation for function space on $\Gamma$.
As we always assume that the domain $\Om$ has a lipschitzian boundary, the spaces $L^p\trkla{\Gamma}$ and $W^{1,p}\trkla{\Gamma}$ are well-defined for all $p\geq1$ (cf.~\cite{Kufner77}) and the trace operator $\trace{\cdot}$ is uniquely defined as an element of $\mathcal{L}\trkla{W^{1,p}\trkla{\Om};W^{1-1/p,p}\trkla{\Gamma}}$ (cf.~\cite{Necas2012}).\\
We also define the space
\begin{align}
\mathcal{H}^1:= H^1\trkla{\Om}\times H^1\trkla{\Gamma}\,,
\end{align}
which is a Hilbert space with respect to the inner product
\begin{align}
\trkla{\trkla{\phi,\psi},\trkla{\zeta,\xi}}_{\mathcal{H}^1}:=\trkla{\phi,\zeta}_{H^1\trkla{\Om}}+\trkla{\psi,\xi}_{H^1\trkla{\Gamma}}
\end{align}
for all $\trkla{\phi,\psi},\trkla{\zeta,\xi}\in\mathcal{H}^1$.
Furthermore, we introduce the following subspace  of $H^1\trkla{\Om}$:
\begin{align}
\mathcal{V}:=\tgkla{\varphi\in H^1\trkla{\Om}\,:\,\trkla{\varphi,\trace{\varphi}}\in\mathcal{H}^1}\,.
\end{align}
For a Banach space $X$ and a set $I$, the symbol $L^p\trkla{I;X}$ ($p\in\trkla{1,\infty}$) stands for the Bochner space of strongly measurable $L^p$-integrable functions on $I$ with values in $X$.
If $X$ is only separable (and not reflexive), we follow the notation used in \cite[Chapter 0.3]{FeireislNovotny2009} and denote the dual space of $L^{p/\trkla{p-1}}\trkla{I;X}$ ($p\in\trkla{1,\infty}$) by $L^p_{\operatorname{weak-}\trkla{*}}\trkla{I;X^\prime}$.
A space $X$ endowed with the weak topology is denoted by $X_{\weaktop}$.\\
By $C^{k,\alpha}\trkla{I;X}$ with $k\in\mathds{N}_0$ and $\alpha\in(0,1]$, we denote the space of $k$-times continuously differentiable functions from $I$ to $X$ whose $k$-th derivatives are Hölder continuous with Hölder exponent $\alpha$.
If $I=X=\mathds{R}$, we shall write $C^{k,\alpha}\trkla{\mathds{R}}$.
We shall also introduce the Nikolskii space $N^{\alpha,\beta}$ ($\alpha\in\trkla{0,1},\,\beta\in\tekla{1,\infty}$) which are defined for a time interval $\trkla{0,T}$ and a Banach space $X$ via
\begin{align}
N^{\alpha,\beta}\trkla{0,T;X}:=\gkla{f\in L^\beta\trkla{0,T;X}\,:\,\sup_{k>0}k^{-\alpha}\norm{f\trkla{\cdot+k}-f\trkla{\cdot}}_{L^\beta\trkla{-k,T-k;X}}<\infty}\,.
\end{align}
Here, the standard convention $f\equiv0$ outside of $\trkla{0,T}$ for $f\in L^\beta\trkla{0,T;X}$ is used.

Concerning the discretization with respect to time, we shall assume that
\begin{itemize}
\labitem{\textbf{(T)}}{item:time} the time interval $I:=\tekla{0,T}$ is subdivided into $N$ equidistant intervals $I_n$ given by nodes $\trkla{t\nn}_{n=0,\ldots,N}$ with $t^0=0$, $t^N=T$, and $t\nn-t\no=\tau=:\tfrac{T}{N}$. Without loss of generality, we assume $\tau<1$.
\end{itemize}
Concerning the spacial domain, we consider a bounded and convex domain $\Om\subset\mathds{R}^d$ ($d\in\tgkla{2,3}$) with boundary $\Gamma=\partial\Om$.
Furthermore, we shall assume that $\Om$ is polygonal (or polyhedral, respectively) to avoid additional technicalities.
The spatial discretization is based on partitions $\Th$ of $\Om$ depending on a discretization parameter $h>0$ satisfying the following assumption:
\begin{itemize}
\labitem{\textbf{(S1)}}{item:space1} Let $\tgkla{\Th}_{h>0}$ be a quasiuniform family (in the sense of \cite{BrennerScott}) of partitions of $\Om$ into disjoint, open simplices $K$, satisfying
\begin{align*}
\overline{\Om}\equiv\bigcup_{K\in\Th}\overline{K}\qquad\text{with~}\max_{K\in\Th}\diam\trkla{K}\leq h\,.
\end{align*}

\end{itemize}
For the discretization of $\Gamma$, we follow the ideas employed in \cite{Metzger2021a, KnopfLamLiuMetzger2021, Metzger2023} and use a partition consisting of the edges (or faces, respectively) of elements of $\Th$. 
In particular, we shall consider partitions $\ThG$ of $\Gamma$ satisfying the following assumption:
\begin{itemize}
\labitem{\textbf{(S2)}}{item:space2} Let $\tgkla{\ThG}_{h>0}$ be a quasiuniform family of partitions of $\Gamma$ into disjoint open simplices $K^\Gamma$, satisfying
\begin{align*}
&&&\forall K^\Gamma\in\ThG\qquad\exists!K\in\Th\text{~such that~}\overline{K^\Gamma}=\overline{K}\cap\Gamma\,,\\
\text{and}&&&\Gamma\equiv\bigcup_{K^\Gamma\in\ThG}\overline{K^\Gamma}\qquad\text{with~}\max_{K^\Gamma\in\ThG}\diam\trkla{K^\Gamma}\leq h\,.
\end{align*}
\end{itemize}
For the approximation of the phase-field $\phi$ and the chemical potential $\mu$, we use continuous, piecewise linear finite element functions on $\Th$.
This space will be denoted by $\Uh$ and is spanned by functions $\tgkla{\chi_{h,k}}_{k=1,\ldots,\dim \Uh}$ forming a dual basis to the vertices $\tgkla{\bs{x}_{h,k}}_{k=1,\ldots,\dim\Uh}$ of $\Th$.
For the discretization of $\trace{\phi}$ and $\theta$, we use continuous, piecewise linear finite element functions on $\ThG$ which shall be denoted by $\UhG$.
Due to the assumptions on $\Th$ and $\ThG$, this space is given by traces of functions in $\Uh$, i.e.
\begin{align}\label{eq:tracespace}
\UhG:=\operatorname{span}\tgkla{\trace{\psi\h}\,:\,\psi\h\in\Uh}\,.
\end{align}
This space can also be described as the span of functions $\tgkla{\chi_{h,k}^\Gamma}_{k=1,\ldots,\dim\UhG}$, which form a dual basis to the vertices $\tgkla{\bs{x}_{h,k}^\Gamma}_{k=1,\ldots,\dim\UhG}\subset\tgkla{\bs{x}_{h,j}}_{j=1,\ldots,\dim\Uh}$ of $\ThG$.
For a different discretization approach which relaxes the cennection between $\Uh$ and $\UhG$ by introducing Lagrange multipliers, we refer to \cite{AltmannZimmer23} and the references therein.
We define the nodal interpolation operators $\Ihop\,:\,C^0\trkla{\overline{\Om}}\rightarrow\Uh$ and $\IhGop\,:\,C^0\trkla{\Gamma}\rightarrow\UhG$ by
\begin{align}
\Ih{a}:=\sum_{k=1}^{\dim\Uh}a\trkla{\bs{x}_{h,k}}\chi_{h,k}\,,&&\text{and}&&\IhG{a}:=\sum_{k=1}^{\dim\UhG}a\trkla{\bs{x}_{h,k}^\Gamma}\chi_{h,k}^\Gamma\,.
\end{align}
For future reference, we state the following norm equivalences for $p\in[1,\infty)$ and $f\h\in\Uh$, $g\h\in\UhG$:
\begin{subequations}\label{eq:normequivalence}
\begin{align}
c\rkla{\iO\tabs{f\h}^p\dx}^{1/p}\leq&\, \rkla{\iO\Ih{\tabs{f\h}^p}\dx}^{1/p}\leq C\rkla{\iO\tabs{f\h}^p\dx}^{1/p}\,,\\
c\rkla{\iGamma\tabs{g\h}^p\dG}^{1/p}\leq &\,\rkla{\iGamma\IhG{\tabs{g\h}^p}\dG}^{1/p}\leq C\rkla{\iGamma\tabs{g\h}^p\dG}^{1/p}
\end{align}
\end{subequations}
with $c,C>0$ independent of $h$.
The proof of this estimate relies on standard inverse estimates and can be found e.g.~in \cite[Lemma 3.2.11]{SieberDiss2021}.
To simplify the notation, we introduce the discrete norms
\begin{subequations}\label{eq:def:hnorms}
\begin{align}
\norm{\zeta\h}_{h,\Om}&:=\sqrt{\iO\Ih{\tabs{\zeta\h}^2}\dx}&\text{and}&&\norm{\zeta\h}_{H\h^1\trkla{\Om}}&:=\sqrt{\norm{\zeta\h}_{h,\Om}^2+\norm{\nabla\zeta\h}_{L^2\trkla{\Om}}^2}\,,\\
\norm{\widehat{\zeta}\h}_{h,\Gamma}&:=\sqrt{\iGamma\IhG{\abs{\widehat{\zeta}\h}^2}\dx}&\text{and}&&\norm{\widehat{\zeta}\h}_{H\h^1\trkla{\Gamma}}&:=\sqrt{\norm{\widehat{\zeta}\h}_{h,\Gamma}^2+\norm{\nablaG\widehat{\zeta}\h}_{L^2\trkla{\Gamma}}^2}
\end{align}
\end{subequations}
on $\Uh$ and $\UhG$.
Due to \eqref{eq:normequivalence}, these norms are equivalent to the usual $L^2$- and $H^1$-norms on $\Omega$ and $\Gamma$, respectively.
Furthermore, we state the following lemma that was proven in \cite{Metzger2020}:
\begin{lemma}\label{lem:interpolationerror}
Let $\Th$ and $\ThG$ satisfy \ref{item:space1} and \ref{item:space2}, respectively. 
Furthermore, let $p\in[1,\infty)$, $1\leq q\leq \infty$ and $q^*=\tfrac{q}{q-1}$ for $q<\infty$ or $q^*=1$ for $q=\infty$.
Then
\begin{subequations}
\begin{align}
\norm{\trkla{1-\Ihop}\tgkla{f\h g\h}}_{L^p\trkla{\Om}}\leq& \,Ch^2\norm{\nabla f\h}_{L^{pq}\trkla{\Om}}\norm{\nabla g\h}_{L^{pq^*}\trkla{\Om}}\,,\\
\norm{\trkla{1-\Ihop}\tgkla{f\h g\h}}_{W^{1,p}\trkla{\Om}}\leq &\, Ch\norm{\nabla f\h}_{L^{pq}\trkla{\Om}}\norm{\nabla g\h}_{L^{pq^*}\trkla{\Om}}\,,\\
\norm{\trkla{1-\IhGop}\tgkla{u\h v\h}}_{L^p\trkla{\Gamma}}\leq&\,Ch^2\norm{\nablaG u\h}_{L^{pq}\trkla{\Gamma}}\norm{\nablaG v\h}_{L^{pq^*}\trkla{\Gamma}}\,,\\
\norm{\trkla{1-\IhGop}\tgkla{u\h v\h}}_{W^{1,p}\trkla{\Gamma}}\leq &\,Ch\norm{\nablaG u\h}_{L^{pq}\trkla{\Gamma}}\norm{\nablaG v\h}_{L^{pq^*}\trkla{\Gamma}}
\end{align}
\end{subequations}
hold true for all $f\h,g\h\in\Uh$ and $u\h,v\h\in\UhG$.
\end{lemma}
Using the nodal interpolation operator, we define the discrete Laplacian $\Delta\h$ on $\Om$ via
\begin{align}\label{eq:def_discLaplace}
\iO\Ih{\Delta\h\zeta\h\psi\h}\dx=-\iO\nabla\zeta\h\cdot\nabla\psi\h\dx
\end{align}
for $\zeta\h,\psi\h\in\Uh$.\\
Concerning the potentials $F$ and $G$, we postulate assumptions similar to the ones used in \cite{\citeASAV}. 
In particular, we shall assume that
\begin{itemize}
\labitem{\textbf{(P)}}{item:potentials} $F,G\in C^2\trkla{\mathds{R}}$ are bounded from below by a positive constant $\gamma>0$ and satisfy the growth estimates
\begin{align*}
c_\gamma\trkla{1+\tabs{\zeta}^4}\leq&\, F\trkla{\zeta}\leq C\trkla{1+\tabs{\zeta}^4}\,,&\tabs{F^\prime\trkla{\zeta}}\leq &\,C\trkla{1+\tabs{\zeta}^3}\,,&\tabs{F^{\prime\prime}\trkla{\zeta}}\leq&\,C\trkla{1+\tabs{\zeta}^2}\,,\\
c_\gamma\trkla{1+\tabs{\zeta}^4}\leq&\, G\trkla{\zeta}\leq C\trkla{1+\tabs{\zeta}^4}\,,&\tabs{G^\prime\trkla{\zeta}}\leq &\,C\trkla{1+\tabs{\zeta}^3}\,,&\tabs{G^{\prime\prime}\trkla{\zeta}}\leq&\,C\trkla{1+\tabs{\zeta}^2}
\end{align*}
for all $\zeta\in \mathds{R}$ and a positive constant $c_\gamma$.\\
Furthermore, $F$ and $G$ can be decomposed into $F_1, G_1\in C^{2,\nu}\trkla{\mathds{R}}$ with $\nu\in\trkla{0,1}$ and $F_2,G_2\in C^3\trkla{\mathds{R}}$, $\tabs{F_2^{\prime\prime\prime}\trkla{\zeta}}\leq C\trkla{1+\tabs{\zeta}^2}$, and $\tabs{G_2^{\prime\prime\prime}\trkla{\zeta}}\leq C\trkla{1+\tabs{\zeta}^2}$.
\end{itemize}
\begin{remark}
Having a positive lower bound for $F$ and $G$ is a purely technical assumptions which is required to state the SAV scheme.
In particular, we will need that $\rkla{\iO F\trkla{\phi}\dx}^{-1/2}$ and $\rkla{\iGamma G\trkla{\phi}\dG}^{-1/2}$ are well-defined and bounded.
As the system stated in \eqref{eq:model} only depends on the derivatives of $F$ and $G$, it is always possible to add arbitrary constants without changing \eqref{eq:model}.
Hence, any lower bounds for $F$ and $G$ are sufficient.
The necessary shift can then be interpreted as a numerical parameter.\\
The assumptions stated in \ref{item:potentials} are satisfied for instance by a shifted polynomial double-well potential $\widetilde{W}_{\operatorname{pol}}\trkla{\phi}:=\tfrac14\trkla{\phi^2-1}^2+\gamma$ with $\gamma>0$. 
We are, however, not limited to this specific choice.
As stated before in \eqref{eq:CHAC}, we also consider a wetting energy density $\gamma_{fs}$ which can be chosen e.g.~as
\begin{align}
\gamma_{fs}\trkla{\phi}:=\left\{\begin{array}{cc}
\frac{\gamma_{fs,2}-\gamma_{fs,1}}{2}\rkla{\tfrac38\phi^5-\tfrac54\phi^3+\tfrac{15}5\phi }+\frac{\gamma_{fs,2}+\gamma_{fs,1}}{2}&\text{if~}\phi\in\tekla{-1,+1}\,,\\
\gamma_{fs,1}&\text{if~}\phi<-1\,,\\
\gamma_{fs,2}&\text{if~}\phi>1\,
\end{array}\right.
\end{align}
to satisfy assumption \ref{item:potentials}.
An admissible choice obtained by merging this wetting energy density with the polynomial double-well potential is e.g.
\begin{align}
G\trkla{\phi}=\widetilde{W}_{\operatorname{pol}}+\gamma_{fs}\trkla{\phi} -\min\tgkla{\gamma_{fs,1},\gamma_{fs,2}}\,.
\end{align}

In comparison to the assumptions used in \cite{Metzger2023} for the numerical treatment of a deterministic Cahn--Hilliard--Cahn--Hilliard-system, the assumptions in \ref{item:potentials} are stricter.
In particular, the decrease in regularity caused by the additional stochastic noise terms requires us to enhance our discrete scheme by additional terms involving higher derivatives of the potentials (cf. Section \ref{sec:discretescheme}).
Hence, the potentials used in the stochastic setting need to allow for one derivative more than their counterparts in the deterministic setting.
The growth conditions imposed on $F$ and $G$ are necessary to control the source terms using a Gronwall argument.
As discussed in \cite[Remark 5]{LamWang2023} these growth conditions seem to be sharp even for deterministic source terms.
\end{remark}

For simplicity, we consider deterministic initial data and assume that
\begin{itemize}
\labitem{\textbf{(I)}}{item:initial} $\phi_0\in H^2\trkla{\Om}$.
The discrete initial data $\phi\h^0\in\Uh$ is then obtained via $\phi\h^0:=\Ih{\phi_0}$.
\end{itemize}
Immediate consequences of \ref{item:initial} are
\begin{subequations}
\begin{align}
&&\norm{\phi\h^0}_{W^{1,6}\trkla{\Om}}+\norm{\trace{\phi\h^0}}_{W^{1,4}\trkla{\Gamma}}\leq&\, C\trkla{\phi_0}\,,\\
\text{and}&&\norm{\phi\h^0-\phi_0}_{H^1\trkla{\Om}}+\norm{\trace{\phi\h^0}-\trace{\phi_0}}_{L^4\trkla{\Gamma}}\leq&\,C\trkla{\phi_0}h\,.
\end{align}
\end{subequations}
The stochastic source terms in \eqref{eq:model} are governed by a $\mathcal{Q}$-Wiener process and an operator $\Phi$.
Throughout this paper, we shall assume that
\begin{itemize}
\labitem{\textbf{(W1)}}{item:W1} the trace class operator $\mathcal{Q}\,:\,L^2\trkla{\Om}\rightarrow L^2\trkla{\Om}$ satisfies
\begin{align}
\mathcal{Q}g:=\sum_{k\in\mathds{Z}}\lambda_k^2\trkla{g,\g{k}}_{L^2\trkla{\Om}}\g{k}\,,
\end{align}
where $\trkla{\lambda_k}_{k\in\mathds{Z}}$ are given real numbers, and $\trkla{\g{k}}_{k\in\mathds{Z}}$ is an orthonormal basis of $L^2\trkla{\Om}$.
\end{itemize}
Hence, $\mathcal{Q}^{1/2}$ given by $\mathcal{Q}^{1/2}g:=\sum_{k\in\mathds{Z}}\lambda_k\trkla{g,\g{k}}_{L^2\trkla{\Om}}\g{k}$ is a Hilbert--Schmidt operator from $L^2\trkla{\Om}$ to $L^2\trkla{\Om}$.
We shall denote its image of $L^2\trkla{\Om}$ by
\begin{align}
\mathcal{Q}^{1/2}L^2\trkla{\Om}:=\gkla{\mathcal{Q}^{1/2}g\,:\,g\in L^2\trkla{\Om}}\,.
\end{align}
Hence, the $\mathcal{Q}$-Wiener process $\trkla{W_t}_{t\in\tekla{0,T}}$ has the representation
\begin{align}
W_t:=\sum_{k\in\mathds{Z}}\lambda_k\g{k}\beta_k\trkla{t}
\end{align}
with mutually independent Brownian motions $\trkla{\beta_k}_{k\in\mathds{Z}}$.
Furthermore, we assume that the $\mathcal{Q}$-Wiener process is colored in the sense that
\begin{itemize}
\labitem{\textbf{(W2)}}{item:W2} there exists a positive constant $\widetilde{C}$ such that
\begin{align}\label{eq:colorW2}
\sum_{k\in\mathds{Z}}\lambda_k^2\norm{\g{k}}_{W^{2,\infty}\trkla{\Om}}^2\leq\widetilde{C}\,.
\end{align}
\end{itemize}
The operator $\Phi$ mapping the stochastic process $\phi$ into the space of Hilbert--Schmidt operators from $\mathcal{Q}^{1/2}L^2\trkla{\Om}$ to $L^2\trkla{\Om}$ is defined via
\begin{align}\label{eq:defPhi}
\Phi\trkla{\phi}g:=\varrho\trkla{\phi}\sum_{k\in\mathds{Z}}\trkla{g,\g{k}}_{L^2\trkla{\Om}}\g{k}
\end{align}
for all $g\in\mathcal{Q}^{1/2}L^2\trkla{\Om}$.
Concerning the coefficient function $\varrho\,:\,\mathds{R}\rightarrow\mathds{R}$, we follow \cite{MajeeProhl2018} and \cite{\citeASAV} and assume that
\begin{itemize}
\labitem{\textbf{(C)}}{item:rho} $\varrho\in L^\infty\trkla{\mathds{R}}\cap C^{0,1}\trkla{\mathds{R}}$.
\end{itemize}
This assumption in particular guarantees that $\Phi$ as defined in \eqref{eq:defPhi} is indeed a mapping into the space of Hilbert--Schmidt operators from $\mathcal{Q}^{1/2}L^2\trkla{\Om}$ to $L^2\trkla{\Om}$, since
\begin{align}
\sum_{k\in\mathds{Z}}\norm{\Phi\trkla{\phi}\trkla{\lambda_k}\g{k}}_{L^2\trkla{\Om}}^2=\sum_{k\in\mathds{Z}}\norm{\varrho\trkla{\phi}\lambda_k\g{k}}_{L^2\trkla{\Om}}^2\leq \norm{\varrho}_{L^\infty\trkla{\mathds{R}}}^2\sum_{k\in\mathds{Z}}\norm{\mathcal{Q}^{1/2}\g{k}}_{L^2\trkla{\Om}}^2\leq C\,.
\end{align}
Hence, the source terms on the right-hand side of \eqref{eq:model} are governed by
\begin{align}\label{eq:def:Phi}
\Phi\trkla{\phi}\mathrm{d}W:=\sum_{k\in\mathds{Z}}\lambda_k \g{k}\varrho\trkla{\phi}\mathrm{d}\beta_{k}\,.
\end{align}
Another straightforward consequence of Assumption \ref{item:rho} are the estimates
\begin{align}
\norm{\nabla\Ih{\varrho\trkla{\zeta\h}}}_{L^p\trkla{\Om}}&\,\leq \overline{C}\norm{\nabla \zeta\h}_{L^p\trkla{\Om}}\,,& \norm{\nablaG\IhG{\varrho\trkla{\widehat{\zeta\h}}}}_{L^p\trkla{\Gamma}}\leq \overline{C}\norm{\nablaG\widehat{\zeta}\h}_{L^p\trkla{\Gamma}}
\end{align}
for any $\zeta\h\in\Uh$, $\widehat{\zeta}\h\in\UhG$, and $p\in\tekla{1,\infty}$ with a constant $\overline{C}$ depending on the Lipschitz constant of $\varrho$.
\begin{remark}
As discussed in \cite{Scarpa2021b} the multiplicative noise can be chosen in the form
\begin{align}
\Phi\trkla{\phi}\g{k}:=\widetilde{\varrho}_k\trkla{\phi}-\iO\widetilde{\varrho}_k\trkla{\phi}\dx\,,
\end{align}
resulting in $\iO\phi\dx$ being conserved.
Such a choice makes sense from the modeling point of view, but is not necessary for the results presented in this paper.
\end{remark}
In our discrete scheme, we shall approximate the $\mathcal{Q}$-Wiener process $W$ by a sequence of discrete stochastic increments $\tgkla{\sinc{n}}_{n=1,\ldots,N}$ which are supposed to have the following properties:
\begin{itemize}
\labitem{\textbf{(D0)}}{item:filtration} Let $\mathcal{F}^\tau=\trkla{\mathcal{F}_t^\tau}_{t\in\tekla{0,T}}$ be defined by $\mathcal{F}_t^\tau=\mathcal{F}_{t\no}$ for $t\in[t\no,t\nn)$.
\labitem{\textbf{(D1)}}{item:increment} $\sinc{n}$ is $\mathcal{F}^\tau_{t\nn}$-measurable and independent of $\mathcal{F}^\tau_{t^m}$ for all $0\leq m\leq n-1$.
\labitem{\textbf{(D2)}}{item:randomvars} There exist mutually independent symmetric random variables $\xi_k^{n,\tau}$ such that
\begin{align*}
\sinc{n}=\sqrt{\tau}\sum_{k\in\mathds{Z}}\lambda_k\g{k}\xi_k^{n,\tau}\,,
\end{align*}
where $\expected{\xi_k^{n,\tau}}=0$, $\expected{\tabs{\xi_k^{n,\tau}}^2}=1$, and $\expected{\tabs{\xi_k^{n,\tau}}^p}\leq C_p$ for all integer $p\geq2$ with a constant $C_p$ depending on the exponent $p$ but not on $h$, $\tau$, $k$, or $n$.
\end{itemize}
Here $\trkla{\g{k}}_{k\in\mathds{Z}}$ are an orthogonal basis of $L^2\trkla{\Om}$ and $\trkla{\lambda_k}_{k\in\mathds{Z}}$ are given real numbers such that
\begin{itemize}
\labitem{\textbf{(D3)}}{item:color} the noise $\sinc{n}$ is colored in the sense that \eqref{eq:colorW2} in Assumption \ref{item:W2} is satisfied for a positive constant $\widetilde{C}$ that is independent of $h$ and $\tau$.
\end{itemize}
Due to the trace theorem (see e.g~\cite{Kufner77, Necas2012}), assumption \ref{item:color} also implies
\begin{align}\label{eq:tracecolor}
\sum_{k\in\mathds{Z}}\lambda_k^2\norm{\trace{\g{k}}}_{W^{1,\infty}\trkla{\Gamma}}^2\leq C\,.
\end{align}
Summing $\sinc{n}$ from $n=1$ to $m$ provides the discrete approximation $\bs{\xi}^{m,\tau}$ of the $\mathcal{Q}$-Wiener process.
We want to emphasize that the assumptions stated in \ref{item:randomvars} are more general than assuming that $\bs{\xi}^{m,\tau}$ is obtained by evaluating a $\mathcal{Q}$-Wiener process satisfying \ref{item:W1} at given points in time, as \ref{item:randomvars} does not require the independent random variables $\xi_k^{n,\tau}$ to be $\mathcal{N}\trkla{0,1}$-distributed.
Further, we shall approximate $\Phi$ by $\Phi\h$ which is defined via
\begin{align}\label{eq:def:phih}
\Phi\h\trkla{\zeta\h} f:=\sum_{k\in\Zh}\Ih{\varrho\trkla{\zeta\h}\trkla{f,\g{k}}_{L^2\trkla{\Om}}\g{k}}
\end{align}
for all $f\in \mathcal{Q}^{1/2} L^2\trkla{\Om}$ and $\zeta\h\in\Uh$.
Here, $\Zh$ is an $h$-dependent finite subset of $\mathds{Z}$ satisfying $\bigcup_{h>0}\Zh=\mathds{Z}$ and $\mathds{Z}_{h_1}\subset\mathds{Z}_{h_2}$ for $h_1\geq h_2$.
As a consequence, only a finite number of modes will enter the discrete scheme for every given $h$.
Hence, it suffices to approximate the $\mathcal{Q}$-Wiener process by
\begin{align}\label{eq:defxih}
\bs{\xi}\h^{m,\tau}:=\sum_{n=1}^m\sqrt{\tau}\sum_{k\in\Zh}\lambda_k\g{k}\xi_k^{n,\tau}\,.
\end{align}

In order to control higher moments of the (discrete) stochastic integrals, we will need the following estimates:
\begin{lemma}\label{lem:BDG}
Let $\trkla{\sinc{n}}_{n=1,\ldots,N}$ be a sequence of discrete stochastic increments satisfying \ref{item:filtration}-\ref{item:color}.
Furthermore, let $\trkla{\Phi\h\no}_{n=1,\ldots,N}$ be a sequence of $\trkla{\mathcal{F}_{t\no}^\tau}_{n=1,\ldots,N}$-measurable mappings from $\Omega$ to the set of Hilbert--Schmidt operators mapping $\mathcal{Q}^{1/2}L^2\trkla{\Om}$ to a separable Hilbert space $H$, i.e.~$\Phi\h\no\in L_2\trkla{\mathcal{Q}^{1/2}L^2\trkla{\Om};H}$, satisfying $\Phi\h\no \g{k}=0$ for all $k\notin\Zh$ (cf.~\eqref{eq:def:phih}).
Then, for $p\in[1,\infty)$ the estimates
\begin{subequations}
\begin{align}\label{eq:bdg1}
\expected{\norm{\Phi\h\sinc{n}}_H^p}\leq&\, C_p\tau^{p/2}\expected{\rkla{\sum_{k\in\Zh}\norm{\lambda_k\Phi\h\no\g{k}}_H^2}^{p/2}}\,,\\
\label{eq:bdg2}\expected{\max_{1\leq l\leq m }\norm{\sum_{n=1}^l \Phi\h\no\sinc{n}}_H^p}\leq &\,C_p\rkla{m\tau}^{\tfrac{p-2}2}\sum_{n=1}^m\tau\expected{\rkla{\sum_{k\in\Zh}\norm{\lambda_k\Phi\h\no\g{k}}_H^2}^{p/2}}
\end{align}
\end{subequations}
hold true with a constant $C_p>0$ which depends on $p$ but is independent of $h$ and $\tau$.
\end{lemma}
These results are an extension of the results established for $\mathds{R}$-valued stochastic increments in \cite[Lemma 2.8]{Ondrejat2022}.
For a proof of Lemma \ref{lem:BDG}, we refer to \cite[Lemma 2.4]{\citeASAV}.\\
When passing to the limit $\trkla{h,\tau}\rightarrow\trkla{0,0}$ with families of fully discrete random variables, we need time-interpolants of our time-discrete random variables.
For a time-discrete function $a\nn$, $n=0,\ldots,N$ we introduce some time-index-free notation as follows:
\begin{subequations}\label{eq:timeinterpol}
\begin{align}
a\tl\trkla{.,t}&:=\tfrac{t-t\no}\tau a\nn\trkla{.}+\tfrac{t\nn-t}{\tau}a\no\trkla{.}&&t\in\tekla{t\no,t\nn},\,n\geq1\,,\\
a\tm\trkla{.,t}&:=a\no\trkla{.}&&t\in[t\no,t\nn),\,n\geq1\,,\\
a\tp\trkla{.,t}&:=a\nn\trkla{.}&&t\in(t\no,t\nn],\,n\geq1\,.
\end{align}
\end{subequations}
we shall complete this definition by setting $a\tp\trkla{.,0}:=a^0\trkla{.}$.
Obviously, we have
\begin{align}
a\tl\trkla{.,t}-a\tm\trkla{.,t}=\tfrac{t-t\no}\tau\trkla{a\nn-a\no}&&\text{for all~}t\in[t\no,t\nn)\,.
\end{align}
To indicate that a statement is valid for all three time-interpolants defined in \eqref{eq:timeinterpol} we use the abbreviation $a\tpm$. 

\section{The discrete scheme}\label{sec:discretescheme}
For the ease of representation, we define the abbreviations 
\begin{subequations}
\begin{align}
E\h^\Om\trkla{\zeta\h}&:=\iO\Ih{F\trkla{\zeta\h}}\dx\,,\\
E\h^\Gamma\trkla{\widehat{\zeta}\h}&:=\iGamma\IhG{G\trkla{\widehat{\zeta}\h}}\dG
\end{align}
\end{subequations}
for $\zeta\h\in\Uh$ and $\widehat{\zeta}\h\in\UhG$.
To linearize the discrete system, we introduce the stochastic auxiliary variables $r,\,s\,:\,\Omega\times\tekla{0,T}\rightarrow\mathds{R}$ with $r\trkla{\omega,t}:=\sqrt{\iO F\trkla{\phi\trkla{\omega,t,x}}\dx}$ and $s\trkla{\omega,t}:=\sqrt{\iGamma G\trkla{\trace{\phi\trkla{\omega,t,x}}}\dG}$ as the square roots of the non-quadratic parts of the energy.
An application of the chain rule suggests that the evolution of these auxiliary variables is given by
\begin{subequations}
\begin{align}
\para{t}r &=\frac{1}{2\sqrt{\iO F\trkla{\phi}\dx}}\iO F^\prime\trkla{\phi}\para{t}\phi\dx\,,\label{eq:SAVrformal}\\
\para{t}s&=\frac{1}{2\sqrt{\iGamma G\trkla{\trace{\phi}}\dG}}\iGamma G^\prime\rkla{\trace{\phi}}\para{t}\trace{\phi}\dG\,.
\end{align}
\end{subequations}
The original idea of scalar auxiliary variable method, as presented in \cite{ShenXuYang2018}, is to approximate \eqref{eq:SAVrformal} by
\begin{align}\label{eq:SAVr:semidisc}
r\trkla{t\nn}-r\trkla{t\no}=\frac{1}{2\sqrt{\iO F\trkla{\phi\trkla{t\no}}\dx}}\iO F^\prime\trkla{\phi\trkla{t\no}}\trkla{\phi\trkla{t\nn}-\phi\trkla{t\no}}\dx\,,
\end{align}
which leads to a time-discrete scheme that is linear with respect to the unknown quantities, but also still stable w.r.t.~the modified energy expressed in terms of $r\trkla{t\nn}$ and $s\trkla{t\nn}$ instead of $\iO F\trkla{\phi\trkla{t\nn}}\dx$ and $\iGamma G\trkla{\trace{\phi\trkla{t\nn}}}\dG$.
When discussing the convergence behavior of the auxiliary variables, i.e.~the question whether the discrete  approximations of $r\trkla{t\nn}$ converge towards $\sqrt{\iO F\trkla{\phi\trkla{t\nn}}\dx}$ in a suitable sense, it is important to notice that the right-hand side of \eqref{eq:SAVr:semidisc} is given as the first order Taylor approximation of $\sqrt{\iO F\trkla{\phi\trkla{t\nn}}}\dx$ around $\phi\trkla{t\no}$. 
Hence, as outlined in \cite[Remark 3.1]{\citeASAV}, we can anticipate that the additional approximation error per time step stemming from the time-discretization of $r$ will depend on $\trkla{\phi\trkla{t\nn}-\phi\trkla{t\no}}^2$, i.e.~the global approximation error will depend on $\sum_{k=1}^n\trkla{\phi\trkla{t^k}-\phi\trkla{t^{k-1}}}^2$.
Unfortunately, in the stochastic setting, the time regularity of $\phi$ is severely limited by the regularity of the Brownian motions on the right-hand side of \eqref{eq:model}.
As a consequence, this approximation error will not vanish for $\tau\searrow0$.
Hence, we follow the ideas presented in \cite{\citeASAV} and add suitable linear approximations of the second order terms of the Taylor approximation that guarantee convergence of the scalar auxiliary variables.
For the rigorous computations, we refer to Lemma \ref{lem:errorsav} below.\\
In our discrete scheme, the scalar auxiliary variables are approximated using a sequence of random variables $\trkla{r\h\nn}_{n=0,\ldots,N}$ and $\trkla{s\h\nn}_{n=0,\ldots,N}$ which are supposed to approximate $\trkla{\sqrt{E\h^\Om\trkla{\phi\h\nn}}}_{n=0,\ldots,N}$ and $\trkla{\sqrt{E\h^\Gamma\trkla{\trace{\phi\h\nn}}}}_{n=0,\ldots,N}$.
Defining the discrete initial data via $\phi\h^0:=\Ih{\phi_0}$ (cf.~Assumption \ref{item:initial}), $r\h^0:=\sqrt{E\h^\Om\trkla{\phi\h^0}}$, and $s\h^0:=\sqrt{E\h^\Gamma\trkla{\phi\h^0}}$, we state the following discrete scheme:\\

For a given $\Uh\times\mathds{R}\times\mathds{R}$-valued random variable $\trkla{\phi\h\no,r\h\no,s\h\no}$, find an $\Uh\times\Uh\times\UhG\times\mathds{R}\times\mathds{R}$-valued random variable $\trkla{\phi\h\nn,\mu\h\nn,\theta\h\nn,r\h\nn,s\h\nn}$ such that pathwise
\begin{subequations}\label{eq:modeldisc}
\begin{align}\label{eq:modeldisc:phiO}
\iO\Ih{\trkla{\phi\h\nn-\phi\h\no}\psi\h}\dx+\tau\iO\nabla\mu\h\nn\cdot\nabla\psi\h\dx&=\iO\Ih{\Phi\h\trkla{\phi\h\no}\sinc{n}\psi\h}\dx\,,\\
\iGamma\IhG{\trace{\phi\h\nn-\phi\h\no}\widehat{\psi}\h}\dG+\tau\iGamma\IhG{\theta\h\nn\widehat{\psi}\h}\dG&=\iGamma\IhG{\trace{\Phi\h\trkla{\phi\h\no}\sinc{n}}\widehat{\psi}\h}\dG\,,\label{eq:modeldisc:phiG}
\end{align}
\begin{align}\label{eq:modeldisc:pot}
\begin{split}
\iO&\Ih{\mu\h\nn\eta\h}\dx+\iGamma\IhG{\theta\h\nn\trace{\eta\h}}\dG= \iO\nabla\phi\h\nn\cdot\nabla\eta\h\dx + \iGamma\nablaG\trace{\phi\h\nn}\cdot\nablaG\trace{\eta\h}\dG\\
&+\left[\frac{r\h\nn}{\sqrt{E\h^\Om\trkla{\phi\h\no}}}\iO\Ih{F^\prime\trkla{\phi\h\no}\eta\h}\dx\right.\\
&-\frac{r\h\nn}{4\tekla{E\h^\Om\trkla{\phi\h\no}}^{3/2}}\iO\Ih{F^\prime\trkla{\phi\h\no}\Phi\h\trkla{\phi\h\no}\sinc{n}}\dx\iO\Ih{F^\prime\trkla{\phi\h\no}\eta\h}\dx\\
&+\left.\frac{r\h\nn}{2\sqrt{E\h^\Om\trkla{\phi\h\no}}}\iO\Ih{F^{\prime\prime}\trkla{\phi\h\no}\Phi\h\trkla{\phi\h\no}\sinc{n}\eta\h}\dx\right]\\
&+\left[\frac{s\h\nn}{\sqrt{E\h^\Gamma\rkla{\trace{\phi\h\no}}}}\iGamma\IhG{G^\prime\rkla{\trace{\phi\h\no}}\trace{\eta\h}}\dG\right.\\
&-\frac{s\h\nn}{4\ekla{E\h^\Gamma\rkla{\trace{\phi\h\no}}}^{3/2}}\iGamma\IhG{G^\prime\rkla{\trace{\phi\h\no}}\trace{\Phi\h\trkla{\phi\h\no}\sinc{n}}}\dG\\
&\qquad\qquad\times\iGamma\IhG{G^\prime\rkla{\trace{\phi\h\no}}\trace{\eta\h}}\dG\\
&+\left.\frac{s\h\nn}{2\sqrt{E\h^\Gamma\rkla{\trace{\phi\h\no}}}}\iGamma\IhG{G^{\prime\prime}\rkla{\trace{\phi\h\no}}\trace{\Phi\h\trkla{\phi\h\no}\sinc{n}}\trace{\eta\h}}\dG\right]\,,
\end{split}
\end{align}
for all $\psi\h,\eta\h\in\Uh$ and $\widehat{\psi}\h\in\UhG$ together with
\begin{align}\label{eq:modeldisc:r}
\begin{split}
r\h\nn&-r\h\no=\frac{1}{2\sqrt{E\h^\Om\trkla{\phi\h\no}}}\iO\Ih{F^\prime\trkla{\phi\h\no}\trkla{\phi\h\nn-\phi\h\no}}\dx\\
&-\frac{1}{8\tekla{E\h^\Om\trkla{\phi\h\no}}^{3/2}}\iO\Ih{F^\prime\trkla{\phi\h\no}\Phi\h\trkla{\phi\h\no}\sinc{n}}\dx\iO\Ih{F^\prime\trkla{\phi\h\no}\trkla{\phi\h\nn-\phi\h\no}}\dx\\
&+\frac{1}{4\sqrt{E\h^\Om\trkla{\phi\h\no}}}\iO\Ih{F^{\prime\prime}\trkla{\phi\h\no}\Phi\h\trkla{\phi\h\no}\sinc{n}\trkla{\phi\h\nn-\phi\h\no}}\dx
\end{split}
\end{align}
and 
\begin{align}\label{eq:modeldisc:s}
\begin{split}
s\h\nn&-s\h\no=\frac{1}{2\sqrt{E\h^\Gamma\rkla{\trace{\phi\h\no}}}}\iGamma\IhG{G^\prime\rkla{\trace{\phi\h\no}}\trace{\phi\h\nn-\phi\h\no}}\dG\\
&-\frac{1}{8\ekla{E\h^\Gamma\rkla{\trace{\phi\h\no}}}^{3/2}}\iGamma\IhG{G^\prime\rkla{\trace{\phi\h\no}}\trace{\Phi\h\trkla{\phi\h\no}\sinc{n}}}\dG\\
&\qquad\qquad\times\iGamma\IhG{G^\prime\rkla{\trace{\phi\h\no}}\trace{\phi\h\nn-\phi\h\no}}\dG\\
&+\frac{1}{4\sqrt{E\h^\Gamma\rkla{\trace{\phi\h\no}}}}\iGamma\IhG{G^{\prime\prime}\rkla{\trace{\phi\h\no}}\trace{\Phi\h\trkla{\phi\h\no}\sinc{n}}\trace{\phi\h\nn-\phi\h\no}}\dG\,.
\end{split}
\end{align}
\end{subequations}
Here, the terms in the square brackets in \eqref{eq:modeldisc:pot} are the discrete approximations of $\iO F^\prime\trkla{\phi\trkla{t\nn}}\eta\dx$ and $\iGamma G^\prime\trkla{\trace{\phi\trkla{t\nn}}}\trace{\eta}\dG$.
To simplify the notation, we shall suppress the trace operator $\trace{\cdot}$ whenever the notation is unambiguous.
For future reference, we shall denote the additional terms added to \eqref{eq:modeldisc:pot} by the more accurate Taylor expansion by $\Uh$- and $\UhG$-valued random variables $\Xi_{h,\Om}\nn$ and $\Xi_{h,\Gamma}\nn$, i.e.
\begin{subequations}\label{eq:deferrorterms}
\begin{align}
\begin{split}\label{eq:deferrorterms:Om}
\Xi_{h,\Om}\nn:=&\,-\frac{r\h\nn}{4\ekla{E\h^\Om\trkla{\phi\h\no}}^{3/2}}\iO\Ih{F^\prime\trkla{\phi\h\no}\Phi\h\trkla{\phi\h\no}\sinc{n}}\dx\,\Ih{F^\prime\trkla{\phi\h\no}}\\
&\,+\frac{r\h\nn}{2\sqrt{E\h^\Om\trkla{\phi\h\no}}}\Ih{F^{\prime\prime}\trkla{\phi\h\no}\Phi\h\trkla{\phi\h\no}\sinc{n}}\,,
\end{split}\\
\begin{split}\label{eq:deferrorterms:Gamma}
\Xi_{h,\Gamma}\nn:=&\,-\frac{s\h\nn}{4\ekla{E\h^\Gamma\trkla{\phi\h\no}}^{3/2}}\iGamma\IhG{G^\prime\trkla{\phi\h\no}\trace{\Phi\h\trkla{\phi\h\no}\sinc{n}}}\dG\,\IhG{G^\prime\trkla{\phi\h\no}}\\
&\,+\frac{s\h\nn}{2\sqrt{E\h^\Gamma\trkla{\phi\h\no}}}\IhG{G^{\prime\prime}\trkla{\phi\h\no}\trace{\Phi\h\trkla{\phi\h\no}\sinc{n}}}\,.
\end{split}
\end{align}
\end{subequations}
As we shall show in Lemma \ref{lem:potential} that these additional terms vanish for $\tau\searrow 0$.
With these definitions, \eqref{eq:modeldisc:pot} can be written as
\begin{align}
\begin{split}
\iO\Ih{\mu\h\nn\eta\h}\dx+\iGamma\IhG{\theta\h\nn\trace{\eta\h}}\dG=\iO\nabla\phi\h\nn\cdot\nabla\eta\h\dx +\iGamma\nablaG\phi\h\nn\cdot\nablaG\eta\h\dG\\
+\frac{r\h\nn}{\sqrt{E\h^\Omega\trkla{\phi\h\no}}}\iO\Ih{F^\prime\trkla{\phi\h\no}\eta\h}\dx+\iO\Ih{\Xi_{h,\Om}\nn\eta\h}\dx\\
+\frac{s\h\nn}{\sqrt{E\h^\Gamma\trkla{\phi\h\no}}}\iGamma\IhG{G^\prime\trkla{\phi\h\no}\eta\h}\dG+\iGamma\IhG{\Xi_{h,\Gamma}\nn\eta\h}\dG\,.
\end{split}
\end{align}
\section{Main results}\label{sec:mainresults}
In this section, we shall state the main results which will be proven in the subsequent sections.
The first result provides the existence of pathwise unique solutions to the discrete scheme \eqref{eq:modeldisc}:
\begin{lemma}\label{lem:existence}
Let the assumptions \ref{item:time}, \ref{item:space1}, \ref{item:space2}, and \ref{item:filtration} hold true.
Then, there exists a sequence $\trkla{\phi\h\nn,\mu\h\nn,\theta\h\nn,r\h\nn,s\h\nn}_{n\geq1}$ of $\Uh\times\Uh\times\UhG\times\mathds{R}\times\mathds{R}$-valued random variables that solves \eqref{eq:modeldisc} for each $\omega\in\Omega$. Furthermore, the map $\trkla{\phi\h\nn,\mu\h\nn,\theta\h\nn,r\h\nn,s\h\nn}\,:\,\Omega\rightarrow\Uh\times\Uh\times\UhG\times\mathds{R}\times\mathds{R}$ is $\mathcal{F}_{t\nn}$-measurable.
\end{lemma}
The proof of this result can be found in Section \ref{sec:existence}.
Starting from these fully discrete solutions, we can pass to the limit $\trkla{h,\tau}\rightarrow\trkla{0,0}$ to obtain the existence of pathwise unique martingale solutions to \eqref{eq:model}:
\begin{theorem}\label{thm:mainresult}
Let the assumptions \ref{item:time}, \ref{item:space1}, \ref{item:space2}, \ref{item:potentials}, \ref{item:initial}, \ref{item:rho} and \ref{item:filtration}-\ref{item:color}, hold true.
Then, there exist a filtered probability space $\trkla{\widetilde{\Omega},\widetilde{\mathcal{A}},\widetilde{\mathcal{F}},\widetilde{\Prob}}$ and a sequence of random variables $\trkla{\widetilde{\phi}\h\tl,\widetilde{\phi_{\Gamma}}\h\tl, \widetilde{\mu}\h\tp,\widetilde{\theta}\h\tp }_{h,\tau}$ on $\widetilde{\Omega}$ whose laws coincide with the laws of the time-interpolants $\trkla{{\phi}\h\tl,\trace{\phi\h\tl}, {\mu}\h\tp,{\theta}\h\tp }_{h,\tau}$ of the discrete solutions to \eqref{eq:modeldisc} established in Lemma \ref{lem:existence}.
Furthermore, there exists an $\widetilde{\mathcal{F}}$-measurable $\mathcal{Q}$-Wiener process $\widetilde{W}:=\sum_{k\in\mathds{Z}}\lambda_k\g{k}\widetilde{\beta}_k$ on $\widetilde{\Omega}$ and progressively $\mathcal{F}$-measurable processes
\begin{align*}
\widetilde{\phi}&\in L^{2\p}_{\operatorname{weak-}(*)}\trkla{\widetilde{\Omega};L^\infty\trkla{0,T;H^1\trkla{\Om}}}\cap L^{4\p}\trkla{\widetilde{\Omega};C^{0,\trkla{\p-1}/\trkla{4\p}}\trkla{\tekla{0,T};L^2\trkla{\Om}}}\,,\\
\widetilde{\phi_\Gamma}&\in L^{2\p}_{\operatorname{weak-}(*)}\trkla{\widetilde{\Omega};L^\infty\trkla{0,T;H^1\trkla{\Gamma}}}\cap L^{2\p}\trkla{\widetilde{\Omega};C^{0,\trkla{\p-1}/\trkla{2\p}}\trkla{\tekla{0,T};L^2\trkla{\Gamma}}}\,,\\
\widetilde{\mu}&\in L^{2\p}\trkla{\widetilde{\Omega};L^2\trkla{0,T;H^1\trkla{\Om}}}\,,\\
\widetilde{\theta}&\in L^{2\p}\trkla{\widetilde{\Omega};L^2\trkla{0,T;L^2\trkla{\Gamma}}}
\end{align*}
for $\p\in\trkla{1,\infty}$ such that $\widetilde{\Prob}$-almost surely
\begin{align*}
\lim_{\trkla{h,\tau}\rightarrow\trkla{0,0}}\widetilde{\phi}\h\tl&=\widetilde{\phi}&&\text{in~}C\trkla{\tekla{0,T};L^s\trkla{\Om}}\,,\\
\lim_{\trkla{h,\tau}\rightarrow\trkla{0,0}}\widetilde{\phi_{\Gamma}}\h\tl&=\widetilde{\phi_\Gamma}&&\text{in~}C\trkla{\tekla{0,T};L^r\trkla{\Gamma}}\,,\\
\lim_{\trkla{h,\tau}\rightarrow\trkla{0,0}}\widetilde{\mu}\h\tp&=\widetilde{\mu}&&\text{in~}L^2\trkla{0,T;H^1\trkla{\Om}}_{\weaktop}\,,\\
\lim_{\trkla{h,\tau}\rightarrow\trkla{0,0}}\widetilde{\theta}\h\tp&=\widetilde{\theta}&&\text{in~}L^2\trkla{0,T;L^2\trkla{\Gamma}}_{\weaktop}
\end{align*}
with $s\in[1,\tfrac{2d}{d-2})$ and $r\in[1,\infty)$.
These processes are pathwise unique and satisfy
\begin{subequations}\label{eq:weakform}
\begin{align}
\iO\trkla{\widetilde{\phi}\trkla{t}-\phi\trkla{0}}\psi\dx+\int_0^t\iO\nabla \widetilde{\mu}\cdot\nabla\psi\dx\ds=\sum_{k\in\mathds{Z}}\int_0^t\iO\varrho\trkla{\widetilde{\phi}}\lambda_k\g{k}\psi\dx\,\mathrm{d}\widetilde{\beta}_k\,,\\
\iGamma\rkla{\widetilde{\phi_\Gamma}\trkla{t}-\trace{\phi}\trkla{0}}\widehat{\psi}\dG+\int_0^t\iGamma\widetilde{\theta}\widehat{\psi}\dG\ds=\sum_{k\in\mathds{Z}}\int_0^t\iGamma\trace{\varrho\trkla{\widetilde{\phi}}\lambda_k\g{k}}\widehat{\psi}\dG\,\mathrm{d}\widetilde{\beta}_k\,,
\end{align}
$\widetilde{\Prob}$-a.s.~for all $t\in\tekla{0,T}$ and test functions $\psi\in H^1\trkla{\Om}$ and $\widehat{\psi}\in L^2\trkla{\Gamma}$.
Furthermore,
\begin{multline}
\iO\widetilde{\mu}\eta\dx+\iGamma\widetilde{\theta}\trace{\eta}\dG=\iO\nabla\widetilde{\phi}\cdot\nabla\eta\dx+\iGamma\nablaG\widetilde{\phi_\Gamma}\cdot\nablaG\trace{\eta}\dG\\
+\iO F^\prime\trkla{\widetilde{\phi}}\eta\dx+\iGamma G^\prime\trkla{\widetilde{\phi_\Gamma}}\trace{\eta}\dG
\end{multline}
\end{subequations}
holds true $\widetilde{\Prob}$-a.s.~for almost all $t\in\tekla{0,T}$ and test functions $\eta\in\mathcal{V}$.
In addition, $\trace{\widetilde{\phi}}=\widetilde{\phi_\Gamma}$ $\widetilde{\Prob}$-almost surely almost everywhere on $\trkla{0,T}\times\Gamma$. 
\end{theorem}
The above results use an arbitrary finite dimensional random walk $\bs{\xi}\h^{m,\tau}$ (cf.~\eqref{eq:defxih}) satisfying \ref{item:filtration}-\ref{item:color} as starting point in the finite element scheme and provide the existence of pathwise unique martingale solutions.
Hence, the Yamada--Watanabe theorem provides the existence of strong solutions (cf.~\cite{RoecknerSchmulandZhang2008} or Theorem E.0.8 in \cite{LiuRoeckner}).
These strong solutions can also be obtained as the limit of a sequence of fully discrete solutions:
If we approximate a $\mathcal{Q}$-Wiener process $W$ satisfying \ref{item:W1} and \ref{item:W2} via
\begin{align}\label{eq:finiteQWiener}
\bs{\xi}\h^{m,\tau}=\sum_{k\in\mathds{Z}_h}\rkla{W\trkla{t^m},\g{k}}_{L^2\trkla{\Om}}\g{k}\,,
\end{align}
$\bs{\xi}\h^{m,\tau}$ still satisfies \ref{item:filtration}-\ref{item:color} with Gaussian random variables $\xi_k^{n,\tau}$.
For this choice, we can establish convergence towards strong solutions:

\begin{theorem}\label{thm:strongsolutions}
Let $W$ be a Wiener process defined on $\trkla{\Omega,\mathcal{A},\mathcal{F},\Prob}$ satisfying \ref{item:W1} and \ref{item:W2} with finite dimensional approximations given by \eqref{eq:finiteQWiener}.
Furthermore, let the assumptions \ref{item:time}, \ref{item:space1}, \ref{item:space2}, \ref{item:potentials}, \ref{item:initial}, \ref{item:rho}, and \ref{item:filtration}--\ref{item:randomvars} hold true. Then, there exist pathwise unique, progressively $\mathcal{F}$-measurable processes
\begin{align*}
\phi&\in L^{2\p}_{\operatorname{weak-}(*)}\trkla{{\Omega};L^\infty\trkla{0,T;H^1\trkla{\Om}}}\cap L^{4\p}\trkla{\Omega;C^{0,\trkla{\p-1}/\trkla{4\p}}\trkla{\tekla{0,T};L^2\trkla{\Om}}}\,,\\
\phi_\Gamma&\in L^{2\p}_{\operatorname{weak-}(*)}\trkla{{\Omega};L^\infty\trkla{0,T;H^1\trkla{\Gamma}}}\cap L^{2\p}\trkla{{\Omega};C^{0,\trkla{\p-1}/\trkla{2\p}}\trkla{\tekla{0,T};L^2\trkla{\Gamma}}}\,,\\
\mu&\in L^{2\p}\trkla{{\Omega};L^2\trkla{0,T;H^1\trkla{\Om}}}\,,\\
\theta&\in L^{2\p}\trkla{{\Omega};L^2\trkla{0,T;L^2\trkla{\Gamma}}}
\end{align*}
for $\p\in\trkla{1,\infty}$, which are the limits of the time interpolants $\trkla{\phi\h\tl,\trace{\phi\h\tl},\mu\h\tp,\theta\h\tp}_{h,\tau}$ of the discrete solutions to \eqref{eq:modeldisc}.
In particular, we have $\Prob$-almost surely
\begin{align*}
\lim_{\trkla{h,\tau}\rightarrow\trkla{0,0}}{\phi}\h\tl&={\phi}&&\text{in~}C\trkla{\tekla{0,T};L^s\trkla{\Om}}\,,\\
\lim_{\trkla{h,\tau}\rightarrow\trkla{0,0}}\trace{\phi\h\tl}&={\phi_\Gamma}&&\text{in~}C\trkla{\tekla{0,T};L^r\trkla{\Gamma}}\,,\\
\lim_{\trkla{h,\tau}\rightarrow\trkla{0,0}}\mu\h\tp&=\mu&&\text{in~}L^2\trkla{0,T;H^1\trkla{\Om}}_{\weaktop}\,,\\
\lim_{\trkla{h,\tau}\rightarrow\trkla{0,0}}\theta\h\tp&=\theta&&\text{in~}L^2\trkla{0,T;L^2\trkla{\Gamma}}_{\weaktop}
\end{align*}
for $s\in[1,\tfrac{2d}{d-2})$ and $r\in[1,\infty)$.
These processes satisfy
\begin{subequations}
\begin{align*}
\iO\trkla{{\phi}\trkla{t}-\phi\trkla{0}}\psi\dx+\int_0^t\iO\nabla {\mu}\cdot\nabla\psi\dx\ds=\sum_{k\in\mathds{Z}}\int_0^t\iO\varrho\trkla{{\phi}}\lambda_k\g{k}\psi\dx\,\mathrm{d}{\beta}_k\,,\\
\iGamma\rkla{{\phi_\Gamma}\trkla{t}-\trace{\phi}\trkla{0}}\widehat{\psi}\dG+\int_0^t\iGamma{\theta}\widehat{\psi}\dG\ds=\sum_{k\in\mathds{Z}}\int_0^t\iGamma\trace{\varrho\trkla{{\phi}}\lambda_k\g{k}}\widehat{\psi}\dG\,\mathrm{d}{\beta}_k\,,
\end{align*}
${\Prob}$-a.s.~for all $t\in\tekla{0,T}$ and test functions $\psi\in H^1\trkla{\Om}$ and $\widehat{\psi}\in L^2\trkla{\Gamma}$.
Furthermore,
\begin{multline*}
\iO{\mu}\eta\dx+\iGamma{\theta}\trace{\eta}\dG=\iO\nabla{\phi}\cdot\nabla\eta\dx+\iGamma\nablaG{\phi_\Gamma}\cdot\nablaG\trace{\eta}\dG\\
+\iO F^\prime\trkla{{\phi}}\eta\dx+\iGamma G^\prime\trkla{{\phi_\Gamma}}\trace{\eta}\dG
\end{multline*}
\end{subequations}
${\Prob}$-a.s.~for almost all $t\in\tekla{0,T}$ and test functions $\eta\in\mathcal{V}$.
Further, $\trace{{\phi}}={\phi_\Gamma}$ ${\Prob}$-almost surely almost everywhere on $\trkla{0,T}\times\Gamma$. 
\end{theorem}
The remainder of the paper is structured as follows:
In Section \ref{sec:existence}, we present the proof of Lemma \ref{lem:existence}.
The proof of Theorem \ref{thm:mainresult} can be found in the Sections \ref{sec:regularity}--\ref{sec:uniqueness}:
In Section \ref{sec:regularity}, we establish regularity results for discrete solutions to \eqref{eq:modeldisc} which are independent of the discretization parameters $h$ and $\tau$.
These results will be used in Section \ref{sec:compactness} together with Jakubowski's theorem (cf.~\cite{Jakubowski1998}) to establish the existence of converging subsequences.
In Section \ref{sec:limit}, we discuss the passage to the limit $\trkla{h,\tau}\searrow0$ which provides the existence statement in Theorem \ref{thm:mainresult}.
The pathwise uniqueness of these martingale solutions is then established in Section \ref{sec:uniqueness}.
Section \ref{sec:strongsolutions} is devoted to the proof of Theorem \ref{thm:strongsolutions}.
\section{Existence of discrete solutions}\label{sec:existence}
In this section we will prove Lemma \ref{lem:existence} by establishing the existence of pathwise unique solutions to \eqref{eq:modeldisc}.
We adapt the ideas of \cite{\citeASAV} and start by showing that for each fixed $\omega\in\Omega$, \eqref{eq:modeldisc} has a unique solution:
Since $\sinc{n}$ is given for fixed $\omega$, \eqref{eq:modeldisc} is linear with respect to the unknown quantities $\phi\h\nn,\mu\h\nn,\theta\h\nn,r\h\nn$, and $s\h\nn$.
As for finite dimensional linear problems, uniqueness of possible solutions guarantees the existence of solutions for arbitrary right-hand sides, we assume that for given $\phi\h\no$, $r\h\no$, $s\h\no$, and $\sinc{n}$ there exist two solutions.
We denote their difference by $\hat{\phi}$, $\hat{\mu}$, $\hat{\theta}$, $\hat{r}$, and $\hat{s}$. 
Obviously, these differences satisfy
\begin{subequations}
\begin{align}\label{eq:existencephi}
\iO\Ih{\hat{\phi}\psi\h}\dx+\tau\iO\nabla\hat{\mu}\cdot\nabla\psi\h\dx&=0\,,\\
\iGamma\IhG{\hat{\phi}\widehat{\psi}\h}\dG+\tau\iGamma\IhG{\hat{\theta}\widehat{\psi}\h}\dG&=0\,,\label{eq:existencephi2}
\end{align}
\begin{align}\label{eq:existencepot}
\begin{split}
\iO&\Ih{\hat{\mu}\eta\h}\dx+\iGamma\IhG{\hat{\theta}\eta\h}\dG=\iO\nabla\hat{\phi}\cdot\nabla\eta\h\dx+\iGamma\nablaG \hat{\phi}\cdot\nablaG \eta\h\dG\\
&+\frac{\hat{r}}{\sqrt{E\h^\Om\trkla{\phi\h\no}}}\iO\Ih{F^\prime\trkla{\phi\h\no}\eta\h}\dx\\
&-\frac{\hat{r}}{4\tekla{E\h^\Om\trkla{\phi\h\no}}^{3/2}}\iO\Ih{F^\prime\trkla{\phi\h\no}\Phi\h\trkla{\phi\h\no}\sinc{n}}\dx\iO\Ih{F^\prime\trkla{\phi\h\no}\eta\h}\dx\\
&+\frac{\hat{r}}{2\sqrt{E\h^\Om\trkla{\phi\h\no}}}\iO\Ih{F^{\prime\prime}\trkla{\phi\h\no}\Phi\h\trkla{\phi\h\no}\sinc{n}\eta\h}\dx +\frac{\hat{s}}{\sqrt{E\h^\Gamma\trkla{\phi\h\no}}}\iGamma\IhG{G^\prime\trkla{\phi\h\no}\eta\h}\dG\\
&-\frac{\hat{s}}{4\ekla{E\h^\Gamma\trkla{\phi\h\no}}^{3/2}}\iGamma\IhG{G^\prime\trkla{\phi\h\no}\trace{\Phi\h\trkla{\phi\h\no}\sinc{n}}}\dG\iGamma\IhG{G^\prime\trkla{\phi\h\no}\eta\h}\dG\\
&+\frac{\hat{s}}{2\sqrt{E\h^\Gamma\trkla{\phi\h\no}}}\iGamma\IhG{G^{\prime\prime}\trkla{\phi\h\no}\trace{\Phi\h\trkla{\phi\h\no}\sinc{n}}\eta\h}\dG\,,
\end{split}
\end{align}
\begin{align}
\begin{split}\label{eq:existencer}
\hat{r}=&\frac{1}{2\sqrt{E\h^\Om\trkla{\phi\h\no}}}\iO\Ih{F^\prime\trkla{\phi\h\no}\hat{\phi}}\dx\\
&-\frac{1}{8\ekla{E\h^\Om\trkla{\phi\h\no}}^{3/2}}\iO\Ih{F^\prime\trkla{\phi\h\no}\Phi\h\trkla{\phi\h\no}\sinc{n}}\dx\iO\Ih{F^\prime\trkla{\phi\h\no}\hat{\phi}}\dx\\
&+\frac{1}{4\sqrt{E\h^\Om\trkla{\phi\h\no}}}\iO\Ih{F^{\prime\prime}\trkla{\phi\h\no}\hat{\phi}\Phi\trkla{\phi\h\no}\sinc{n}}\dx\,,
\end{split}\\
\begin{split}\label{eq:existences}
\hat{s}=&\frac{1}{2\sqrt{E\h^\Gamma\trkla{\phi\h\no}}}\iGamma\IhG{G^\prime\trkla{\phi\h\no}\hat{\phi}}\dG\\
&-\frac{1}{8\ekla{E\h^\Gamma\trkla{\phi\h\no}}^{3/2}}\iGamma\IhG{G^\prime\trkla{\phi\h\no}\trace{\Phi\h\trkla{\phi\h\no}\sinc{n}}}\dG\iGamma\IhG{G^\prime\trkla{\phi\h\no}\hat{\phi}}\dG\\
&+\frac{1}{4\sqrt{E\h^\Gamma\trkla{\phi\h\no}}}\iGamma\IhG{G^{\prime\prime}\trkla{\phi\h\no}\hat{\phi}\trace{\Phi\h\trkla{\phi\h\no}\sinc{n}}}\dG\,.
\end{split}
\end{align}
\end{subequations}
Choosing $\psi\h\equiv\hat{\mu}$ in \eqref{eq:existencephi}, $\widehat{\psi}\h\equiv\tau^{-1}\hat{\phi}$ in \eqref{eq:existencephi2}, and $\eta\h\equiv\hat{\phi}$ in \eqref{eq:existencepot}, we obtain after substituting \eqref{eq:existencer} and \eqref{eq:existences}
\begin{align}\label{eq:differgy}
\tau\iOmega\abs{\nabla\hat{\mu}}^2dx+\tau^{-1}\iGamma\IhG{\abs{\hat{\phi}}^2}\dG +\iO\abs{\nabla\hat{\phi}}^2\dx+\iGamma\abs{\nablaG\hat{\phi}}^2\dG+2\tabs{\hat{r}}^2+2\tabs{\hat{s}}^2=0\,.
\end{align}
This immediately implies $\hat{r}=\hat{s}=0$ and $\hat{\phi}\equiv0$ which entails $\hat{\theta}\equiv0$ by \eqref{eq:existencephi2}.
Finally, using $\eta\h\equiv 0$ in \eqref{eq:existencepot}, we obtain $\iO\hat{\mu}\dx=1$.
Due to \eqref{eq:differgy}, we have $\hat{\mu}\equiv0$.
This provides the uniqueness and therefore the existence of solutions. 
As shown in, e.g., Theorem 6.7 in \cite{Grillmeier2020} the uniqueness of the solution obtained for each $\omega\in\Omega$ also entails its measurability.

\section{Regularity results}\label{sec:regularity}
In this section, we shall establish uniform regularity results for the discrete solutions obtained in Lemma \ref{lem:existence}.

\begin{lemma}\label{lem:energy}
Let the assumptions \ref{item:time}, \ref{item:space1}, \ref{item:space2}, \ref{item:potentials}, \ref{item:initial}, \ref{item:rho}, and \ref{item:filtration}-\ref{item:color} hold true.
Then, for every $1\leq\p<\infty$, there exists a constant $C\equiv C\trkla{\p,T}>0$ independent of $h$ and $\tau$ such that
\begin{align}\label{eq:energyestimate}
\begin{split}
\expected{\max_{0\leq m\leq N}\norm{\phi\h^m}_{H^1\trkla{\Om}}^{2\p}}+\expected{\max_{0\leq m\leq N}\norm{\phi\h^m}_{H^1\trkla{\Gamma}}^{2\p}} +\expected{\max_{0\leq m\leq N}\tabs{r\h^m}^{2\p}} +\expected{\max_{0\leq m\leq N}\tabs{s\h^m}^{2\p}}\\
+\expected{\rkla{\sum_{n=1}^N\norm{\phi\h\nn-\phi\h\no}_{H^1\trkla{\Om}}^2}^\p}+\expected{\rkla{\sum_{n=1}^N\norm{\phi\h\nn-\phi\h\no}_{H^1\trkla{\Gamma}}^2}^\p} +\expected{\rkla{\sum_{n=1}^m\tabs{r\h\nn-r\h\no}^2}^\p}\\
+\expected{\rkla{\sum_{n=1}^m\tabs{s\h\nn-s\h\no}^2}^\p} + \expected{\rkla{\sum_{n=1}^N\tau\norm{\nabla\mu\h\nn}_{L^2\trkla{\Om}}^2}^\p} +\expected{\rkla{\sum_{n=1}^N\tau\norm{\theta\h\nn}_{L^2\trkla{\Gamma}}^2}^\p}\leq C\,.
\end{split}
\end{align}
\end{lemma}
\begin{proof}
In a first step, we will establish
\begin{align}\label{eq:energystep1}
\max_{0\leq n\leq N}\expected{\norm{\phi\h^m}_{H^1\trkla{\Om}}^{2\p}}+\max_{0\leq n\leq N}\expected{\norm{\phi\h^m}_{H^1\trkla{\Gamma}}^{2\p}}+\max_{0\leq n\leq N}\expected{\tabs{r\h^m}^{2\p}}+\max_{0\leq n\leq N}\expected{\tabs{s\h^m}^{2\p}}\leq C
\end{align}
before proving \eqref{eq:energyestimate}.
For fixed $\omega\in\Omega$, we test \eqref{eq:modeldisc:phiO} by $\psi\h\equiv\mu\h\nn$, \eqref{eq:modeldisc:phiG} by $\widehat{\psi}\h\equiv\theta\h\nn$, and \eqref{eq:modeldisc:pot} by $\eta\h\equiv\phi\h\nn-\phi\h\no$ and $\eta\h\equiv -\Phi\h\trkla{\phi\h\no}\sinc{n}$.
This provides
\begin{align}
\begin{split}
0=& \iO\nabla\phi\h\nn\cdot\nabla\rkla{\phi\h\nn-\phi\h\no}\dx+ \iGamma\nablaG{\phi\h\nn}\cdot\nablaG\trkla{\phi\h\nn-\phi\h\no}\dG\\
&+\frac{r\h\nn}{\sqrt{E\h^\Om\trkla{\phi\h\no}}}\iO\Ih{ F^\prime\trkla{\phi\h\no}\trkla{\phi\h\nn-\phi\h\no}}\dx\\
&-\frac{r\h\nn}{4\tekla{E\h^\Om\trkla{\phi\h\no}}^{3/2}}\iO\Ih{F^\prime\trkla{\phi\h\no}\Phi\h\trkla{\phi\h\no}\sinc{n}}\dx\iO\Ih{F^\prime\trkla{\phi\h\no}\trkla{\phi\h\nn-\phi\h\no}}\dx\\
&+\frac{r\h\nn}{2\sqrt{E\h^\Om\trkla{\phi\h\no}}}\iO\Ih{F^{\prime\prime}\trkla{\phi\h\no}\Phi\h\trkla{\phi\h\no}\sinc{n}\trkla{\phi\h\nn-\phi\h\no}}\dx\\
&+\frac{s\h\nn}{\sqrt{E\h^\Gamma\trkla{\phi\h\no}}}\iGamma\IhG{G^\prime\trkla{\phi\h\no}\trkla{\phi\h\nn-\phi\h\no}}\dG\\
&-\frac{s\h\nn}{4\tekla{E\h^\Gamma\trkla{\phi\h\no}}^{3/2}}\iGamma\IhG{G^\prime\trkla{\phi\h\no}\trace{\Phi\h\trkla{\phi\h\no}\sinc{n}}}\dG\\
&\qquad\qquad\qquad\times\iGamma\IhG{G^\prime\trkla{\phi\h\no}\trkla{\phi\h\nn-\phi\h\no}}\dG
\end{split}\nonumber\\+\begin{split}
&+\frac{s\h\nn}{2\sqrt{E\h^\Gamma\trkla{\phi\h\no}}}\iGamma\IhG{G^{\prime\prime}\trkla{\phi\h\no}\trace{\Phi\h\trkla{\phi\h\no}\sinc{n}}\trkla{\phi\h\nn-\phi\h\no}}\dG\\
&-\iO\nabla\phi\h\nn\cdot\nabla\rkla{\Phi\h\trkla{\phi\h\no}\sinc{n}}\dx-\iGamma\nablaG\phi\h\nn\cdot\nablaG\trace{\Phi\h\trkla{\phi\h\no}\sinc{n}}\dG\\
&-\frac{r\h\nn}{\sqrt{E\h^\Om\trkla{\phi\h\no}}}\iO\Ih{F^\prime\trkla{\phi\h\no}\Phi\h\trkla{\phi\h\no}\sinc{n}}\dx\\
&+\frac{r\h\nn}{4\tekla{E\h^\Om\trkla{\phi\h\no}}^{3/2}}\abs{\iO\Ih{F^\prime\trkla{\phi\h\no}\Phi\h\trkla{\phi\h\no}\sinc{n}}\dx}^2\\
&-\frac{r\h\nn}{2\sqrt{E\h^\Om\trkla{\phi\h\no}}}\iO\Ih{F^{\prime\prime}\trkla{\phi\h\no}\abs{\Phi\h\trkla{\phi\h\no}\sinc{n}}^2}\dx\\
&-\frac{s\h\nn}{\sqrt{E\h^\Gamma\trkla{\phi\h\no}}}\iGamma\IhG{G^\prime\trkla{\phi\h\no}\trace{\Phi\h\trkla{\phi\h\no}\sinc{n}}}\dG\\
&+\frac{s\h\nn}{4\tekla{E\h^\Gamma\trkla{\phi\h\no}}^{3/2}}\abs{\iGamma\IhG{G^\prime\trkla{\phi\h\no}\trace{\Phi\h\trkla{\phi\h\no}\sinc{n}}}\dG}^2\\
&-\frac{s\h\nn}{2\sqrt{E\h^\Gamma\trkla{\phi\h\no}}}\iGamma\IhG{G^{\prime\prime}\trkla{\phi\h\no}\trace{\Phi\h\trkla{\phi\h\no}\sinc{n}}^2}\dG\\
&+\tau\iO\tabs{\nabla\mu\h\nn}^2\dx+\tau\iGamma\IhG{\tabs{\theta\h\nn}^2}\dG\,.
\end{split}
\end{align}
Using \eqref{eq:existencer} and \eqref{eq:existences} multiplied by $r\h\nn$ and $s\h\nn$, respectively,

\begin{align}
\begin{split}
\tfrac12&\norm{\nabla\phi\h\nn}_{L^2\trkla{\Om}}^2+\tfrac12\norm{\nabla\trkla{\phi\h\nn-\phi\h\no}}_{L^2\trkla{\Om}}^2-\tfrac12\norm{\nabla\phi\h\no}_{L^2\trkla{\Om}}^2\\
&+\tfrac12\norm{\nablaG\phi\h\nn}_{L^2\trkla{\Gamma}}^2+\tfrac12\norm{\nablaG\trkla{\phi\h\nn-\phi\h\no}}_{L^2\trkla{\Gamma}}^2-\tfrac12\norm{\nablaG\phi\h\no}_{L^2\trkla{\Gamma}}^2\\
&+\tabs{r\h\nn}^2+\tabs{r\h\nn-r\h\no}^2-\tabs{r\h\no}^2 +\tabs{s\h\nn}^2+\tabs{s\h\nn-s\h\no}^2-\tabs{s\h\no}^2\\
&+\tau\norm{\nabla\mu\h\nn}_{L^2\trkla{\Om}}^2+\tau\norm{\theta\h\nn}_{h,\Gamma}^2
\end{split}\nonumber\\
\begin{split}\label{eq:discen1}
=&\,\iO\nabla\phi\h\nn\cdot\nabla\rkla{\Phi\h\trkla{\phi\h\no}\sinc{n}}\dx+\iGamma\nablaG\phi\h\nn\cdot\nablaG\trace{\Phi\h\trkla{\phi\h\no}\sinc{n}}\dG\\
&+\frac{r\h\nn}{\sqrt{E\h^\Om\trkla{\phi\h\no}}}\iO\Ih{F^\prime\trkla{\phi\h\no}\Phi\h\trkla{\phi\h\no}\sinc{n}}\dx\\
&-\frac{r\h\nn}{4\tekla{E\h^\Om\trkla{\phi\h\no}}^{3/2}}\abs{\iO\Ih{F^\prime\trkla{\phi\h\no}\Phi\h\trkla{\phi\h\no}\sinc{n}}\dx}^2\\
&+\frac{r\h\nn}{2\sqrt{E\h^\Om\trkla{\phi\h\no}}}\iO\Ih{F^{\prime\prime}\trkla{\phi\h\no}\abs{\Phi\h\trkla{\phi\h\no}\sinc{n}}^2}\dx
\end{split}\\
\begin{split}\nonumber
&+\frac{s\h\nn}{\sqrt{E\h^\Gamma\trkla{\phi\h\no}}}\iGamma\IhG{G^\prime\trkla{\phi\h\no}\trace{\Phi\h\trkla{\phi\h\no}\sinc{n}}}\dG\\
&-\frac{s\h\nn}{4\tekla{E\h^\Gamma\trkla{\phi\h\no}}^{3/2}}\abs{\iGamma\IhG{G^\prime\trkla{\phi\h\no}\trace{\Phi\h\trkla{\phi\h\no}\sinc{n}}}\dG}^2\\
&+\frac{s\h\nn}{2\sqrt{E\h^\Gamma\trkla{\phi\h\no}}}\iGamma\IhG{G^{\prime\prime}\trkla{\phi\h\no}\trace{\Phi\h\trkla{\phi\h\no}\sinc{n}}^2}\dG\\
=:&\,S_1+S_2+S_3+S_4+S_5+S_6+S_7+S_8\,.
\end{split}
\end{align}
As the scalar auxiliary variables $r\h\nn$ and $s\h\nn$ are merely approximations of positive terms, there is no non-negativity result available. 
Therefore, the terms $S_4$ and $S_7$ can not simply be neglected.
Similarly to \cite{\citeASAV}, we apply Young's inequality to separate the implicit terms and and the stochastic increments $\sinc{n}$.
In particular, we obtain
\begin{align}
\begin{split}
S_1\leq &\,\tfrac14\norm{\nabla\phi\h\nn-\nabla\phi\h\no}_{L^2\trkla{\Om}}^2 + C\norm{\nabla\rkla{\Phi\h\trkla{\phi\h\no}\sinc{n}}}_{L^2\trkla{\Om}}^2 \\
&+\iO\nabla\phi\h\no\cdot\nabla\rkla{\Phi\h\trkla{\phi\h\no}\sinc{n}}\dx\,,
\end{split}\\
\begin{split}
S_2\leq&\,\tfrac14 \norm{\nablaG\phi\h\nn-\nablaG\phi\h\no}_{L^2\trkla{\Gamma}}^2 +C\norm{\nablaG\trace{\Phi\h\trkla{\phi\h\no}\sinc{n}}}_{L^2\trkla{\Gamma}}\\
&+\iGamma\nablaG\phi\h\no\cdot\nablaG\trace{\Phi\h\trkla{\phi\h\no}\sinc{n}}\dG\,,
\end{split}\\
\begin{split}
S_3\leq&\,\tfrac14\tabs{r\h\nn-r\h\no}^2 + C\frac{1}{E\h^\Om\trkla{\phi\h\no}}\abs{\iO\Ih{F^\prime\trkla{\phi\h\no}\Phi\h\trkla{\phi\h\no}\sinc{n}}\dx}^2\\
&+\frac{r\h\no}{\sqrt{E\h^\Om\trkla{\phi\h\no}}}\iO\Ih{F^\prime\trkla{\phi\h\no}\Phi\h\trkla{\phi\h\no}\sinc{n}}\dx\,,
\end{split}\\
\begin{split}
S_4\leq&\,\tfrac14\tabs{r\h\nn-r\h\no}^2 +C\frac{1}{\tekla{E\h^\Om\trkla{\phi\h\no}}^3}\abs{\iO\Ih{F^\prime\trkla{\phi\h\no}\Phi\h\trkla{\phi\h\no}\sinc{n}}\dx}^4\\
&-\frac{r\h\no}{4\ekla{E\h^\Om\trkla{\phi\h\no}}^{3/2}}\abs{\iO\Ih{F^\prime\trkla{\phi\h\no}\Phi\h\trkla{\phi\h\no}\sinc{n}}\dx}^2\,,
\end{split}\\
\begin{split}
S_5\leq&\,\tfrac14\tabs{r\h\nn-r\h\no}^2+C\frac{1}{E\h^\Om\trkla{\phi\h\no}}\abs{\iO\Ih{F^{\prime\prime}\trkla{\phi\h\no}\abs{\Phi\h\trkla{\phi\h\no}\sinc{n}}^2}\dx}^2\\
&+\frac{r\h\no}{2\sqrt{E\h^\Om\trkla{\phi\h\no}}}\iO\Ih{F^{\prime\prime}\trkla{\phi\h\no}\abs{\Phi\h\trkla{\phi\h\no}\sinc{n}}^2}\dx\,,
\end{split}\\
\begin{split}
S_6\leq&\,\tfrac14\tabs{s\h\nn-s\h\no}^2+C\frac{1}{E\h^\Gamma\trkla{\phi\h\no}}\abs{\iGamma\IhG{G^\prime\trkla{\phi\h\no}\trace{\Phi\h\trkla{\phi\h\no}\sinc{n}}}\dG}^2\\
&+\frac{s\h\no}{\sqrt{E\h^\Gamma\trkla{\phi\h\no}}}\iGamma\IhG{G^\prime\trkla{\phi\h\no}\trace{\Phi\h\trkla{\phi\h\no}\sinc{n}}}\dG\,,
\end{split}\\
\begin{split}
S_7\leq&\,\tfrac14\tabs{s\h\nn-s\h\no}^2+C\frac{1}{\ekla{E\h^\Gamma\trkla{\phi\h\no}}^2}\abs{\iGamma\IhG{G^\prime\trkla{\phi\h\no}\trace{\Phi\h\trkla{\phi\h\no}\sinc{n}}}\dG}^4\\
&-\frac{s\h\no}{4\ekla{E\h^\Gamma\trkla{\phi\h\no}}^{3/2}}\abs{\iGamma\IhG{G^\prime\trkla{\phi\h\no}\trace{\Phi\h\trkla{\phi\h\no}\sinc{n}}}\dG}^2\,,
\end{split}\\
\begin{split}
S_8\leq&\,\tfrac14\tabs{s\h\nn-s\h\no}^2+C\frac{1}{E\h^\Gamma\trkla{\phi\h\no}}\abs{\iGamma\IhG{G^{\prime\prime}\trkla{\phi\h\no}\trace{\Phi\h\trkla{\phi\h\no}\sinc{n}}^2}\dG}^2\\
&+\frac{s\h\no}{2\sqrt{E\h^\Gamma\trkla{\phi\h\no}}}\iGamma\IhG{G^{\prime\prime}\trkla{\phi\h\no}\trace{\Phi\h\trkla{\phi\h\no}\sinc{n}}^2}\dG\,.
\end{split}
\end{align}
As we allow the stochastic term in \eqref{eq:modeldisc:phiO} to act as a source or sink term, controlling the $H^1\trkla{\Om}$-semi-norm of the phase-field parameter is insufficient.
In order to obtain the full $H^1$-norms on the left-hand side of \eqref{eq:discen1}, we choose $\psi\h\equiv\phi\h\nn$ in \eqref{eq:modeldisc:phiO} and $\widehat{\psi}\h\equiv\trace{\phi\h\nn}$ in \eqref{eq:modeldisc:phiG}.
After applying Young's inequality, we obtain
\begin{align}
\begin{split}
\tfrac12&\norm{\phi\h\nn}_{h,\Om}^2+\tfrac12\norm{\phi\h\nn-\phi\h\no}_{h,\Om}^2-\tfrac12\norm{\phi\h\no}_{h,\Om}^2 +\tfrac12\norm{\phi\h\nn}_{h,\Gamma}^2+\tfrac12\norm{\phi\h\nn-\phi\h\no}_{h,\Gamma}^2-\tfrac12\norm{\phi\h\no}_{h,\Gamma}^2\\
&\leq \tau\tfrac34\norm{\nabla\mu\h\nn}_{L^2\trkla{\Om}}^2+\tau\tfrac13\norm{\nabla\phi\h\nn}_{L^2\trkla{\Om}}^2 +\tfrac14\norm{\phi\h\nn-\phi\h\no}_{h,\Om}^2+\norm{\Phi\h\trkla{\phi\h\no}\sinc{n}}_{h,\Om}^2\\
&~+\iO\Ih{\phi\h\no\Phi\h\trkla{\phi\h\no}\sinc{n}}\dx +\tau\tfrac34\norm{\theta\h\nn}_{h,\Gamma}^2+\tau\tfrac13\norm{\phi\h\nn}_{h,\Gamma}^2+\tfrac14\norm{\phi\h\nn-\phi\h\no}_{h,\Gamma}^2\\
&~+\norm{\trace{\Phi\h\trkla{\phi\h\no}\sinc{n}}}_{h,\Gamma}^2+\iGamma\IhG{\phi\h\no\trace{\Phi\h\trkla{\phi\h\no}\sinc{n}}}\dG\,.
\end{split}
\end{align}
Combining the above equations and summing from $n=1$ to $m$, we obtain
\begin{align}
\begin{split}
\tfrac12&\norm{\phi\h^m}_{H^1\h\trkla{\Om}}^2+\tfrac12\norm{\phi\h^m}_{H^1\h\trkla{\Gamma}}^2 +\tabs{r\h^m}^2+\tabs{s\h^m}^2+\tfrac14\sum_{n=1}^m\norm{\phi\h\nn-\phi\h\no}_{H^1\h\trkla{\Om}}^2\\
&+\tfrac14\sum_{n=1}^m\norm{\phi\h\nn-\phi\h\no}_{H^1\h\trkla{\Gamma}}^2+\tfrac14\sum_{n=1}^m\tabs{r\h\nn-r\h\no}^2+\tfrac14\sum_{n=1}^m\tabs{s\h\nn-s\h\no}^2\\
&+\tfrac14\tau\sum_{n=1}^m\norm{\nabla\mu\h\nn}_{L^2\trkla{\Om}}^2+\tfrac14\tau \sum_{n=1}^m\norm{\theta\h\nn}_{h,\Gamma}^2\nonumber
\end{split}\\
\begin{split}
\leq&\,\tfrac12\norm{\phi\h^0}_{H^1\h\trkla{\Om}}^2+\tfrac12\norm{\phi\h^0}_{H^1\h\trkla{\Gamma}}^2+\tabs{r\h^0}^2+\tabs{s\h^0}^2 +\tfrac13\tau\sum_{n=1}^m\norm{\nabla\phi\h\nn}_{L^2\trkla{\Om}}^2+\tfrac13\tau\sum_{n=1}^m\norm{\phi\h\nn}_{h,\Gamma}^2\\
&+C\sum_{n=1}^m\norm{\nabla\rkla{\Phi\h\trkla{\phi\h\no}\sinc{n}}}_{L^2\trkla{\Om}}^2+\sum_{n=1}^m\iO\nabla\phi\h\no\cdot\nabla\rkla{\Phi\h\trkla{\phi\h\no}\sinc{n}}\dx\\
&+C\sum_{n=1}^m\frac{1}{E\h^\Om\trkla{\phi\h\no}}\abs{\iO\Ih{F^\prime\trkla{\phi\h\no}\Phi\h\trkla{\phi\h\no}\sinc{n}}\dx}^2\\
&+\sum_{n=1}^m\frac{r\h\no}{\sqrt{E\h^\Om\trkla{\phi\h\no}}}\iO\Ih{F^\prime\trkla{\phi\h\no}\Phi\h\trkla{\phi\h\no}\sinc{n}}\dx\\
&+C\sum_{n=1}^m\frac{1}{\ekla{E\h^\Om\trkla{\phi\h\no}}^{3}}\abs{\iO\Ih{F^\prime\trkla{\phi\h\no}\Phi\h\trkla{\phi\h\no}\sinc{n}}\dx}^4\\
&-\sum_{n=1}^m\frac{r\h\no}{4\ekla{E\h^\Om\trkla{\phi\h\no}}^{3/2}}\abs{\iO\Ih{F^\prime\trkla{\phi\h\no}\Phi\h\trkla{\phi\h\no}\sinc{n}}\dx}^2\\
&+C\sum_{n=1}^m\frac{1}{E\h^\Om\trkla{\phi\h\no}}\abs{\iO\Ih{F^{\prime\prime}\trkla{\phi\h\no}\abs{\Phi\h\trkla{\phi\h\no}\sinc{n}}^2}\dx}^2\\
&+\sum_{n=1}^m\frac{r\h\no}{2\sqrt{E\h^\Om\trkla{\phi\h\no}}}\iO\Ih{F^{\prime\prime}\trkla{\phi\h\no}\abs{\Phi\h\trkla{\phi\h\no}\sinc{n}}^2}\dx\\
&+C\sum_{n=1}^m\norm{\nablaG\trace{\Phi\h\trkla{\phi\h\no}\sinc{n}}}_{L^2\trkla{\Gamma}}^2 \\
&+\sum_{n=1}^m\iGamma\nablaG\phi\h\no\cdot\nablaG\trace{\Phi\h\trkla{\phi\h\no}\sinc{n}}\dG\\
&+C\sum_{n=1}^m\frac{1}{E\h^\Gamma\trkla{\phi\h\no}}\abs{\iGamma\IhG{G^\prime\trkla{\phi\h\no}\trace{\Phi\h\trkla{\phi\h\no}\sinc{n}}}\dG}^2\\
&+\sum_{n=1}^m\frac{s\h\no}{\sqrt{E\h^\Gamma\trkla{\phi\h\no}}}\iGamma\IhG{G^\prime\trkla{\phi\h\no}\trace{\Phi\h\trkla{\phi\h\no}\sinc{n}}}\dG\\
&+C\sum_{n=1}^m\frac{1}{\ekla{E\h^\Gamma\trkla{\phi\h\no}}^3}\abs{\iGamma\IhG{G^\prime\trkla{\phi\h\no}\trace{\Phi\h\trkla{\phi\h\no}\sinc{n}}}\dG}^4\\
&-\sum_{n=1}^m\frac{s\h\no}{4\ekla{E\h^\Gamma\trkla{\phi\h\no}}^{3/2}}\abs{\iGamma\IhG{G^\prime\trkla{\phi\h\no}\trace{\Phi\h\trkla{\phi\h\no}\sinc{n}}}\dG}^2\\
&+C\sum_{n=1}^m\frac{1}{E\h^\Gamma\trkla{\phi\h\no}}\abs{\iGamma\IhG{G^{\prime\prime}\trkla{\phi\h\no}\trace{\Phi\h\trkla{\phi\h\no}\sinc{n}}^2}\dG}^2\\
&+\sum_{n=1}^m\frac{s\h\no}{2\sqrt{E\h^\Gamma\trkla{\phi\h\no}}}\iGamma\IhG{G^{\prime\prime}\trkla{\phi\h\no}\trace{\Phi\h\trkla{\phi\h\no}\sinc{n}}^2}\dG\\
=:&\,\tfrac12\norm{\phi\h^0}_{H^1\h\trkla{\Om}}^2+\tfrac12\norm{\phi\h^0}_{H^1\h\trkla{\Gamma}}^2+\tabs{r\h^0}^2+\tabs{s\h^0}^2 +\tfrac13\tau\sum_{n=1}^m\norm{\nabla\phi\h\nn}_{L^2\trkla{\Om}}^2+\tfrac13\tau\sum_{n=1}^m\norm{\phi\h\nn}_{h,\Gamma}^2\\
&+\sum_{\alpha=1}^{16} R_{\alpha,m}\,.
\end{split}
\end{align}
Absorbing $\tau\tfrac13\norm{\nabla\phi\h^m}_{L^2\trkla{\Om}}^2$ and $\tau\tfrac13\norm{\phi\h^m}_{h,\Gamma}^2$ on the left-hand side, taking the $\p$-th power, and using he equivalence between the norms defined in \eqref{eq:def:hnorms} and their standard counterparts, we obtain for any $m\in\tgkla{0,\ldots,N}$ the estimate
\begin{align}\label{eq:energytmp:1}
\begin{split}
&\norm{\phi\h^m}_{H^1\trkla{\Om}}^{2\p}+\norm{\phi\h^m}_{H^1\trkla{\Gamma}}^{2\p}+\tabs{r\h^m}^{2\p}+\tabs{s\h^m}^{2\p}+\rkla{\sum_{n=1}^m\norm{\phi\h\nn-\phi\h\no}_{H^1\trkla{\Om}}^2}^\p\\
&\qquad+\rkla{\sum_{n=1}^m\norm{\phi\h\nn-\phi\h\no}_{H^1\trkla{\Gamma}}^2}^\p+\rkla{\sum_{n=1}^m\tabs{r\h\nn-r\h\no}^2}^\p+\rkla{\sum_{n=1}^m\tabs{s\h\nn-s\h\no}^2}^\p\\
&\qquad+\rkla{\tau\sum_{n=1}^m\norm{\nabla\mu\h\nn}_{L^2\trkla{\Om}}^2}^\p+\rkla{\tau\sum_{n=1}^m\norm{\theta\h\nn}_{L^2\trkla{\Gamma}}^2}^\p\\
&\quad\leq C\norm{\phi\h^0}_{H^1\trkla{\Om}}^{2\p}+C\norm{\phi\h^0}_{H^1\trkla{\Gamma}}^{2\p} +C\tabs{r\h^0}^{2\p}+C\tabs{s\h^0}^{2\p}+C\sum_{n=1}^m\tau\norm{\nabla\phi\h\no}_{L^2\trkla{\Om}}^{2\p}\\
&\qquad+C\sum_{n=1}^m\tau\norm{\phi\h\no}_{L^2\trkla{\Gamma}}^{2\p} +C\sum_{\alpha=1}^{16}\tabs{R_{\alpha,m}}^\p\,,
\end{split}
\end{align}
where the constant $C$ depends on $\p$ but not on $h$ or $\tau$.
We will now follow the lines of \cite{\citeASAV} to derive estimates for the expected values of the stochastic terms $\tabs{R_{1,m}}^\p,\ldots,\tabs{R_{16,m}}^\p$.
We start by deducing an estimate for $\expected{\tabs{R_{1,m}}^\p}$.
Applying Hölder's inequality, Lemma \ref{lem:BDG}, \eqref{eq:normequivalence}, and assumptions \ref{item:color} and \ref{item:rho}, we compute
\begin{align}\label{eq:tmp:R1}
\begin{split}
\expected{\tabs{R_{1,m}}^\p}\leq&\,Cm^{\p-1}\sum_{n=1}^m\expected{\norm{\nabla\rkla{\Phi\h\trkla{\phi\h\no}\sinc{n}}}_{L^2\trkla{\Om}}^{2\p}}\\
\leq &\,C\sum_{n=1}^m\tau\expected{\rkla{\sum_{k\in\Zh}\norm{\lambda_k\nabla\rkla{\Phi\h\trkla{\phi\h\no}\g{k}}}_{L^2\trkla{\Om}}^2}^{\p}}\\
\leq&\, C\sum_{n=1}^m\tau\expected{\norm{\Ih{\varrho\trkla{\phi\h\no}}}_{H^1\trkla{\Om}}^{2\p}\rkla{\sum_{k\in\mathds{Z}}\lambda_k^2\norm{\g{k}}_{W^{1,\infty}\trkla{\Om}}^2}^{\p}}\\
\leq&\,C\sum_{n=1}^m\tau\expected{\rkla{1+\norm{\nabla\phi\h\no}_{L^2\trkla{\Om}}^{2\p}}}\,.
\end{split}
\end{align}
For $R_{2,m}$, we obtain from Lemma \ref{lem:BDG}
\begin{align}
\begin{split}
\expected{\tabs{R_{2,m}}^\p}\leq&\,\expected{\rkla{\max_{1\leq l\leq m}\abs{\sum_{n=1}^l\iO\nabla\rkla{\Phi\h\trkla{\phi\h\no}\sinc{n}}\cdot\nabla\phi\h\no\dx}}^\p}\\
\leq&\,C\sum_{n=1}^m\tau\expected{\rkla{\sum_{k\in\Zh}\lambda_k^2\abs{\iO\nabla\Ih{\varrho\trkla{\phi\h\no}\g{k}}\cdot\nabla\phi\h\no\dx}^2}^{\p/2}}\\
\leq &\, C\sum_{n=1}^m\tau\expected{\norm{\Ih{\varrho\trkla{\phi\h\no}}}_{H^1\trkla{\Om}}^\p\norm{\nabla\phi\h\no}_{L^2\trkla{\Om}}^\p}\\
\leq&\,C+C\expected{\sum_{n=1}^m\tau\norm{\nabla\phi\h\no}_{L^2\trkla{\Om}}^{2\p}}\,.
\end{split}
\end{align}
Using similar arguments, we obtain for $R_{3,m}$
\begin{align}
\begin{split}
\expected{\tabs{R_{3,m}}^\p}\leq C\sum_{n=1}^m\tau\expected{\frac1{\tekla{E\h^\Om\trkla{\phi\h\no}}^\p}\norm{\Ih{\varrho\trkla{\phi\h\no} F^\prime\trkla{\phi\h\no}}}_{L^1\trkla{\Om}}^{2\p}}\,.
\end{split}
\end{align}
As the left-hand side of \eqref{eq:energytmp:1} only includes $2\p$-th powers of $\phi\h^m$, it will not be possible to control $\norm{F^\prime\trkla{\phi\h\no}}_{L^1\trkla{\Om}}^{2\p}$.
We therefore follow the approach used in \cite{\citeASAV} (see also \cite{LamWang2023}) and use the negative powers of $E\h^\Om\trkla{\phi\h\no}$ to our advantage.
As $\varrho$ is bounded, we can use the growth condition stated in \ref{item:potentials}, the norm equivalence \eqref{eq:normequivalence}, and Hölder's inequality to obtain
\begin{align}\label{eq:fprime}
\norm{\Ih{\varrho\trkla{\phi\h\no}F^\prime\trkla{\phi\h\no}}}_{L^1\trkla{\Om}}^2\leq C +C\norm{\phi\h\no}_{L^2\trkla{\Om}}^2\iO\Ih{F\trkla{\phi\h\no}}\dx\,.
\end{align}
Hence, the lower bound on $F$ provides
\begin{align}
\expected{\tabs{R_{3,m}}^\p}\leq C+C\sum_{n=1}^m\tau\expected{\norm{\phi\h\no}_{L^2\trkla{\Om}}^{2\p}}\,.
\end{align}
Reusing \eqref{eq:fprime}, we estimate $\expected{\tabs{R_{4,m}}^\p}$ via
\begin{align}
\begin{split}
\expected{\tabs{R_{4,m}}^\p}\leq&\, C\sum_{n=1}^m\tau\expected{\rkla{\frac{\tabs{r\h\no}}{\sqrt{E\h^\Om\trkla{\phi\h\no}}}\norm{\Ih{\varrho\trkla{\phi\h\no}F^\prime\trkla{\phi\h\no}}}_{L^1\trkla{\Om}}}^\p}\\
\leq&\, C+C\sum_{n=1}^m\tau\expected{\tabs{r\h\no}^{2\p}}+C\sum_{n=1}^m\tau\expected{\norm{\phi\h\no}_{L^2\trkla{\Om}}^{2\p}}\,.
\end{split}
\end{align}
Similar to the lines of the proof of \cite[Lemma 6.1]{\citeASAV}, we obtain
\begin{align}
\expected{\tabs{R_{5,m}}^\p}\leq & C\tau^\p\,,\\
\expected{\tabs{R_{6,m}}^\p}\leq & C\sum_{n=1}^m\tau\expected{\tabs{r\h\no}^{2\p}}+C\,,\\
\expected{\tabs{R_{7,m}}^\p}\leq & C\tau^\p\,,\\
\expected{\tabs{R_{8,m}}^\p}\leq & C\sum_{n=1}^m\tau\expected{\tabs{r\h\no}^{2\p}}+C\,.
\end{align}
The remaining terms can be estimated analogously using \eqref{eq:tracecolor}.
In particular, we interpret $\nablaG\trace{\Phi\h\trkla{\phi\h\no}\,\cdot\,}$ as a Hilbert--Schmidt operator onto $\trkla{L^2\trkla{\Gamma}}^{d-1}$ and apply Lemma \ref{lem:BDG} to obtain
\begin{align}
\begin{split}
\expected{\tabs{R_9,m}^\p}\leq&\,C m^{\p-1}\sum_{n=1}^m\expected{\norm{\nablaG\trace{\Phi\h\trkla{\phi\h\no}\sinc{n}}}_{L^2\trkla{\Gamma}}^{2\p}}\\
\leq &\,C\sum_{n=1}^m\tau\expected{\norm{\IhG{\varrho\trkla{\phi\h\no}}}_{H^1\trkla{\Gamma}}^{2\p}\rkla{\sum_{k\in\mathds{Z}}\lambda_k^2\norm{\g{k}}_{W^{1,\infty}\trkla{\Gamma}}}^\p}\\
\leq &\,C\sum_{n=1}^m\tau\expected{\rkla{1+\norm{\nablaG\phi\h\no}_{L^2\trkla{\Gamma}}^{2\p}}}\,.
\end{split}
\end{align}
Similarly, we obtain
\begin{align}
\expected{\tabs{R_{10,m}}^\p}\leq&\, C+C\expected{\sum_{n=1}^m\tau\norm{\nablaG\phi\h\no}_{L^2\trkla{\Gamma}}^{2\p}}\,,\\
\expected{\tabs{R_{11,m}}^\p}\leq&\,C+C\expected{\sum_{n=1}^m\tau\norm{\phi\h\no}_{L^2\trkla{\Gamma}}^{2\p}}\,,\\
\expected{\tabs{R_{12,m}}^\p}\leq&\,C+C\expected{\sum_{n=1}^m\tau\tabs{s\h\no}^{2\p}}+C\expected{\sum_{n=1}^m\tau\norm{\phi\h\no}_{L^2\trkla{\Gamma}}^{2\p}}\,,\\
\expected{\tabs{R_{13,m}}^\p}\leq&\,C\tau^\p\,,\\
\expected{\tabs{R_{14,m}}^\p}\leq&\,C+C\expected{\sum_{n=1}^m\tau\tabs{s\h\no}^{2\p}}\,,\\
\expected{\tabs{R_{15,m}}^\p}\leq&\,C\tau^\p\,,\\
\expected{\tabs{R_{16,m}}^\p}\leq&\,C+C\expected{\sum_{n=1}^m\tau\tabs{s\h\no}^{2\p}}\,.
\end{align}
Hence, we obtain from \eqref{eq:energytmp:1} after neglecting non-negative terms
\begin{align}
\begin{split}
&\expected{\norm{\phi\h^m}_{H^1\trkla{\Om}}^{2p}}+\expected{\norm{\phi\h^m}_{H^1\trkla{\Gamma}}^{2\p}}+\expected{\tabs{r\h^m}^{2\p}}+\expected{\tabs{s\h^m}^{2\p}}\\
&\quad\leq C\norm{\phi\h^0}_{H^1\trkla{\Om}}^{2\p}+C\norm{\phi\h^0}_{H^1\trkla{\Gamma}}^{2\p}+C\tabs{r\h^0}^{2\p}+C\tabs{s\h^0}^{2\p} +C\sum_{n=1}^m\tau\expected{\norm{\phi\h\no}_{H^1\trkla{\Om}}^{2\p}}\\
&\qquad+C\sum_{n=1}^m\tau\expected{\norm{\phi\h\no}_{H^1\trkla{\Gamma}}^{2\p}} +C + C\sum_{n=1}^m\tau\expected{\tabs{r\h\no}^{2\p}}+ C\sum_{n=1}^m\tau\expected{\tabs{s\h\no}^{2\p}}
\end{split}
\end{align}
for all $m\in\tgkla{1,\dots,N}$.
As assumption \ref{item:initial} allows us to control the first four terms on the right-hand side independently of the discretization parameters, we can apply a discrete version of Gronwall's inequality and obtain \eqref{eq:energystep1}.
In the second step, we shall now establish \eqref{eq:energyestimate} for $\p\in2\mathds{N}$.
Starting from \eqref{eq:energytmp:1}, we can take the maximum over $m\in\tgkla{1,\ldots,N}$ before taking the expected value.
Reusing the above calculations and in particular \eqref{eq:energystep1}, provides the claim.
\end{proof}
In the next step, we shall analyze \eqref{eq:modeldisc:pot} in more detail and verify that the additional terms $\Xi_{h,\Om}\nn$ and $\Xi_{h,\Gamma}\nn$ introduced in the approximation of $F^\prime\trkla{\phi}$ and $G^\prime\trkla{\phi}$ (cf.~\eqref{eq:deferrorterms}) will vanish for $\tau\searrow0$.

\begin{lemma}\label{lem:potential}
Let the assumptions \ref{item:time}, \ref{item:space1}, \ref{item:space2}, \ref{item:potentials}, \ref{item:initial}, \ref{item:rho}, and \ref{item:filtration}-\ref{item:color} hold true.
Then, for every $1\leq q<\infty$, there exists a positiv constant $C$ independent of $h$ and $\tau$ such that

\begin{subequations}
\begin{align}
\expected{\rkla{\sum_{n=1}^N\tau\norm{\Xi_{h,\Om}\nn}_{h,\Om}^2}^q}\leq C\tau^q\,,\label{eq:errorOm}\\
\expected{\rkla{\sum_{n=1}^N\tau\norm{\Xi_{h,\Gamma}\nn}_{h,\Om}^2}^q}\leq C\tau^q\,.\label{eq:errorGam}
\end{align}
\end{subequations}
Furthermore, the following estimate on the $L^2\trkla{\Om}$-norm holds true:
\begin{align}\label{eq:mul2}
\expected{\rkla{\sum_{n=1}^N\tau\norm{\mu\h\nn}_{L^2\trkla{\Om}}^2}^q}\leq C\,.
\end{align}
\end{lemma}
\begin{proof}
We start by establishing \eqref{eq:errorOm}.
Discussing both summands in the definition of $\Xi_{h,\Om}\nn$ separately, we obtain by using Hölder's inequality, the lower bound on $E\h^\Om\trkla{\cdot}$, and the growth estimate for $F^\prime$:
\begin{align}
\begin{split}
&\expected{\rkla{\sum_{n=1}^N\tau\frac{\tabs{r\h\nn}^2}{\ekla{E\h^\Om\trkla{\phi\h\no}}^3}\abs{\iO\Ih{F^\prime\trkla{\phi\h\no}\Phi\h\trkla{\phi\h\no}\sinc{n}}\dx}^2\norm{\Ih{F^\prime\trkla{\phi\h\no}}}_{h,\Om}^2}^q}\\
&\quad\leq C\expected{\rkla{\max_{1\leq n\leq N}\tabs{r\h\nn}^2\rkla{1+\max_{1\leq n\leq N}\norm{\phi\h\no}_{H^1\trkla{\Om}}^{6}}}^{2q}}^{1/2}\\
&\qquad\times \expected{\rkla{\sum_{n=1}^N\tau\abs{\iO\Ih{F^\prime\trkla{\phi\h\no}\Phi\h\trkla{\phi\h\no}\sinc{n}}\dx}^2}^{2q}}^{1/2}\,.
\end{split}
\end{align}
Here, the first factor on the right-hand side can immediately be controlled using Lemma \ref{lem:energy}.
To control the second factor, we apply Lemma \ref{lem:BDG}, \eqref{eq:def:phih}, \ref{item:rho}, and \ref{item:color} to obtain
\begin{multline}
\expected{\rkla{\sum_{n=1}^N\tau\abs{\iO\Ih{F^\prime\trkla{\phi\h\no}\Phi\h\trkla{\phi\h\no}\sinc{n}}\dx}^2}^{2q}}\\
\leq C\expected{\sum_{n=1}^N\tau\abs{\iO\Ih{F^\prime\trkla{\phi\h\no}\Phi\h\trkla{\phi\h\no}\sinc{n}}\dx}^{4q}}\\
\leq C\tau^{2q}\sum_{n=1}^N\tau\expected{\rkla{\sum_{k\in\Zh}\abs{\lambda_k\iO\Ih{F^\prime\trkla{\phi\h\no}\varrho\trkla{\phi\h\no}\g{k}}\dx}^2}^{2q}}\\
\leq C\tau^{2q}\sum_{n=1}^N\tau\expected{\rkla{1+\norm{\phi\h\no}_{H^1\trkla{\Om}}^{12q}}\rkla{\sum_{k\in\Zh}\lambda_k^2\norm{\g{k}}_{L^\infty\trkla{\Om}}^2}^{2q}}\leq C\tau^{2q}\,.
\end{multline}
The estimate for the second summand in \eqref{eq:deferrorterms:Om} can be deduced using similar considerations:
\begin{multline}
\expected{\rkla{\sum_{n=1}^N\tau\frac{\tabs{r\h\nn}^2}{E\h^\Om\trkla{\phi\h\no}}\norm{\Ih{F^{\prime\prime}\trkla{\phi\h\no}\Phi\h\trkla{\phi\h\no}\sinc{n}}}_{h,\Om}^2}^q}\\
\leq C\expected{\max_{1\leq n\leq N}\tabs{r\h\nn}^{4q}}^{1/2}\expected{\sum_{n=1}^N\tau\norm{\Ih{F^{\prime\prime}\trkla{\phi\h\no}\Phi\h\trkla{\phi\h\no}\sinc{n}}}_{h,\Om}^{4q}}^{1/2}\\
\leq C\rkla{\sum_{n=1}^N\tau\expected{\rkla{\tau\sum_{k\in\Zh}\norm{\Ih{F^{\prime\prime}\trkla{\phi\h\no}\varrho\trkla{\phi\h\no}\lambda_k\g{k}}}_{h,\Om}^2}^{2q}}}^{1/2}\leq C\tau^q\,,
\end{multline}
due to Hölder's inequality, Lemma \ref{lem:energy}, Lemma \ref{lem:BDG}, \ref{item:rho}, \ref{item:color}, and \ref{item:potentials}.
The estimate \eqref{eq:errorGam} on $\Xi_{h,\Gamma}\nn$ can be proven analogously by applying Lemma \ref{lem:BDG} to the the operators $\IhG{G^\prime\trkla{\phi\h\no}\trace{\Phi\h\trkla{\phi\h\no}\,\cdot\,}}$ and $\IhG{G^{\prime\prime}\trkla{\phi\h\no}\trace{\Phi\h\trkla{\phi\h\no}\,\cdot\,}}$ and recalling \eqref{eq:tracecolor}.\\
Choosing $\eta\h\equiv1$ in \eqref{eq:modeldisc:pot}, we obtain
\begin{align}
\begin{split}
\abs{\iO\mu\h\nn\dx}\leq &\,\abs{\iGamma\theta\h\nn\dG} +\abs{\frac{r\h\nn}{\sqrt{E\h^\Om\trkla{\phi\h\no}}}\iO\Ih{F^\prime\trkla{\phi\h\no}}\dx}\\
&+\abs{\frac{r\h\nn}{4\tekla{E\h^\Om\trkla{\phi\h\no}}^{3/2}}\iO\Ih{F^\prime\trkla{\phi\h\no}\Phi\h\trkla{\phi\h\no}\sinc{n}}\dx\iO\Ih{F^\prime\trkla{\phi\h\no}}\dx}\\
&+\abs{\frac{r\h\nn}{2\sqrt{E\h^\Om\trkla{\phi\h\no}}}\iO\Ih{F^{\prime\prime}\trkla{\phi\h\no}\Phi\h\trkla{\phi\h\no}\sinc{n}}\dx}\\
&+\abs{\frac{s\h\nn}{\sqrt{E\h^\Gamma\trkla{\phi\h\no}}}\iGamma\IhG{G^\prime\trkla{\phi\h\no}}\dG}\\
&+\abs{\frac{s\h\nn}{4\tekla{E\h^\Gamma\trkla{\phi\h\no}}^{3/2}}\iGamma\IhG{G^\prime\trkla{\phi\h\no}\trace{\Phi\h\trkla{\phi\h\no}\sinc{n}}}\dG\iGamma\IhG{G^\prime\trkla{\phi\h\no}}\dG}\\
&+\abs{\frac{s\h\nn}{2\sqrt{E\h^\Gamma\trkla{\phi\h\no}}}\iGamma\IhG{G^{\prime\prime}\trkla{\phi\h\no}\trace{\Phi\h\trkla{\phi\h\no}\sinc{n}}}\dG}\,.
\end{split}
\end{align}
Taking the second power on both sides, summing from $n=1$ to $N$, taking the $q$-th power, and taking the expected value, provides due to the above estimates
\begin{multline}
\expected{\rkla{\sum_{n=1}^N\tau\abs{\iO\mu\h\nn\dx}^2}^q}\\
\leq C\expected{\rkla{\sum_{n=1}^N\tau\norm{\theta\h\nn}_{L^2\trkla{\Gamma}}^2}^q} + C\expected{\rkla{\sum_{n=1}^N\tau\abs{\frac{r\h\nn}{\sqrt{E\h^\Om\trkla{\phi\h\no}}}\iO\Ih{F^\prime\trkla{\phi\h\no}}\dx}^2}^q}\\
  +C\expected{\rkla{\sum_{n=1}^N\tau\abs{\frac{s\h\nn}{\sqrt{E\h^\Gamma\trkla{\phi\h\no}}}\iGamma\IhG{G^\prime\trkla{\phi\h\no}}\dG}^2}^q}+C\tau^q\,.
\end{multline}
Using the growth conditions for $F^\prime$ and $G^\prime$ stated in \ref{item:potentials} and the regularity results established in Lemma \ref{lem:energy}, we obtain
\begin{multline}
\expected{\rkla{\sum_{n=1}^N\tau\abs{\frac{r\h\nn}{\sqrt{E\h^\Om\trkla{\phi\h\no}}}\iO\Ih{F^\prime\trkla{\phi\h\no}}\dx}^2}^q} \\
\leq C\expected{\max_{1\leq n\leq N}\tabs{r\h\nn}^{2q}\rkla{1+\norm{\phi\h\no}_{H^1\trkla{\Om}}^{6q}}}\leq C\,,
\end{multline}
and
\begin{multline}
\expected{\rkla{\sum_{n=1}^N\tau\abs{\frac{s\h\nn}{\sqrt{E\h^\Gamma\trkla{\phi\h\no}}}\iGamma\IhG{G^\prime\trkla{\phi\h\no}}\dG}^2}^q} \\
\leq C\expected{\max_{1\leq n\leq N}\tabs{s\h\nn}^{2q}\rkla{1+\norm{\phi\h\no}_{H^1\trkla{\Gamma}}^{6q}}}\leq C\,.
\end{multline}
Therefore, we can apply Poincar\'e's inequality to deduce \eqref{eq:mul2}.
\end{proof}
\begin{lemma}\label{lem:nikolskii}
Let the assumptions \ref{item:time}, \ref{item:space1}, \ref{item:space2}, \ref{item:potentials}, \ref{item:initial}, \ref{item:rho}, and \ref{item:filtration}-\ref{item:color} hold true. 
Then for all $\alpha,p\geq1$ and $1\leq\beta\leq 2$ there exists a constant $C>0$ which is independent of $h$ and $\tau$ such that
\begin{subequations}
\begin{align}\label{eq:nikolskii:1}
\expected{\sum_{m=0}^{N-l}\tau\norm{\phi\h^{m+l}-\phi\h^m}_{L^2\trkla{\Om}}^{2\alpha}}&\leq C\trkla{l\tau}^{\alpha/2}\,,\\
\expected{\rkla{\sum_{n=1}^N\tau\norm{\phi\h\nn-\phi\h\no}_{L^2\trkla{\Om}}^{2\beta}}^p}&\leq C\tau^{\beta p}\,,\label{eq:nikolskii:2}\\
\expected{\sum_{m=0}^{N-l}\tau\norm{\phi\h^{m+l}-\phi\h^m}_{L^2\trkla{\Gamma}}^{2\alpha}}&\leq C\trkla{l\tau}^\alpha\label{eq:nikolskii:3}
\end{align}
\end{subequations}
for all $l=0,\ldots,N$.
\end{lemma}
\begin{proof}
Summing \eqref{eq:modeldisc:phiO} from $n=m+1$ to $m+l\leq N$, choosing $\psi\h=\trkla{\phi\h^{m+l}-\phi\h^m}$, taking the $\alpha$-th power on both sides, multiplying by $\tau$, summing the result from $m=0$ to $N-l$, and computing the expected value, we obtain
\begin{align}
\begin{split}\label{eq:tmp:nikolskii}
\expected{\sum_{m=0}^{N-l}\tau\norm{\phi\h^{m+l}-\phi\h^m}_{h,\Om}^{2\alpha}}\leq &\,C\expected{\sum_{m=0}^{N-l}\tau\abs{\sum_{n=m+1}^{m+l}\tau\iO\nabla\mu\h\nn\cdot\nabla\trkla{\phi\h^{m+l}-\phi\h^m}\dx}^\alpha}\\
&+C\expected{\sum_{m=0}^{N-l}\tau\abs{\iO\sum_{n=m+1}^{m+l}\Ih{\Phi\h\trkla{\phi\h\no}\sinc{n}\trkla{\phi\h^{m+l}-\phi\h^m}}\dx}^\alpha}\\
=:&\,A+B\,.
\end{split}
\end{align}
Using Hölder's inequality multiple times and recalling the already established regularity results from Lemma \ref{lem:energy}, the first term can be controlled via
\begin{align}
\begin{split}
A\leq &\,C\expected{\max_{0\leq n\leq N} \norm{\nabla\phi\h\nn}_{L^2\trkla{\Om}}^\alpha\trkla{l\tau}^{\alpha/2}\sum_{m=0}^{N-l}\tau  \rkla{\sum_{n=m+1}^{m+l}\tau\norm{\nabla\mu\h\nn}_{L^2\trkla{\Om}}^2}^{\alpha/2}}\leq C\trkla{l\tau}^{\alpha/2}\,.
\end{split}
\end{align}
As shown in \cite[Lemma 6.3]{\citeASAV}, the second term is bounded by 
\begin{align}
\begin{split}\label{eq:nikolskii:B}  
B\leq &\, \tfrac14\expected{\sum_{m=0}^{N-l}\tau\norm{\phi\h^{m+l}-\phi\h^m}_{h,\Om}^{2\alpha}}+C\sum_{m=0}^{N-l}\tau\expected{\norm{\sum_{n=m+1}^{m+l}\Ih{\Phi\h\trkla{\phi\h\no}\sinc{n}}}_{h,\Om}^{2\alpha}}\\
\leq&\, \tfrac14\expected{\sum_{m=0}^{N-l}\tau\norm{\phi\h^{m+l}-\phi\h^m}_{h,\Om}^{2\alpha}}+C\sum_{m=0}^{N-l}\tau\trkla{l\tau}^{\alpha-1}\sum_{n=m+1}^{m+l}\tau\expected{\norm{\Ih{\varrho\trkla{\phi\h\no}}}_{L^2\trkla{\Om}}^{2\alpha}}
\end{split}
\end{align}
due to Young's inequality and Lemma \ref{lem:BDG}.
As $\varrho\in L^\infty\trkla{\mathds{R}}$, the second term can be controlled by $C\trkla{l\tau}^{\alpha}$.
By applying the norm equivalence \eqref{eq:normequivalence} we obtain \eqref{eq:nikolskii:1}.\\
To prove \eqref{eq:nikolskii:2}, we set $\alpha=2$ and $l=1$ and obtain analogously to \eqref{eq:tmp:nikolskii} 
\begin{multline}
\expected{\rkla{\sum_{n=1}^{N}\tau\norm{\phi\h\nn-\phi\h\no}_{h,\Om}^4}^p}\leq C\expected{\rkla{\sum_{n=1}^{N}\tau\abs{\tau\iO\nabla\mu\h\nn\cdot\nabla\trkla{\phi\h\nn-\phi\h\no}\dx}^2}^p}\\
+C\expected{\rkla{\sum_{n=1}^N\tau\abs{\iO\Ih{\Phi\h\trkla{\phi\h\no}\sinc{n}\trkla{\phi\h\nn-\phi\h\no}}\dx}^2}^p}\\
=:\hat{A}+\hat{B}\,.
\end{multline}
Applying Hölder's inequality and the regularity results established in Lemma \ref{lem:energy}, we compute
\begin{align}
\begin{split}
\hat{A}\leq&\, C\tau^{2p}\expected{\max_{1\leq n\leq N}\norm{\nabla\phi\h\nn-\nabla\phi\h\no}_{L^2\trkla{\Om}}^{2p}\rkla{\sum_{n=1}^N\tau\norm{\nabla\mu\h\nn}_{L^2\trkla{\Om}}^2}^p}\\
\leq&\,C\tau^{2p}\expected{\rkla{\sum_{n=1}^N\norm{\nabla\phi\h\nn-\nabla\phi\h\no}_{L^2\trkla{\Om}}^2}^{2p}}^{1/2}\expected{\rkla{\sum_{n=1}^N\tau\norm{\nabla\mu\h\nn}_{L^2\trkla{\Om}}^2}^{2p}}^{1/2}\leq C\tau^{2p}\,.
\end{split}
\end{align}
To obtain an estimate for $\hat{B}$, we adapt \eqref{eq:nikolskii:B} and combine Young's inequality, Hölder's inequality, and Lemma \ref{lem:BDG} to obtain
\begin{align}
\begin{split}
\hat{B}\leq&\, \tfrac14\expected{\rkla{\sum_{n=1}^N\tau\norm{\phi\h\nn-\phi\h\no}_{h,\Om}^4}^p}+C\sum_{n=1}^N\tau\expected{\norm{\Ih{\Phi\h\trkla{\phi\h\no}\sinc{n}}}_{h,\Om}^{4p}}\\
\leq&\,\tfrac14\expected{\rkla{\sum_{n=1}^N\tau\norm{\phi\h\nn-\phi\h\no}_{h,\Om}^4}^p}+ C\tau^{2p}\,.
\end{split}
\end{align}
The estimate \eqref{eq:nikolskii:2} for arbitrary $1\leq \beta\leq 2$ then follows by \eqref{eq:normequivalence} and Hölder's inequality.\\
The remaining estimate \eqref{eq:nikolskii:3} can be shown in a similar manner using
\begin{multline}
C\expected{\sum_{m=0}^{N-l}\tau\abs{\sum_{n=m+1}^{m+l}\tau\iGamma\IhG{\theta\h\nn\trkla{\phi\h^{m+l}-\phi\h^m}}\dG}^\alpha}\\
\leq \tfrac14\expected{\sum_{m=0}^{N-l}\tau\norm{\phi\h^{m+l}-\phi\h^m}_{h,\Gamma}^{2\alpha}}+C\trkla{l\tau}^\alpha\expected{\sum_{m=0}^{N-l}\tau\rkla{\sum_{n=m+1}^{m+l}\tau\norm{\theta\h\nn}_{L^2\trkla{\Gamma}}^2}^\alpha}\\
\leq  \tfrac14\expected{\sum_{m=0}^{N-l}\tau\norm{\phi\h^{m+l}-\phi\h^m}_{h,\Gamma}^{2\alpha}}+C\trkla{l\tau}^\alpha
\end{multline}
and
\begin{multline}
C\sum_{m=0}^{N-l}\tau\expected{\norm{\sum_{n=m+1}^{m+l}\IhG{\trace{\Phi\h\trkla{\phi\h\no}\sinc{n}}}}_{h,\Gamma}^{2\alpha}}\\
\leq C\sum_{m=0}^{N-l}\tau\trkla{l\tau}^{\alpha-1}\sum_{n=m+1}^{m+l}\tau\expected{\rkla{\sum_{k\in\mathds{Z}}\lambda_k^2\norm{\g{k}}_{L^\infty\trkla{\Gamma}}^2\norm{\IhG{\varrho\trkla{\phi\h\no}}}_{L^2\trkla{\Gamma}}^2}^\alpha}\leq C\trkla{l\tau}^\alpha\,.
\end{multline}
\end{proof}
Next, we shall derive an estimate for the error introduced by the scalar auxiliary variables:
\begin{lemma}\label{lem:errorsav}
Let the assumptions \ref{item:time}, \ref{item:space1}, \ref{item:space2}, \ref{item:potentials}, \ref{item:initial}, \ref{item:rho}, and \ref{item:filtration}-\ref{item:color} hold true.
Then, for all $p\in[1,\infty)$, there exists a constant $C$ which is independent of $h$ and $\tau$ such that the estimates
\begin{subequations}
\begin{align}
\expected{\max_{0\leq m\leq N}\abs{r\h^m-\sqrt{E\h^\Om\trkla{\phi\h^m}}}^p}\leq &C\tau^{p\max\tgkla{\nu/2,\,1/8}}\,,\label{eq:errorsavO}\\
\expected{\max_{0\leq m\leq N}\abs{s\h^m-\sqrt{E\h^\Gamma\trkla{\phi\h^m}}}^p}\leq &C\tau^{\nu p/2}\,,\label{eq:errorsavG}
\end{align}
\end{subequations}
hold true.
\end{lemma}
\begin{proof}
Adapting the ideas in \cite[Lemma 6.5]{\citeASAV}, we start with a Taylor expansion of $\sqrt{E\h^\Om\trkla{\phi\h\nn}}$ to quantify the approximation error for one time step, i.e.~the difference between the increments $r\h\nn-r\h\no$ and $\sqrt{E\h^\Om\trkla{\phi\h\nn}}-\sqrt{E\h^\Om\trkla{\phi\h\no}}$:
\begin{align}\label{eq:taylor1}
\begin{split}
\sqrt{E\h^\Om\trkla{\phi\h\nn}}&-\sqrt{E\h^\Om\trkla{\phi\h\no}}\\
=&\,\frac{1}{2\sqrt{E\h^\Om\trkla{\phi\h\no}}}\trkla{E\h^\Om\trkla{\phi\h\nn}-E\h^\Om\trkla{\phi\h\no}}-\frac{1}{8\tekla{E\h^\Om\trkla{\phi\h\no}}^{3/2}}\trkla{E\h^\Om\trkla{\phi\h\nn}-E\h^\Om\trkla{\phi\h\no}}^2\\
&+\frac{1}{16\tekla{E\h^\Om\trkla{\varphi_1}}^{5/2}}\rkla{E\h^\Om\trkla{\phi\h\nn}-E\h^\Om\trkla{\phi\h\no}}^3\\
=&\,\frac{1}{2\sqrt{E\h^\Om\trkla{\phi\h\no}}}\rkla{\iO\Ih{F^\prime\trkla{\phi\h\no}\trkla{\phi\h\nn-\phi\h\no}+\tfrac12F^{\prime\prime}\trkla{\phi\h\no}\trkla{\phi\h\nn-\phi\no}^2}\dx}\\
&\,+\frac{1}{2\sqrt{E\h^\Om\trkla{\phi\h\no}}}\iO\Ih{\tfrac12\trkla{F^{\prime\prime}\trkla{\varphi_2}-F^{\prime\prime}\trkla{\phi\h\no}}\trkla{\phi\h\nn-\phi\h\no}^2}\dx\\
&\,-\frac{1}{8\tekla{E\h^\Om\trkla{\phi\h\no}}^{3/2}}\rkla{\iO\Ih{F^\prime\trkla{\phi\h\no}\rkla{\phi\h\nn-\phi\h\no}}\dx}^2\\
&\,-\frac{1}{8\tekla{E\h^\Om\trkla{\phi\h\no}}^{3/2}}\rkla{\iO\Ih{F^\prime\trkla{\phi\h\no}\trkla{\phi\h\nn-\phi\h\no}}\dx}\\
&\qquad\qquad\times\rkla{\iO\Ih{F^{\prime\prime}\trkla{\varphi_3}\trkla{\phi\h\nn-\phi\h\no}^2}\dx}\\
&\,-\frac{1}{8\tekla{E\h^\Om\trkla{\phi\h\no}}^{3/2}}\rkla{\tfrac12\iO\Ih{F^{\prime\prime}\trkla{\varphi_3}\trkla{\phi\h\nn-\phi\h\no}^2}\dx}^2\\
&\,+\frac{1}{16\tekla{E\h^\Om\trkla{\varphi_1}}^{5/2}}\rkla{\iO\Ih{F^\prime\trkla{\varphi_4}\trkla{\phi\h\nn-\phi\h\no}}\dx}^3
\end{split}
\end{align}
with $\varphi_1,\varphi_2,\varphi_3,\varphi_4\in\mathrm{conv}\tgkla{\phi\h\nn,\phi\h\no}$.
Recalling the definition of the discrete Laplacian $\Delta\h$ in \eqref{eq:def_discLaplace}, we write \eqref{eq:modeldisc:phiO} as
\begin{align}
\phi\h\nn-\phi\h\no=\tau\Delta\h\mu\h\nn+\Phi\h\trkla{\phi\h\no}\sinc{n}\,.
\end{align}
Hence, we obtain from \eqref{eq:taylor1}
\begin{align}\label{eq:taylor2}
\begin{split}
\sqrt{E\h^\Om\trkla{\phi\h\nn}}&\,-\sqrt{E\h^\Om\trkla{\phi\h\no}}=r\h\nn-r\h\no\\
&\,+\frac{1}{4\sqrt{E\h^\Om\trkla{\phi\h\no}}}\iO\Ih{F^{\prime\prime}\trkla{\phi\h\no}\trkla{\phi\h\nn-\phi\h\no}\tau\Delta\h\mu\h\nn}\dx\\
&\,+\frac{1}{4\sqrt{E\h^\Om\trkla{\phi\h\no}}}\iO\Ih{\trkla{F^{\prime\prime}\trkla{\varphi_2}-F^{\prime\prime}\trkla{\phi\h\no}}\trkla{\phi\h\nn-\phi\h\no}^2}\dx\\
&\,-\frac{1}{8\tekla{E\h^\Om\trkla{\phi\h\no}}^{3/2}}\rkla{\iO\Ih{F^\prime\trkla{\phi\h\no}\trkla{\phi\h\nn-\phi\h\no}}\dx}\\
&\qquad\qquad\times\rkla{\iO\Ih{F^\prime\trkla{\phi\h\no}\tau\Delta\h\mu\h\nn}\dx}\\
&\,-\frac{1}{8\tekla{E\h^\Om\trkla{\phi\h\no}}^{3/2}}\rkla{\iO\Ih{F^\prime\trkla{\phi\h\no}\trkla{\phi\h\nn-\phi\h\no}}\dx}\\
&\qquad\qquad\times\rkla{\iO\Ih{F^{\prime\prime}\trkla{\varphi_3}\trkla{\phi\h\nn-\phi\h\no}^2}\dx}\\
&-\frac{1}{8\tekla{E\h^\Om\trkla{\phi\h\no}}^{3/2}}\rkla{\tfrac12\iO\Ih{F^{\prime\prime}\trkla{\varphi_3}\trkla{\phi\h\nn-\phi\h\no}^2}\dx}^2\\
&+\frac{1}{16\tekla{E\h^\Om\trkla{\varphi_1}}^{5/2}}\rkla{\iO\Ih{F^\prime\trkla{\varphi_4}\trkla{\phi\h\nn-\phi\h\no}}\dx}^3\\
=:&\,r\h\nn-r\h\no+R_{1,n}+R_{2,n}+R_{3,n}+R_{4,n}+R_{5,n}+R_{6,n}\,.
\end{split}
\end{align}
Summing \eqref{eq:taylor2} from $n=1$ to $m\leq N$, noting that by definition $r\h^0=\sqrt{E\h^\Om\trkla{\phi\h^0}}$, taking the $p$-th power and the supremum over all $m\in\tgkla{0,\ldots,N}$, and computing the expected value, we obtain
\begin{multline}
\expected{\sup_{0\leq m\leq N}\abs{r\h^m-\sqrt{E\h^\Om\trkla{\phi\h\nn}}}^p}\leq C\expected{\abs{\sum_{n=1}^NR_{1,n}}^p}+C\expected{\abs{\sum_{n=1}^NR_{2,n}}^p}+C\expected{\abs{\sum_{n=1}^NR_{3,n}}^p}\\
+C\expected{\abs{\sum_{n=1}^NR_{4,n}}^p}+C\expected{\abs{\sum_{n=1}^NR_{5,n}}^p}+C\expected{\abs{\sum_{n=1}^NR_{6,n}}^p}\,.
\end{multline}
Deriving estimates for this right-hand side is more intricate  than for the one discussed in \cite{\citeASAV} as uniform bounds for $\Delta\h\mu\h\nn$ are only available in the $\trkla{H^1\trkla{\Om}}^\prime$-norm.
It is, however, possible to derive a $\tau$-dependent bound that provides sufficient regularity.
Choosing $\psi\h\equiv\tau^{1/2}\Delta\h\mu\h\nn$ in \eqref{eq:modeldisc:phiO}, we obtain
\begin{align}
\begin{split}\label{eq:deltamul2}
\tau^{3/2}\norm{\Delta\h\mu\h\nn}_{h,\Om}^2=&\, \tau^{1/2}\iO\Ih{\trkla{\phi\h\nn-\phi\h\no}\Delta\h\mu\h\nn}\dx-\tau^{1/2}\iO\Ih{\Phi\h\trkla{\phi\h\no}\sinc{n}\Delta\h\mu\h\nn}\dx\\
\leq&\,\tau\norm{\nabla\mu\h\nn}_{L^2\trkla{\Om}}^2+\tfrac12\norm{\nabla\phi\h\nn-\nabla\phi\h\no}_{L^2\trkla{\Om}}^2+\tfrac12\norm{\nabla\rkla{\Phi\h\trkla{\phi\h\no}\sinc{n}}}_{L^2\trkla{\Om}}^2\,.
\end{split}
\end{align}
Recalling the computation in \eqref{eq:tmp:R1} and the regularity results from Lemma \ref{lem:energy}
\begin{multline}\label{eq:Deltahmu}
\expected{\rkla{\sum_{n=1}^N\tau^{3/2}\norm{\Delta\h\mu\h\nn}_{L^2\trkla{\Om}}^2}^p}\leq C\expected{\rkla{\sum_{n=1}^N\tau\norm{\nabla\mu\h\nn}_{L^2\trkla{\Om}}^2}^p}\\
+ C\expected{\rkla{\sum_{n=1}^N\norm{\nabla\phi\h\nn-\nabla\phi\h\no}_{L^2\trkla{\Om}}^2}^p} +C\expected{\rkla{\sum_{n=1}^N\norm{\nabla\rkla{\Phi\h\trkla{\phi\h\no}\sinc{n}}}_{L^2\trkla{\Om}}^2}^p}\leq C
\end{multline}
for any $p\geq1$.
Using this estimate, we compute an upper bound for $R_{1,n}$ by applying Hölder's inequality, \ref{item:potentials}, Young's inequality, and Lemma \ref{lem:energy}
\begin{align}
\begin{split}
&\expected{\abs{\sum_{n=1}^N R_{1,n}}^p}\leq C \expected{\abs{\sum_{n=1}^N\tau\norm{\Delta\h\mu\h\nn}_{L^2\trkla{\Om}}\norm{\Ih{F^{\prime\prime}\trkla{\phi\h\no}}}_{L^3\trkla{\Om}}\norm{\phi\h\nn-\phi\h\no}_{L^6\trkla{\Om}}}^p}\\
&\,\leq C\expected{\abs{\rkla{1+\max_{0\leq n\leq N}\norm{\phi\h\nn}_{H^1\trkla{\Om}}^2}\tau^{1/4}\rkla{\sum_{n=1}^N\tau^{3/2}\norm{\Delta\h\mu\h\nn}_{L^2\trkla{\Om}}^2+\sum_{n=1}^N\norm{\phi\h\nn-\phi\h\no}_{H^1\trkla{\Om}}^2}}^p}\\
&\,\leq C\tau^{p/4}\,.
\end{split}
\end{align}
Concerning the second term $R_{2,n}$ we start by discussing the $C^{2,\nu}$-part $F_1$ of the potential $F$ (cf.~\ref{item:potentials}).
Due to the Gagliardo--Nirenberg inequality, we obtain
\begin{multline}
\abs{\iO\Ih{\rkla{F_1^{\prime\prime}\trkla{\varphi_2}-F_1^{\prime\prime}\trkla{\phi\h\no}}\trkla{\phi\h\nn-\phi\h\no}^2}\dx}\leq C\iO\Ih{\abs{\phi\h\nn-\phi\h\no}^{2+\nu}}\dx\\
\leq C\norm{\phi\h\nn-\phi\h\no}_{H^1\trkla{\Om}}^{3\nu/2}\norm{\phi\h\nn-\phi\h\no}_{L^2\trkla{\Om}}^{\trkla{4-\nu}/2}=:\trkla{*}\,.
\end{multline}
If $\nu\leq 0.8$, we may apply Young's inequality to deduce
\begin{align}
\trkla{*}\leq C\tau^{\nu/2}\norm{\phi\h\nn-\phi\h\no}_{H^1\trkla{\Om}}^2+C\tau^{-3\nu^2/\trkla{8-6\nu}}\norm{\phi\h\nn-\phi\h\no}_{L^2\trkla{\Om}}^{\trkla{8-2\nu}/\trkla{4-3\nu}}\,,
\end{align}
where we will control the second term using \eqref{eq:nikolskii:2} as $\trkla{8-2\nu}/\trkla{4-3\nu}\leq 4$.
Unfortunately, the limited regularity of the phase-field parameter reduces gains for larger $\nu$. 
Yet, we still can rely on the discrete $L^\infty\trkla{0,T;H^1\trkla{\Om}}$-bounds for $\trkla{\phi\h\nn}_n$ and recover the scaling obtained for $\nu=0.8$. 
In particular, we estimate
\begin{align}
\trkla{*}\leq C\norm{\phi\h\nn-\phi\h\no}_{H^1\trkla{\Om}}^{\trkla{5\nu-4}/4}\rkla{\tau^{4/10}\norm{\phi\h\nn-\phi\h\no}_{H^1\trkla{\Om}}^2+\tau^{-6/10}\norm{\phi\h\nn-\phi\h\no}_{L^2\trkla{\Om}}^4}\,.
\end{align}
To deal with the remaining part, we recall the growth condition for $F_2^{\prime\prime\prime}$ and apply the Gagliardo-Nirenberg inequality to deduce
\begin{align}
\begin{split}
&\abs{\iO\Ih{\trkla{F_2^{\prime\prime}\trkla{\varphi_2}-F_2^{\prime\prime}\trkla{\phi\h\no}}\trkla{\phi\h\nn-\phi\h\no}^2}\dx}\\
&\qquad\leq C\rkla{1+\norm{\phi\h\nn}_{H^1\trkla{\Om}}^2+\norm{\phi\h\no}_{H^1\trkla{\Om}}^2}\norm{\phi\h\nn-\phi\h\no}_{H^1\trkla{\Om}}^{5/2}\norm{\phi\h\nn-\phi\h\no}_{L^2\trkla{\Om}}^{1/2}\\
&\qquad\leq C\rkla{1+\norm{\phi\h\nn}_{H^1\trkla{\Om}}^{11/4}+\norm{\phi\h\no}_{H^1\trkla{\Om}}^{11/4}}\norm{\phi\h\nn-\phi\h\no}_{H^1\trkla{\Om}}^{7/4}\norm{\phi\h\nn-\phi\h\no}_{L^2\trkla{\Om}}^{1/2}\\
&\qquad \leq C\rkla{1+\norm{\phi\h\nn}_{H^1\trkla{\Om}}^{11/4}+\norm{\phi\h\no}_{H^1\trkla{\Om}}^{11/4}}\rkla{\tau^{1/8}\norm{\phi\h\nn-\phi\h\no}_{H^1\trkla{\Om}}^2+\tau^{-7/8}\norm{\phi\h\nn-\phi\h\no}_{L^2\rkla{\Om}}^4}\,.
\end{split}
\end{align}
Combining the above results with the uniform lower bound on $E\h^\Om\trkla{.}$, Hölder's inequality, Lemma \ref{lem:energy}, and Lemma \ref{lem:nikolskii}, we compute
\begin{align}
\expected{\abs{\sum_{n=1}^N R_{2,n}}^p}\leq C\tau^{p\min\tgkla{\nu/2,1/8}}\,.
\end{align}
Combining Hölder's inequality, \ref{item:potentials}, Young's inequality, Lemma \ref{lem:energy}, and \eqref{eq:Deltahmu} provides
\begin{align}
\begin{split}
&\expected{\abs{\sum_{n=1}^N R_{3,n}}^p}\leq C\expected{\rkla{\sum_{n=1}^N\tau\norm{\Ih{F^\prime\trkla{\phi\h\no}}}_{L^2\trkla{\Om}}^2\norm{\phi\h\nn-\phi\h\no}_{L^2\trkla{\Om}}\norm{\Delta\h\mu\h\nn}_{L^2\trkla{\Om}}}^p}\\
&~\leq C\expected{\rkla{1+\max_{1\leq n\leq N}\norm{\phi\h\no}_{H^1\trkla{\Om}}^6}^p\tau^{p/4}\rkla{\sum_{n=1}^N\rkla{\tau^{3/2}\norm{\Delta\h\mu\h\nn}_{L^2\trkla{\Om}}^2+\norm{\phi\h\nn-\phi\h\no}_{L^2\trkla{\Om}}^2}}^p}\\
&~\leq C\tau^{p/4}\,.
\end{split}
\end{align}
The remaining terms $R_{4,n}$, $R_{5,n}$, and $R_{6,n}$ can be treated similarly to $R_{2,n}$.
In particular, applying the Gagliardo--Nirenberg inequality, \ref{item:potentials}

\begin{align}
\begin{split}
\expected{\abs{\sum_{n=1}^N R_{4,n}}^p}\leq&\,C\expected{\rkla{\rkla{1+\max_{0\leq n\leq N}\norm{\phi\h\nn}_{H^1\trkla{\Om}}^5}\sum_{n=1}^N\tau^{1/2}\norm{\phi\h\nn-\phi\h\no}_{H^1\trkla{\Om}}^2}^p}\\
&+C\expected{\rkla{\rkla{1+\max_{0\leq n\leq N}\norm{\phi\h\nn}_{H^1\trkla{\Om}}^5}\sum_{n=1}^N\tau^{-1/2}\norm{\phi\h\nn-\phi\h\no}_{L^2\trkla{\Om}}^4}^p}\\
\leq &\, C\tau^{p/2}\,,
\end{split}\\
\begin{split}
\expected{\abs{\sum_{n=1}^N R_{5,n}}^p}\leq&\,C\expected{\rkla{\rkla{1+\max_{0\leq n\leq N}\norm{\phi\h\nn}_{H^1\trkla{\Om}}^5}\sum_{n=1}^N\tau^{1/2}\norm{\phi\h\nn-\phi\h\no}_{H^1\trkla{\Om}}^2}^p}\\
&+C\expected{\rkla{\rkla{1+\max_{0\leq n\leq N}\norm{\phi\h\nn}_{H^1\trkla{\Om}}^5}\sum_{n=1}^N\tau^{-1/2}\norm{\phi\h\nn-\phi\h\no}_{L^2\trkla{\Om}}^4}^p}\\
\leq & \,C\tau^{p/2}\,,
\end{split}\\
\begin{split}
\expected{\abs{\sum_{n=1}^N R_{6,n}}^p}\leq&\,C\expected{\rkla{\max_{0\leq n\leq N}\rkla{1+\norm{\phi\h\nn}_{H^1\trkla{\Om}}^9}\sum_{n=1}^N\norm{\phi\h\nn-\phi\h\no}_{L^2\trkla{\Om}}^3}^p}\leq C\tau^{p/2}\,.
\end{split}
\end{align}
Combining the above estimates provides \eqref{eq:errorsavO}.
To obtain \eqref{eq:errorsavG}, we follow a similar pathway and obtain
\begin{align}
\begin{split}
\sqrt{E\h^\Gamma\trkla{\phi\h\nn}}&-\sqrt{E\h^\Gamma\trkla{\phi\h\no}}=s\h\nn-s\h\no\\
&-\frac{1}{4\sqrt{E\h^\Gamma\trkla{\phi\h\no}}}\iGamma\IhG{G^{\prime\prime}\trkla{\phi\h\no}\trkla{\phi\h\nn-\phi\h\no}\tau\theta\h\nn}\dG\\
&+\frac{1}{4\sqrt{E\h^\Gamma\trkla{\phi\h\no}}}\iGamma\IhG{\rkla{G^{\prime\prime}\trkla{\widehat{\varphi}_2}-G^{\prime\prime}\trkla{\phi\h\no}}\trkla{\phi\h\nn-\phi\h\no}^2}\dG\\
&+\frac{1}{8\tekla{E\h^\Gamma\trkla{\phi\h\no}}^{3/2}}\rkla{\iGamma\IhG{G^\prime\trkla{\phi\h\no}\trkla{\phi\h\nn-\phi\h\no}}\dG}\\
&\qquad\times\rkla{\iGamma\IhG{G^{\prime}\trkla{\phi\h\no}\tau\theta\h\nn}\dG}\\
&-\frac{1}{8\tekla{E\h^\Gamma\trkla{\phi\h\no}}^{3/2}}\rkla{\iGamma\IhG{G^\prime\trkla{\phi\h\no}\trkla{\phi\h\nn-\phi\h\no}}\dG}\\
&\qquad\times\rkla{\iGamma\IhG{G^{\prime\prime}\trkla{\widehat{\varphi}_3}\trkla{\phi\h\nn-\phi\h\no}^2}\dG}\\
&-\frac{1}{8\tekla{E\h^\Gamma\trkla{\phi\h\no}}^{3/2}}\rkla{\tfrac12\iGamma\IhG{G^{\prime\prime}\trkla{\widehat{\varphi}_3}\trkla{\phi\h\nn-\phi\h\no}^2}\dG}^2\\
&+\frac{1}{16    \tekla{E\h^\Gamma\trkla{\widehat{\varphi}_1}}^{5/2}}\rkla{\iGamma\IhG{G^\prime\trkla{\widehat{\varphi}_4}\trkla{\phi\h\nn-\phi\h\no}}\dG}^3\\
=:&\,s\h\nn-s\h\no+\widehat{R}_{1,n}+\widehat{R}_{2,n}+\widehat{R}_{3,n}+\widehat{R}_{4,n}+\widehat{R}_{5,n}+\widehat{R}_{6,n}
\end{split}
\end{align}
with $\widehat{\varphi}_1,\widehat{\varphi}_2,\widehat{\varphi}_3,\widehat{\varphi}_4\in\mathrm{conv}\tgkla{\phi\h\nn,\phi\h\no}$.
The error terms $\widehat{R}_{1,n},\ldots,\widehat{R}_{6,n}$ can in principle be estimated similarly to $R_{1,n},\ldots,R_{6,n}$.
On the boundary, however, this requires only control over the potential $\theta\h\nn$ instead of its discrete Laplacian.
Furthermore, the Gagliardo--Nirenberg inequality in $\trkla{d-1}$ dimensions allows for better estimates.
In particular, we obtain for the first term
\begin{multline}
\expected{\abs{\sum_{n=1}^N \widehat{R}_{1,n}}^p}\\
\leq C\expected{\abs{\rkla{1+\max_{0\leq n\leq N}\norm{\phi\h\nn}_{H^1\trkla{\Gamma}}^2}\tau^{1/2}\sum_{n=1}^N\rkla{\tau\norm{\theta\h\nn}_{L^2\trkla{\Gamma}}^2+\norm{\phi\h\nn-\phi\h\no}_{H^1\trkla{\Gamma}}^2}}^p} \leq C\tau^{p/2}\,.
\end{multline}
To estimate the second term, we again use the decomposition of the potential $G$.
Starting with the $C^{2,\nu}$-part, we compute by using the Gagliardo--Nirenberg inequality
\begin{multline}
\abs{\iGamma\IhG{\rkla{G_1^{\prime\prime}\trkla{\widehat{\varphi}_2}-G_1^{\prime\prime}\trkla{\phi\h\no}}\trkla{\phi\h\nn-\phi\h\no}^2}\dG}\leq C\iGamma\IhG{\abs{\phi\h\nn-\phi\h\no}^{2+\nu}}\dG\\
\leq C\norm{\phi\h\nn-\phi\h\no}_{H^1\trkla{\Gamma}}^\nu\norm{\phi\h\nn-\phi\h\no}_{L^2\trkla{\Gamma}}^2\\
\leq C\rkla{\tau^{\nu/2}\norm{\phi\h\nn-\phi\h\no}_{H^1\trkla{\Gamma}}^2+\tau^{-\nu^2/\trkla{4-2\nu}}\norm{\phi\h\nn-\phi\h\no}_{L^2\trkla{\Gamma}}^{4/\trkla{2-\nu}}}\,.
\end{multline}
Concerning the $C^3$-part of $G$, we deduce
\begin{multline}
\abs{\iGamma\IhG{\rkla{G_2^{\prime\prime}\trkla{\widehat{\varphi}_2}-G_2^{\prime\prime}\trkla{\phi\h\no}}\trkla{\phi\h\nn-\phi\h\no}^2}\dG}\\
\leq C\rkla{1+\norm{\phi\h\nn}_{H^1\trkla{\Gamma}}^2+\norm{\phi\h\no}_{H^1\trkla{\Gamma}}^2}\norm{\phi\h\nn-\phi\h\no}_{L^{3+\delta}\trkla{\Gamma}}^3\\
\leq C\rkla{1+\norm{\phi\h\nn}_{H^1\trkla{\Gamma}}^2+\norm{\phi\h\no}_{H^1\trkla{\Gamma}}^2}\norm{\phi\h\nn-\phi\h\no}_{H^1\trkla{\Gamma}}^{\trkla{3+3\delta}/\trkla{3+\delta}}\norm{\phi\h\nn-\phi\h\no}_{L^2\trkla{\Gamma}}^{6/\trkla{3+\delta}}\\
\leq C\rkla{1+\norm{\phi\h\nn}_{H^1\trkla{\Gamma}}^2+\norm{\phi\h\no}_{H^1\trkla{\Gamma}}^2}\rkla{\tau^{1/2}\norm{\phi\h\nn-\phi\h\no}_{H^1\trkla{\Gamma}}^2+\tau^{-\tfrac{3+3\delta}{6-2\delta}}\norm{\phi\h\nn-\phi\h\no}_{L^2\trkla{\Gamma}}^{12/\trkla{3-\delta}}}
\end{multline}
for arbitrary $0<\delta<1$.
Using \eqref{eq:nikolskii:3}, we obtain
\begin{align}
\expected{\abs{\sum_{n=1}^N \widehat{R}_{2,n}}^p}\leq C\tau^{\nu p/2}
\end{align}
as $\nu\leq 1$.
Furthermore, we have
\begin{align}
\begin{split}
\expected{\abs{\sum_{n=1}^N\widehat{R}_{3,n}}^p}\leq&\, C\expected{\abs{\sum_{n=1}^N\tau\norm{\IhG{G^\prime\trkla{\phi\h\no}}}_{L^2\trkla{\Gamma}}^2\norm{\phi\h\nn-\phi\h\no}_{L^2\trkla{\Gamma}}\norm{\theta\h\nn}_{L^2\trkla{\Gamma}}}^p}\\
\leq&\,C\tau^{p/2}\,,
\end{split}\\
\begin{split}
\expected{\abs{\sum_{n=1}^N\widehat{R}_{4,n}}^p}\leq&\, C\tau^{p/2}\,,
\end{split}\\
\begin{split}
\expected{\abs{\sum_{n=1}^N\widehat{R}_{5,n}}^p}\leq&\,C\tau^{p/2}\,,
\end{split}\\
\begin{split}
\expected{\abs{\sum_{n=1}^N\widehat{R}_{6,n}}^p}\leq&\,C\tau^{p/2}\,,
\end{split}
\end{align}
which concludes the proof.
\end{proof}
In the proof of Lemma \ref{lem:errorsav}, we relied on a bound on $\tau^{3/4}\norm{\Delta\h \mu\h\nn}_{h,\Om}^2$ (cf.~\eqref{eq:deltamul2}).
When passing to the limit $\tau\searrow0$, this bound loses its significance.
Using a weaker spatial norm, however, allows to obtain $\tau$-independent bounds for the weak form of the Laplacian of $\mu\h\nn$.
Introducing the linear form $l\tekla{\mu\h\nn}\,:\,H^1\trkla{\Om}\rightarrow \mathds{R}$ given via
\begin{align}\label{eq:def:weaklaplace}
\psi\mapsto l\tekla{\mu\h\nn}\trkla{\psi}:=\iO\nabla\mu\h\nn\cdot\nabla\psi\dx
\end{align}
for all $\psi\in H^1\trkla{\Om}$ a straightforward computation provides the following result:
\begin{corollary}\label{cor:laplacemu}
Let the assumptions \ref{item:time}, \ref{item:space1}, \ref{item:space2}, \ref{item:potentials}, \ref{item:initial}, \ref{item:rho}, and \ref{item:filtration}-\ref{item:color} hold true.
Then, for every $1\leq p<\infty$, there exists a constant $C\equiv C\trkla{p,T}>0$ independent of $h$ and $\tau$ such that the linear form $l\tekla{\mu\h\nn}\,:\,H^1\trkla{\Om}\rightarrow \mathds{R}$ given in \eqref{eq:def:weaklaplace}
satisfies
\begin{align*}
\expected{\rkla{\sum_{n=1}^N\tau\norm{l\tekla{\mu\h\nn}}_{\trkla{H^1\trkla{\Om}}^\prime}^2}^p}\leq C\,.
\end{align*}
\end{corollary}

\section{Compactness properties of discrete solutions}\label{sec:compactness}
In this section, we shall establish the tightness of the laws of the discrete solutions and deduce the existence of weakly and strongly converging subsequences.

Using the time-index free notation defined in \eqref{eq:timeinterpol}, we restate the regularity results from the last section as follows:
\begin{subequations}\label{eq:regularity:timecont}
\begin{align}
\begin{split}\label{eq:regularity:timecont:1}
\norm{\phi\h\tpm}_{L^{2\p}\trkla{\Omega;L^\infty\trkla{0,T;H^1\trkla{\Om}}}}+\norm{\phi\h\tpm}_{L^{2\p}\trkla{\Omega;L^\infty\trkla{0,T;H^1\trkla{\Gamma}}}} +\norm{r\h\tpm}_{L^{2\p}\trkla{\Omega;L^\infty\trkla{0,T}}} \\
+\norm{s\h\tpm}_{L^{2\p}\trkla{\Omega;L^\infty\trkla{0,T}}} +\norm{\mu\h\tp}_{L^{2\p}\trkla{\Omega;L^2\trkla{0,T;H^1\trkla{\Om}}}} +\norm{\theta\h\tp}_{L^{2\p}\trkla{\Omega;L^2\trkla{0,T;L^2\trkla{\Gamma}}}}\\
+\tau^{-1/2}\norm{\phi\h\tp-\phi\h\tm}_{L^{2\p}\trkla{\Omega;L^2\trkla{0,T;H^1\trkla{\Omega}}}}+\tau^{-1/2}\norm{\phi\h\tp-\phi\h\tm}_{L^{2\p}\trkla{\Omega;L^2\trkla{0,T;H^1\trkla{\Gamma}}}}\\
+\tau^{-1/2}\norm{r\h\tp-r\h\tm}_{L^{2\p}\trkla{\Omega;L^2\trkla{0,T}}}+\tau^{-1/2}\norm{s\h\tp-s\h\tm}_{L^{2\p}\trkla{\Omega;L^2\trkla{0,T}}}\\
+\norm{l\tekla{\mu\h\tp}}_{L^{2\p}\trkla{\Omega;L^2\trkla{0,T;\trkla{H^1\trkla{\Om}}^\prime}}}&\leq C\,,
\end{split}
\end{align}
\begin{align}\label{eq:regularity:timecont:1.5}
\norm{\Xi_{h,\Om}\tp}_{L^\p\trkla{\Omega;L^2\trkla{0,T;L^2\trkla{\Om}}}}+\norm{\Xi_{h,\Gamma}\tp}_{L^\p\trkla{\Omega;L^2\trkla{0,T;L^2\trkla{\Gamma}}}}\leq C \tau^{1/2}\,,
\end{align}
\begin{align}\label{eq:regularity:timecont:2}
\norm{\phi\h\tl}_{L^{2\alpha}\trkla{\Omega;N^{1/4,2\alpha}\trkla{0,T;L^2\trkla{\Om}}}}+\norm{\phi\h\tl}_{L^{2\alpha}\trkla{\Omega;N^{1/2,2\alpha}\trkla{0,T;L^2\trkla{\Gamma}}}}\leq C\,,
\end{align}
\begin{align}\label{eq:regularity:timecont:3}
\norm{\phi\h\tl}_{L^{4\alpha}\trkla{\Omega;C^{0,\trkla{\alpha-1}/\trkla{4\alpha}}\trkla{\tekla{0,T};L^2\trkla{\Om}}}}+\norm{\phi\h\tl}_{L^{2\alpha}\trkla{\Omega;C^{0,\trkla{\alpha-1}/\trkla{2\alpha}}\trkla{\tekla{0,T};L^2\trkla{\Gamma}}}}\leq C\,,
\end{align}
\end{subequations}
for $\p\in[1,\infty)$ and $\alpha\in\trkla{1,\infty}$.
While \eqref{eq:regularity:timecont:1} is a direct consequence of Lemma \ref{lem:energy}, Lemma \ref{lem:potential}, and Corollary \ref{cor:laplacemu}, \eqref{eq:regularity:timecont:1.5} follows from Lemma \ref{lem:potential} and \eqref{eq:normequivalence}, \eqref{eq:regularity:timecont:2} follows from the results of Lemma \ref{lem:nikolskii} and Lemma 3.2 in \cite{Banas2013}.
The last inequality \eqref{eq:regularity:timecont:3} then follows from the embedding results in \cite{Simon1990}.\\
In the next step, we want to identify almost surely converging subsequences by applying Jakubowski's theorem (cf.~\cite{Jakubowski1998}).
In particular, we are interested in the convergence properties of $\trkla{\phi\h\tl,\trace{\phi\h\tl},r\h\tl,s\h\tl,\mu\h\tp,l\tekla{\mu\h\tp},\theta\h\tp}_{h,\tau}$ and the linear interpolation $\bs{\xi}\h\tl$ of $\tgkla{\bs{\xi}\h^{m,\tau}}_{m}$.
Although the time-continuous processes $\bs{\xi}\h\tl$ are not martingales, we will later show that they converge towards martingales.\\
We start by establishing uniform estimates for $\bs{\xi}\h\tl$:
\begin{lemma}\label{lem:boundxi}
Let the assumptions \ref{item:filtration}-\ref{item:color} hold true.
Then, the piecewise linear process $\bs{\xi}\h\tl$ satisfies
\begin{align}
\norm{\bs{\xi}\h\tl}_{L^{2p}\trkla{\Omega;C^{0,\trkla{p-1}/\trkla{2p}}\trkla{\tekla{0,T};H^2\trkla{\Om}}}}\leq C
\end{align}
for arbitrary $p\in\trkla{1,\infty}$ with a constant $C>0$ independent of $h$ and $\tau$.
\end{lemma}
\begin{proof}
Recalling the definition of $\bs{\xi}\h^{m,\tau}$ (cf.~\eqref{eq:defxih}) and Lemma \ref{lem:BDG}, we obtain
\begin{align}
\begin{split}
\expected{\sum_{m=0}^{N-l}\tau\norm{\bs{\xi}\h^{m+l,\tau}-\bs{\xi}\h^{m,\tau}}_{H^2\trkla{\Om}}^{2p}}=\expected{\sum_{m=0}^{N-l}\tau\norm{\sum_{n=m+1}^{m+l}\sqrt{\tau}\sum_{k\in\Zh}\lambda_k\g{k}\xi\h^{n,\tau} }_{H^2\trkla{\Om}}^{2p}}\\
\leq C\sum_{m=0}^{N-l}\tau\trkla{l\tau}^{p-1}\sum_{n=m+1}^{m+l}\tau\expected{\rkla{\sum_{k\in\Zh}\norm{\lambda_k\g{k}}_{H^2\trkla{\Om}}^2}^p}\leq C\trkla{l\tau}^p\,,
\end{split}
\end{align}
which is sufficient to establish the result (cf.~\cite[Lemma 3.2]{Banas2013}). 
\end{proof}
Lemma \ref{lem:boundxi} in particular provides the tightness of the laws of $\trkla{\bs{\xi}\h\tl}_{h,\tau}$ in $C\trkla{\tekla{0,T};H^1\trkla{\Om}}$.
With the next lemma, we establish the tightness of the laws of $\trkla{\phi\h\tl}_{h,\tau}$ and $\trkla{\trace{\phi\h\tl}}_{h,\tau}$ in $C\trkla{\tekla{0,T};L^s\trkla{\Om}}$ and $C\trkla{\tekla{0,T};L^r\trkla{\Gamma}}$ with $s\in[1,\tfrac{2d}{d-2})$ and $r\in[1,\infty)$:

\begin{lemma}\label{lem:tight}
Let $\trkla{\phi\h\tl,\trace{\phi\h\tl}}_{h,\tau}$ be a family of continuous, piecewise linear processes that satisfy the bounds stated in \eqref{eq:regularity:timecont}.
Then the family of laws $\trkla{\nu_{\phi\h\tl}}_{h,\tau}$ generated by $\trkla{\phi\h\tl}_{h,\tau}$ is tight on $C\trkla{\tekla{0,T};L^s\trkla{\Om}}$ ($s\in[1,\tfrac{2d}{d-2})$) and the family of laws $\trkla{\nu_{\trace{\phi\h\tl}}}_{h,\tau}$ generated by $\trkla{\trace{\phi\h\tl}}_{h,\tau}$ is tight on $C\trkla{\tekla{0,T};L^r\trkla{\Gamma}}$ with $r\in[1,\infty)$.
\end{lemma}
\begin{proof}
Recalling the compactness theorem by Simon (cf.~\cite{Simon1987}) we note that the closed ball $\overline{B}^\Om_R$ in $L^\infty\trkla{0,T;H^1\trkla{\Om}}\cap C^{0,\trkla{\alpha-1}/\trkla{4\alpha}}\trkla{\tekla{0,T};L^2\trkla{\Om}}$ is a compact subset of $C\trkla{\tekla{0,T};L^s\trkla{\Om}}$. 
Furthermore, the family of laws $\rkla{\nu_{\phi\h\tl}}_{h,\tau}$ satisfies for any $R>0$
\begin{multline}
\nu_{\phi\h\tl}\trkla{C\trkla{\tekla{0,T};L^s\trkla{\Om}}\setminus\overline{B}_R^\Om} =\Prob\ekla{\norm{\phi\h\tl}_{L^\infty\trkla{0,T;H^1\trkla{\Om}}}^{4\alpha}+\norm{\phi\h\tl}_{C^{0,\trkla{\alpha-1}/\trkla{4\alpha}}\trkla{\tekla{0,T};L^2\trkla{\Om}}}^{4\alpha} >R^{4\alpha}}\\
\leq R^{-4\alpha}\expected{\norm{\phi\h\tl}_{L^\infty\trkla{0,T;H^1\trkla{\Om}}}^{4\alpha}+\norm{\phi\h\tl}_{C^{0,\trkla{\alpha-1}/\trkla{4\alpha}}\trkla{\tekla{0,T};L^2\trkla{\Om}}}^{4\alpha}}\,.
\end{multline}
The tightness of the family of laws $\rkla{\nu_{\trace{\phi\h\tl}}}_{h,\tau}$ on $C\trkla{\tekla{0,T};L^r\trkla{\Gamma}}$ follows by similar arguments.
\end{proof}
As the closed balls in $L^2\trkla{0,T}$, $L^2\trkla{0,T;H^1\trkla{\Om}}$, $L^2\trkla{0,T;\trkla{H^1\trkla{\Om}}^\prime}$, and $L^2\trkla{0,T;L^2\trkla{\Gamma}}$ are compact in the weak topology, the laws $\trkla{\nu_{r\h\tl}}_{h,\tau}$, $\trkla{\nu_{s\h\tl}}_{h,\tau}$, $\trkla{\nu_{\mu\h\tp}}_{h,\tau}$, $\trkla{\nu_{l\tekla{\mu\h\tp}}}_{h,\tau}$, and $\trkla{\nu_{\theta\h\tp}}_{h,\tau}$ generated by $\trkla{r\h\tl}_{h,\tau}$, $\trkla{s\h\tl}_{h,\tau}$, $\trkla{\mu\h\tp}_{h,\tau}$, $\trkla{l\tekla{\mu\h\tp}}_{h,\tau}$, and $\trkla{\theta\h\tp}_{h,\tau}$ are tight in $L^2\trkla{0,T}_{\weaktop}$, $L^2\trkla{0,T}_{\weaktop}$, $L^2\trkla{0,T;H^1\trkla{\Om}}_{\weaktop}$, $L^2\trkla{0,T;\trkla{H^1\trkla{\Om}}^\prime}_{\weaktop}$, and $L^2\trkla{0,T;L^2\trkla{\Gamma}}_{\weaktop}$ due to Markov's inequality and the bounds collected in \eqref{eq:regularity:timecont}.
Hence, the joint laws $\trkla{\nu\h\tl}_{h,\tau}$ of $\trkla{\phi\h\tl}_{h,\tau}$, $\trkla{\trace{\phi\h\tl}}_{h,\tau}$, $\trkla{r\h\tl}_{h,\tau}$, $\trkla{s\h\tl}_{h,\tau}$, $\trkla{\mu\h\tp}_{h,\tau}$, $\trkla{l\tekla{\mu\h\tp}}_{h,\tau}$, $\trkla{\theta\h\tp}_{h,\tau}$, $\trkla{\bs{\xi}\h\tl}_{h,\tau}$ are tight on the path space
\begin{align}
\begin{split}
\mathcal{X}:=&\, C\trkla{\tekla{0,T};L^s\trkla{\Om}}\times C\trkla{\tekla{0,T};L^r\trkla{\Gamma}}\times L^2\trkla{0,T}_{\weaktop}\times L^2\trkla{0,T}_{\weaktop}\\
&\quad\times L^2\trkla{0,T;H^1\trkla{\Om}}_{\weaktop}\times L^2\trkla{0,T;\trkla{H^1\trkla{\Om}}^\prime}_{\weaktop}\\
&\quad\times L^2\trkla{0,T;L^2\trkla{\Gamma}}_{\weaktop}\times C\trkla{\tekla{0,T};H^1\trkla{\Om}}\,,
\end{split}
\end{align}
$s\in[1,\tfrac{2d}{d-2})$ and $r\in[1,\infty)$.
Therefore, we can apply Jakubowski's generalization of the Skorokhod theorem (cf.~\cite{Jakubowski1998}) to obtain the following convergence result:
\begin{theorem}\label{thm:jakubowski}
Let $\rkla{\phi\h\tl,\trace{\phi\h\tl},r\h\tl,s\h\tl,\mu\h\tp,l\tekla{\mu\h\tp},\theta\h\tp}_{h,\tau}$ satisfy the estimates \eqref{eq:regularity:timecont} and let the family of discrete approximations $\trkla{\bs{\xi}\h\tl}_{h,\tau}$ of the $\mathcal{Q}$-Wiener process satisfy the bounds in Lemma \ref{lem:boundxi}.
Then there exists a subsequence
\begin{align*}
\rkla{\phi_j,\trace{\phi_j},r_j,s_j,\mu_j^+,l\tekla{\mu_j^+},\theta_j^+,\bs{\xi}_j}_{j}:=\rkla{\phi\hj^{\tau_j},\trace{\phi\hj^{\tau_j}},r\hj^{\tau_j},s\hj^{\tau_j},\mu\hj^{\tau_j,+},l\tekla{\mu\hj^{\tau_j,+}},\theta\hj^{\tau_j,+},\bs{\xi}\hj^{\tau_j}}_{j}\,,
\end{align*}
a stochastic basis $\trkla{\widetilde{\Omega},\widetilde{\mathcal{A}},\widetilde{\Prob}}$, a sequence of random variables
\begin{align*}
\rkla{\widetilde{\phi}_j,\trace{\widetilde{\phi}_j},\widetilde{r}_j,\widetilde{s}_j,\widetilde{\mu}_j^+,l\tekla{\widetilde{\mu}_j^+},\widetilde{\theta}_j^+,\widetilde{\bs{\xi}}_j}\,:\,\widetilde{\Omega}\rightarrow\mathcal{X}\,,
\end{align*}
and random variables
\begin{align*}
\rkla{\widetilde{\phi},\widetilde{\phi_\Gamma},\widetilde{r},\widetilde{s},\widetilde{\mu},\widetilde{L},\widetilde{\theta},\widetilde{W}}\,:\,\widetilde{\Omega}\rightarrow\mathcal{X}\,
\end{align*}
such that the following holds:
\begin{itemize}
\item The law of $\rkla{\widetilde{\phi}_j,\trace{\widetilde{\phi}_j},\widetilde{r}_j,\widetilde{s}_j,\widetilde{\mu}_j^+,l\tekla{\widetilde{\mu}_j^+},\widetilde{\theta}_j^+,\widetilde{\bs{\xi}}_j}$ on $\mathcal{X}$ under $\widetilde{\Prob}$ coincides with for any $j\in\mathds{N}$ with the law of $\rkla{\phi_j,\trace{\phi_j},r_j,s_j,\mu_j^+,l\tekla{\mu_j^+},\theta_j^+,\bs{\xi}_j}$ under $\Prob$.
\item The sequence $\rkla{\widetilde{\phi}_j,\trace{\widetilde{\phi}_j},\widetilde{r}_j,\widetilde{s}_j,\widetilde{\mu}_j^+,l\tekla{\widetilde{\mu}_j^+},\widetilde{\theta}_j^+,\widetilde{\bs{\xi}}_j}$ converges $\widetilde{\Prob}$-almost surely towards $\rkla{\widetilde{\phi},\widetilde{\phi_\Gamma},\widetilde{r},\widetilde{s},\widetilde{\mu},\widetilde{L},\widetilde{\theta},\widetilde{W}}$ in the topology of $\mathcal{X}$.
\end{itemize}
\end{theorem}
\begin{proof}
As the combined laws on $\mathcal{X}$ are tight, Jakuboswki's theorem provides the existence of a stochastic basis and a sequence of random variables $\rkla{\widetilde{\phi}_j,\widetilde{\phi_{\Gamma,j}},\widetilde{r}_j,\widetilde{s}_j,\widetilde{\mu}_j^+,\widetilde{L}_j^+,\widetilde{\theta}_j^+,\widetilde{\bs{\xi}}}$ on $\widetilde{\Omega}$ with the stated convergence properties.
Hence, it only remains to identify $\widetilde{\phi_{\Gamma,j}}$ and $\widetilde{L}_j^+$ with $\trace{\widetilde{\phi}_j}$ and $l\tekla{\widetilde{\mu}_j^+}$.
As the laws of $\widetilde{\phi}_j$ and $\phi_j$ coincide, $\widetilde{\phi}_j$ is $\widetilde{\Prob}$-almost surely a piecewise linear (time-)interpolation of an $\Uh$-valued random variable.
Hence, the trace of $\widetilde{\phi}_j$ is well-defined as a continuous mapping from $C\trkla{\overline{\Omega}}$ to $C\trkla{\Gamma}$.
Therefore, by identity of laws, we may identify $\widetilde{\phi_{\Gamma,j}}$ as $\trace{\widetilde{\phi}_j}$.
Similarly, we deduce that $\widetilde{\mu}_j^+$ and $\widetilde{L}_j^+$ are $\widetilde{\Prob}$-almost surely piecewise constant in time.
Using again the identity of laws, we can identify $\widetilde{L}_j\trkla{\psi}$ with $\iO\nabla\widetilde{\mu}_j^+\cdot\nabla\psi\dx=l\tekla{\widetilde{\mu}_j^+}\trkla{\psi}$ for all $\psi\in H^1\trkla{\Om}$ $\widetilde{\Prob}$-almost surely.
\end{proof}
\begin{remark}
The restriction to subsequences in Theorem \ref{thm:jakubowski} is necessary, as Jakubowski's theorem uses a generalization of Prokhorov's theorem to deduce convergence in distribution for a subsequence from uniform bounds.
For $\trkla{\bs{\xi}\h\tl}_{h,\tau}$, convergence in distribution can already be deduced from Donsker's invariance theorem.
\end{remark}
As the random variables introduced in Theorem \ref{thm:jakubowski} are $\widetilde{\Prob}$-almost surely continuous and piecewise linear (or left continuous and piecewise constant, respectively), we can evaluate them at the nodes of the time grid and recover the remaining time-interpolants. 
Analogously to \eqref{eq:deferrorterms}, we introduce the $\widetilde{\Prob}$-almost surely left continuous and piecewise constant in time finite element-valued random variables $\widetilde{\Xi}_{j,\Om}^+$ and $\widetilde{\Xi}_{j,\Gamma}^+$.
Due to the identity of laws, the constructed sequences satisfy the bounds collected in \eqref{eq:regularity:timecont} with respect to the new probability space $\trkla{\widetilde{\Omega},\widetilde{\mathcal{A}},\widetilde{\Prob}}$.
Hence, we obtain the following additional convergence properties:
\begin{lemma}\label{lem:convergence}
Let $\rkla{\widetilde{\phi}_j\pml}_{j\in\mathds{N}}$, $\rkla{\trace{\widetilde{\phi}_j\pml}}_{j\in\mathds{N}}$, $\rkla{\widetilde{r}_j\pml}_{j\in\mathds{N}}$, $\rkla{\widetilde{s}_j\pml}_{j\in\mathds{N}}$, $\rkla{\widetilde{\mu}_j^+}_{j\in\mathds{N}}$, $\rkla{l\tekla{\widetilde{\mu}_j^+}}_{j\in\mathds{N}}$, and $\rkla{\widetilde{\theta}_j^+}_{j\in\mathds{N}}$ be the sequences defined by interpolation of the sequences from Theorem \ref{thm:jakubowski}.
Furthermore, let the assumptions \ref{item:time}, \ref{item:space1}, \ref{item:space2}, \ref{item:potentials}, \ref{item:initial}, \ref{item:rho}, and \ref{item:filtration}-\ref{item:color} hold true. 
Then there exist functions
\begin{subequations}
\begin{align}
\begin{split}
\widetilde{\phi}\in&\,L^{2\p}_{\operatorname{weak-}(*)}\trkla{\widetilde{\Omega};L^\infty\trkla{0,T;H^1\trkla{\Om}}}\cap L^{4\p}\trkla{\widetilde{\Omega};C^{0,\trkla{\p-1}/\trkla{4\p}}\trkla{\tekla{0,T};L^2\trkla{\Om}}}\,,
\end{split}\\
\begin{split}
\widetilde{\phi_\Gamma}\in&\, L^{2\p}_{\operatorname{weak-}(*)}\trkla{\widetilde{\Omega};L^\infty\trkla{0,T;H^1\trkla{\Gamma}}}\cap L^{2\p}\trkla{\widetilde{\Omega};C^{0,\trkla{\p-1}/\trkla{2\p}}\trkla{\tekla{0,T};L^2\trkla{\Gamma}}}\,,
\end{split}\\
\widetilde{r}\in&\,L^{2\p}_{\operatorname{weak-}(*)}\trkla{\widetilde{\Omega};L^\infty\trkla{0,T}}\,,\\
\widetilde{s}\in&\,L^{2\p}_{\operatorname{weak-}(*)}\trkla{\widetilde{\Omega};L^\infty\trkla{0,T}}\,,\\
\begin{split}
\widetilde{\mu}\in &\, L^{2\p}\trkla{\widetilde{\Omega};L^2\trkla{0,T;H^1\trkla{\Om}}}\,,
\end{split}\\
\widetilde{L}\in&\,L^{2\p}\trkla{\widetilde{\Omega};L^2\trkla{0,T;\trkla{H^1\trkla{\Om}}^\prime}}\,,\\
\begin{split}
\widetilde{\theta}\in &\,L^{2\p}\trkla{\widetilde{\Omega};L^2\trkla{0,T;L^2\trkla{\Gamma}}}
\end{split}
\end{align}
\end{subequations}
for any $\p\in\trkla{1,\infty}$ such that $\trace{\widetilde{\phi}}=\widetilde{\phi_\Gamma}$ $\widetilde{\Prob}$-almost surely almost everywhere on $\trkla{0,T}\times\Gamma$, $\widetilde{r}=\sqrt{\iO F\rkla{\widetilde{\phi}}\dx}$ and $\widetilde{s}=\sqrt{\iGamma G\rkla{\widetilde{\phi_\Gamma}}\dG}$ $\widetilde{\Prob}$-almost surely almost everywhere on $\trkla{0,T}$, and $\widetilde{L}=l\tekla{\widetilde{\mu}}$ $\widetilde{\Prob}$-almost surely almost everywhere on $\trkla{0,T}\times\Om$.
Furthermore, for $j\rightarrow\infty$ the following convergence statements hold true along a subsequence:
\begin{subequations}
\begin{align}
\widetilde{\phi}_j&\rightarrow\widetilde{\phi}&&\text{in~}L^p\trkla{\widetilde{\Omega};C\trkla{\tekla{0,T};L^s\trkla{\Om}}}\,,\label{eq:conv:phiO:strong}\\
\widetilde{\phi}_j\pml&\rightarrow\widetilde{\phi}&&\text{in~}L^p\trkla{\widetilde{\Omega};L^p\trkla{0,T;L^s\trkla{\Om}}}\,,\label{eq:conv:phiO:strongLp}\\
\widetilde{\phi}_j\pml&\weakstar\widetilde{\phi}&&\text{in~}L^p_{\operatorname{weak-}(*)}\trkla{\widetilde{\Omega};L^\infty\trkla{0,T;H^1\trkla{\Om}}}\,,\label{eq:conv:phiO:weakstar}\\
\trace{\widetilde{\phi}_j}&\rightarrow \widetilde{\phi_\Gamma}&&\text{in~}L^p\trkla{\widetilde{\Omega};C\trkla{\tekla{0,T};L^r\trkla{\Gamma}}}\,,\label{eq:conv:phiG:strong}\\
\trace{\widetilde{\phi}_j\pml}&\rightarrow\widetilde{\phi_\Gamma}&&\text{in~} L^p\trkla{\widetilde{\Omega};L^p\trkla{0,T;L^r\trkla{\Gamma}}}\label{eq:conv:phiG:strongLp}\\
\trace{\widetilde{\phi}_j\pml}&\weakstar\widetilde{\phi_\Gamma}&&\text{in~}L^p_{\operatorname{weak-}(*)}\trkla{\widetilde{\Omega};L^\infty\trkla{0,T;H^1\trkla{\Gamma}}}\,,\label{eq:conv:phiG:weakstar}\\
\widetilde{\mu}_j^+&\weak \widetilde{\mu}&&\text{in~}L^p\trkla{\widetilde{\Omega};L^2\trkla{0,T;H^1\trkla{\Om}}}\,,\label{eq:conv:mu:weak}\\
\widetilde{\theta}_j^+&\weak\widetilde{\theta}&&\text{in~}L^p\trkla{\widetilde{\Omega};L^2\trkla{0,T;L^2\trkla{\Gamma}}}\,,\label{eq:conv:theta:weak}\\
\iO\Ihj{F\rkla{\widetilde{\phi}_j\pml}}\dx&\rightarrow\iO F\rkla{\widetilde{\phi}}\dx&&\text{in~}L^p\trkla{\widetilde{\Omega};L^p\trkla{0,T}}\,,\label{eq:conv:intF}\\
\widetilde{r}_j\pml&\rightarrow\widetilde{r}&&\text{in~} L^p\trkla{\widetilde{\Omega};L^p\trkla{0,T}}\,,\label{eq:conv:r:strong}\\
\widetilde{r}_j\pml&\weakstar\widetilde{r}&&\text{in~} L^p_{\operatorname{weak-}(*)}\trkla{\widetilde{\Omega};L^\infty\trkla{0,T}}\,,\label{eq:conv:r:weakstar}\\
\iGamma\IhGj{G\rkla{\trace{\widetilde{\phi}_j\pml}}}\dG&\rightarrow \iGamma G\rkla{\widetilde{\phi_\Gamma}}\dG&&\text{in~}L^p\trkla{\widetilde{\Omega},L^p\trkla{0,T}}\,,\label{eq:conv:intG}\\
\widetilde{s}_j\pml&\rightarrow\widetilde{s}&&\text{in~}L^p\trkla{\widetilde{\Omega};L^p\trkla{0,T}}\,,\label{eq:conv:s:strong}\\
\widetilde{s}_j\pml&\weakstar\widetilde{s}&&\text{in~}L^p_{\operatorname{weak-}(*)}\trkla{\widetilde{\Omega};L^\infty\trkla{0,T}}\label{eq:conv:s:weakstar}\,,\\
\widetilde{\Xi}_{j,\Om}^+&\rightarrow 0&&\text{in~} L^p\trkla{\widetilde{\Omega};L^2\trkla{0,T;L^2\trkla{\Om}}}\label{eq:conv:xiO}\,,\\
\widetilde{\Xi}_{j,\Gamma}^+&\rightarrow 0&&\text{in~} L^p\trkla{\widetilde{\Omega};L^2\trkla{0,T;L^2\trkla{\Gamma}}}\,\label{eq:conv:xiG}
\end{align}
\end{subequations}
for $p\in[1,\infty)$, $s\in[1,\tfrac{2d}{d-2})$, and $r\in[1,\infty)$.
\end{lemma}
\begin{proof}
Due to Theorem \ref{thm:jakubowski}, we have $\widetilde{\phi}_j\rightarrow\widetilde{\phi}$ in $C\trkla{\tekla{0,T};L^s\trkla{\Om}}$ $\widetilde{\Prob}$-almost surely.
By Vitali's convergence theorem and the bounds in \eqref{eq:regularity:timecont}, we obtain the strong convergence expressed in \eqref{eq:conv:phiO:strong}.
This strong convergence provides the existence of a subsequence $\widetilde{\phi_j}\rightarrow\widetilde{\phi}$ pointwise almost everywhere in $\widetilde{\Omega}\times\tekla{0,T}\times\Om$.
As \eqref{eq:regularity:timecont:1} entails $\norm{\widetilde{\phi}_j\pml-\widetilde{\phi}_j}_{L^{2\p}\trkla{\widetilde{\Omega};L^2\trkla{0,T;H^1\trkla{\Om}}}}\rightarrow0$, there exists a subsequence $\rkla{\widetilde{\phi}_j\pml}_{j\in\mathds{N}}$ converging pointwise almost everywhere towards $\widetilde{\phi}$.
Hence, we can again combine the uniform bounds from \eqref{eq:regularity:timecont:1} and Vitali's convergence theorem to establish \eqref{eq:conv:phiO:strongLp}.
Similar arguments provide \eqref{eq:conv:phiG:strong} and \eqref{eq:conv:phiG:strongLp}.
Finally, choosing $\bs{\psi}\in L^2\trkla{\widetilde{\Omega};L^2\trkla{0,T;\trkla{C^\infty\trkla{\overline{\Om}}}^d}}$ we obtain
\begin{multline}
\expectedt{\int_0^T\iO\widetilde{\phi}\div\bs{\psi}\dx\dt}\leftarrow \expectedt{\int_0^T\iO\widetilde{\phi}_j\pml\div\bs{\psi}\dx\dt}\\
=\expectedt{-\int_0^T\iO\nabla\widetilde{\phi}_j\pml\cdot\bs{\psi}\dx\dt+\int_0^T\iGamma\trace{\widetilde{\phi}_j\pml}\bs{\psi}\cdot\bs{n}\dG\dt}\\
\rightarrow \expectedt{-\int_0^T\iO\nabla\widetilde{\phi}\cdot\bs{\psi}\dx\dt+\int_0^T\iGamma\widetilde{\phi_\Gamma}\bs{\psi}\cdot\bs{n}\dG\dt}\\
=\expectedt{\int_0^T\iO\widetilde{\phi}\div\bs{\psi}\dx\dt}+\expectedt{\int_0^T\iGamma\rkla{\widetilde{\phi_\Gamma}-\trace{\widetilde{\phi}}}\bs{\theta}\cdot\bs{n}\dG\dt}
\end{multline}
allowing us to identify $\widetilde{\phi_\Gamma}$ with $\trace{\widetilde{\phi}}$.\\
The weak and weak* convergence results stated in \eqref{eq:conv:phiO:weakstar}, \eqref{eq:conv:phiG:weakstar}, \eqref{eq:conv:mu:weak}, \eqref{eq:conv:theta:weak}, \eqref{eq:conv:r:weakstar}, and \eqref{eq:conv:s:weakstar} are a direct consequence of the uniform bounds in \eqref{eq:regularity:timecont:1}.\\
The convergence results \eqref{eq:conv:intF}--\eqref{eq:conv:r:strong} can be obtained as follows (see also \cite{\citeASAV}):
By estimate (4.8) in \cite{Metzger2023} we have
\begin{align}
\esssup_{t\in\trkla{0,T}}\iO\abs{\trkla{1-\Ihjop}\gkla{F\rkla{\widetilde{\phi}_j}}}\dx\leq Ch_j\tabs{\Om}+C h_j\esssup_{t\in\trkla{0,T}}\norm{\widetilde{\phi}_j\pml}_{H^1\trkla{\Om}}^6
\end{align}
which indicates that the interpolation operator $\Ihjop$ is negligible when passing to the limit.
Starting from a pointwise almost everywhere convergent subsequence, we may use the continuity of $F$, the growth condition in \ref{item:potentials}, and Vitali's convergence theorem to establish \eqref{eq:conv:intF}.
The strong convergence follows from \eqref{eq:conv:intF} and Lemma \ref{lem:errorsav}.
This also allows for the identification of $\widetilde{r}$ with $\sqrt{\iO F\rkla{\widetilde{\phi}}\dx}$.
Similar arguments provide \eqref{eq:conv:intG}, \eqref{eq:conv:s:strong}, and the identification of $\widetilde{s}$.
The strong convergence of $\Xi_{j,\Om}^+$ and $\Xi_{j,\Gamma}^+$ stated in \eqref{eq:conv:xiO} and \eqref{eq:conv:xiG} follows from \eqref{eq:regularity:timecont:1.5}.
\end{proof}

\section{Passage to the limit}\label{sec:limit}
The goal of this section is the identification of the limits of the (discrete) stochastic integrals.
We shall use a more precise notation for our time grid and write $t\nn_j$ ($n=0,\ldots,N_j$) for the nodes of the equidistant grid obtained using the step size $\tau_j$.
Denoting the increments $\widetilde{\bs{\xi}}_j\trkla{t\nn_j}-\widetilde{\bs{\xi}}_j\trkla{t\no_j}$ by $\sinctilde{n}$ and using the stochastic processes defined in the last section, we can rewrite \eqref{eq:modeldisc:phiO} and \eqref{eq:modeldisc:phiG} as
\begin{subequations}\label{eq:model:timecont}
\begin{multline}\label{eq:model:timecont:Om}
\iO\Ihj{\rkla{\widetilde{\phi}_j\trkla{t}-\widetilde{\phi}_j^-}\psi\hj}\dx +\trkla{t-t_j\no}\iO\nabla\widetilde{\mu}_j^+\cdot\nabla\psi\hj\dx\\
=\frac{t-t_j\no}{\tau_j}\iO\Ihj{\Phi\hj\rkla{\widetilde{\phi}_j^-}\sinctilde{n}\psi\hj}\dx\,,
\end{multline}
\begin{multline}\label{eq:model:timecont:Gamma}
\iGamma\IhGj{\trace{\widetilde{\phi}_j\trkla{t}-\widetilde{\phi}_j^-}\widehat{\psi}\hj}\dG+\trkla{t-t\no_j}\iGamma\IhG{\widetilde{\theta}_j^+\widehat{\psi}\hj}\dG\\
=\frac{t-t_j\no}{\tau_j}\iGamma\IhGj{\trace{\Phi\hj\rkla{\widetilde{\phi}_j^-}\sinctilde{n}}\widehat{\psi}\hj}\dG\,,
\end{multline}
for all $\psi\hj\in\Uhj$ and $\widehat{\psi}\hj\in\UhGj$.
The relation between $\widetilde{\phi}_j$, $\widetilde{r}_j^+$, $\widetilde{s}_j^+$, $\widetilde{\mu}_j^+$, and $\widetilde{\theta}_j^+$ in the interval $\trkla{t_j\no,t_j\nn}$ is given by
\begin{align}\label{eq:model:timecont:potentials}
\begin{split}
\iO&\Ihj{\widetilde{\mu}_j^+\eta\hj}\dx+\iGamma\IhGj{\widetilde{\theta}_j^+\trace{\eta\hj}}\dG\\
=&\iO\nabla\widetilde{\phi}_j^+\cdot\nabla\eta\hj\dx +\iGamma\nablaG \trace{\widetilde{\phi}_j^+}\cdot\nablaG\trace{\eta\hj}\dG\\
&+ \frac{\widetilde{r}_j^+}{\sqrt{E\hj^\Om\trkla{\widetilde{\phi}_j^-}}}\iO\Ihj{F^\prime\trkla{\widetilde{\phi}_j^-}\eta\hj}\dx +\iO\Ihj{\widetilde{\Xi}_{j,\Om}^+\eta\hj}\dx\\
&+\frac{\widetilde{s}_j^+}{\sqrt{E\hj^\Gamma\trkla{\trace{\widetilde{\phi}_j^-}}}}\iGamma\IhGj{G^\prime\rkla{\trace{\widetilde{\phi}_j^-}}\trace{\eta\hj}}\dG + \iGamma\IhGj{\widetilde{\Xi}_{j,\Gamma}^+\trace{\eta\hj}}\dG\,
\end{split}
\end{align}
for all $\eta\hj\in\Uhj$.

\end{subequations}
As the convergence properties collected in Lemma \ref{lem:convergence} allow to pass to the limit $j\rightarrow\infty$ in the \eqref{eq:model:timecont:potentials} and the left-hand sides of \eqref{eq:model:timecont:Om} and \eqref{eq:model:timecont:Gamma}, it remains to identify the limit of the right-hand sides of \eqref{eq:model:timecont:Om} and \eqref{eq:model:timecont:Gamma} as suitable \Ito-integrals.\\
We start by introducing the filtration $\trkla{\widetilde{\mathcal{F}}_t}_{t\in\tekla{0,T}}$ as the augmentation of the filtration generated by $\widetilde{\phi}$, $\widetilde{L}$, $\widetilde{\mu}$, $\widetilde{\phi_\Gamma}$, $\widetilde{\theta}$, and $\widetilde{W}$ and show that $\widetilde{W}$ is a $\mathcal{Q}$-Wiener process.
\begin{lemma}\label{lem:wienerprocess}
The process $\widetilde{W}$ obtained in Theorem \ref{thm:jakubowski} is a $\mathcal{Q}$-Wiener process adapted to $\trkla{\widetilde{\mathcal{F}}_t}_{t\in\tekla{0,T}}$ and can be written as
\begin{align*}
\widetilde{W}=\sum_{k\in\mathds{Z}}\lambda_k\g{k}\widetilde{\beta}_k\,.
\end{align*}
Here $\trkla{\widetilde{\beta}_k}_{k\in\mathds{Z}}$ is a family of independently and identically distributed Brownian motions with respect to $\trkla{\widetilde{\mathcal{F}}_t}_{t\in\tekla{0,T}}$.
\end{lemma}
\begin{proof}
As the laws of $\bs{\xi}\hj^{\tau_j}$ and $\widetilde{\bs{\xi}}_j$ coincide, we have $\widetilde{\Prob}$-almost surely that $\widetilde{\bs{\xi}}_j$ is piecewise linear in time and satisfies
\begin{align}
\widetilde{\bs{\xi}}_j\trkla{t_j^m}=\sum_{n=1}^m\sqrt{\tau_j}\sum_{k\in\mathds{Z}\hj}\g{k}\widetilde{\xi}_k^{n,\tau_j}\,,
\end{align}
where $\widetilde{\xi}_k^{n,\tau_j}$ are mutually independent random variables on $\widetilde{\Omega}$ satisfying \ref{item:randomvars}. 
Although $\widetilde{\bs{\xi}}_j$ is not a martingale, we can use the $\widetilde{\Prob}$-almost sure convergence of $\widetilde{\phi}_j$, $\trace{\widetilde{\phi}_j}$, $l\tekla{\widetilde{\mu}_j^+}$, $\widetilde{\mu}_j^+$, $\widetilde{\theta}_j^+$, and $\widetilde{\bs{\xi}}_j$ established in Theorem \ref{thm:jakubowski} to show that the limit process $\widetilde{W}$ is a martingale.
According to \cite[Lemma 5.8]{Ondrejat2022} it suffices to show that
\begin{align}
\expectedt{\rkla{\widetilde{W}\trkla{t_2}-\widetilde{W}\trkla{t_1}}\prod_{i=1}^q\prod_{\iota=1}^o f_{i\iota}\rkla{\Psi_{i\iota}\rkla{\widetilde{\phi}\trkla{s_\iota},\widetilde{\phi_\Gamma}\trkla{s_\iota},\widetilde{L}\trkla{s_\iota},\widetilde{\mu}\trkla{s_\iota},\widetilde{\theta}\trkla{s_\iota},\widetilde{W}\trkla{s_\iota}}}}=0
\end{align}
for all times $0\leq s_1<\ldots<s_q\leq t_1<t_2\leq T$, all $f_{i\iota}\in C_b\trkla{\mathds{R}}$, and for all $\Psi_{1\iota},\ldots,\Psi_{r\iota}\in\mathcal{A}_{s_\iota}$, where $\mathcal{A}_{s_\iota}$ is a subset of real-valued functions on 
\begin{align}
\begin{split}
\widehat{\mathcal{X}}_{s_\iota}:=&\, C\trkla{\tekla{0,s_\iota};L^s\trkla{\Om}}\times C\trkla{\tekla{0,s_\iota};L^r\trkla{\Gamma}}\times L^2\trkla{0,s_\iota;\trkla{H^1\trkla{\Om}}^\prime}_{\weaktop}\\
&\times L^2\trkla{0,s_\iota;H^1\trkla{\Om}}_{\weaktop}\times L^2\trkla{0,s_\iota;L^2\trkla{\Gamma}}_{\weaktop}\times C\trkla{\tekla{0,s_\iota};H^1\trkla{\Om}}
\end{split}
\end{align}
for which the $\sigma$-algebra of $\mathcal{A}_{s_\iota}$ equals the Borel $\sigma$-algebra on $\widehat{\mathcal{X}}_{s_\iota}$.
As explained in Example 5.9 in \cite{Ondrejat2022}, these functions can always be chosen in a manner such that they are continuous with respect to weakly converging sequences.
Hence, the uniform integrability of $\widetilde{\bs{\xi}}_j$ established in Lemma \ref{lem:boundxi} and a Vitali argument provide
\begin{align}
\begin{split}
&\expectedt{\rkla{\widetilde{W}\trkla{t_2}-\widetilde{W}\trkla{t_1}}\prod_{i=1}^q\prod_{\iota=1}^o f_{i\iota}\rkla{\Psi_{i\iota}\rkla{\widetilde{\phi}\trkla{s_\iota},\widetilde{\phi_\Gamma}\trkla{s_\iota},\widetilde{L}\trkla{s_\iota},\widetilde{\mu}\trkla{s_\iota},\widetilde{\theta}\trkla{s_\iota},\widetilde{W}\trkla{s_\iota}}}}\\
&=\lim_{j\rightarrow\infty}\expectedt{\rkla{\widetilde{\bs{\xi}}_j\trkla{t_2}-\widetilde{\bs{\xi}}_j\trkla{t_1}}\prod_{i=1}^q\prod_{\iota=1}^o f_{i\iota}\rkla{\Psi_{i\iota}\rkla{\widetilde{\phi}_j\trkla{s_\iota},\trace{\widetilde{\phi}_j}\trkla{s_\iota},l\tekla{\widetilde{\mu}_j^+}\trkla{s_\iota},\widetilde{\mu}_j^+\trkla{s_\iota},\widetilde{\theta}_j^+\trkla{s_\iota},\widetilde{\bs{\xi}}_j\trkla{s_\iota}}}}\\
&=\lim_{j\rightarrow\infty}\expectedt{\rkla{\widetilde{\bs{\xi}}_j\trkla{t_j^n}-\widetilde{\bs{\xi}}_j\trkla{t_j^m}}\prod_{i=1}^q\prod_{\iota=1}^o f_{i\iota}\rkla{\Psi_{i\iota}\rkla{\widetilde{\phi}_j\trkla{s_\iota},\trace{\widetilde{\phi}_j}\trkla{s_\iota},l\tekla{\widetilde{\mu}_j^+}\trkla{s_\iota},\widetilde{\mu}_j^+\trkla{s_\iota},\widetilde{\theta}_j^+\trkla{s_\iota},\widetilde{\bs{\xi}}_j\trkla{s_\iota}}}}\\
&+\lim_{j\rightarrow\infty}\expectedt{\rkla{\widetilde{\bs{\xi}}_j\trkla{t_2}-\widetilde{\bs{\xi}}_j\trkla{t_j^n}}\prod_{i=1}^q\prod_{\iota=1}^o f_{i\iota}\rkla{\Psi_{i\iota}\rkla{\widetilde{\phi}_j\trkla{s_\iota},\trace{\widetilde{\phi}_j}\trkla{s_\iota},l\tekla{\widetilde{\mu}_j^+}\trkla{s_\iota},\widetilde{\mu}_j^+\trkla{s_\iota},\widetilde{\theta}_j^+\trkla{s_\iota},\widetilde{\bs{\xi}}_j\trkla{s_\iota}}}}\\
&-\lim_{j\rightarrow\infty}\expectedt{\rkla{\widetilde{\bs{\xi}}_j\trkla{t_1}-\widetilde{\bs{\xi}}_j\trkla{t_j^m}}\prod_{i=1}^q\prod_{\iota=1}^o f_{i\iota}\rkla{\Psi_{i\iota}\rkla{\widetilde{\phi}_j\trkla{s_\iota},\trace{\widetilde{\phi}_j}\trkla{s_\iota},l\tekla{\widetilde{\mu}_j^+}\trkla{s_\iota},\widetilde{\mu}_j^+\trkla{s_\iota},\widetilde{\theta}_j^+\trkla{s_\iota},\widetilde{\bs{\xi}}_j\trkla{s_\iota}}}}\\
&=\lim_{j\rightarrow\infty}\expectedt{\sum_{a=m+1}^n\sqrt{\tau_j}\sum_{k\in\mathds{Z}\hj}\g{k}\widetilde{\xi}_k^{a,\tau_j}}=0\,.
\end{split}
\end{align}
Here, we compared $t_1$ and $t_2$ to grid points $t_j^n$ and $t_j^m$ satisfying $0\leq t_j^n-t_2\leq \tau_j$ and $0\leq t_j^m-t_1\leq \tau_j$ and used the continuity of $\widetilde{\bs{\xi}}_j$.
We now define $\widetilde{\beta}_k\trkla{t}:=\lambda_k^{-1}\iO \widetilde{W}\trkla{t,x}\g{k}\trkla{x}\dx$ and use similar arguments to show that
\begin{align}
\widetilde{\beta}_k\trkla{t}\widetilde{\beta}_l\trkla{t}-\delta_{kl}t
\end{align}
with $\delta_{kl}$ denoting the usual Kronecker delta is a martingale.
Hence, by Levy's characterization of Brownian motions we obtain the result.
\end{proof}

In the next step, we shall show that the limit processes of the left-hand sides of \eqref{eq:model:timecont:Om} and \eqref{eq:model:timecont:Gamma} are martingales with respect to $\trkla{\widetilde{\mathcal{F}}_t}_{t\in\tekla{0,T}}$ and identify the corresponding quadratic variation processes.
Hence, we consider for arbitrary but fixed $v\in C^\infty\trkla{\overline{\Om}}$ and $w\in C^\infty\trkla{\Gamma}$ the families of processes
\begin{subequations}
\begin{align}
\widetilde{M}_j^v\trkla{t}:=&\,\iO\Ihj{\trkla{\widetilde{\phi}_j\trkla{t}-\widetilde{\phi}_j\trkla{0}}v}\dx + \int_0^t\iO\nabla\widetilde{\mu}_j^+\cdot\nabla \Ihj{v}\dx\ds\,,\label{eq:def:Mvj}\\
\widetilde{N}_j^w\trkla{t}:=&\,\iGamma\IhGj{\trkla{\widetilde{\phi}_j\trkla{t}-\widetilde{\phi}_j\trkla{0}}w}\dG +\int_0^t\iGamma\IhG{\widetilde{\theta}^+_jw}\dG\ds\,.
\end{align}
\end{subequations}
We shall now show that these processes converge towards the processes
\begin{subequations}\label{eq:def:MvNv}
\begin{align}
\widetilde{M}^v\trkla{t}:=&\,\iO\trkla{\widetilde{\phi}\trkla{t}-\widetilde{\phi}\trkla{0}}v\dx+\int_0^t\iO\nabla\widetilde{\mu}\cdot\nabla v\dx\ds\,,\\
\widetilde{N}^w\trkla{t}:=&\iGamma\trkla{\widetilde{\phi_\Gamma}\trkla{t}-\widetilde{\phi_\Gamma}\trkla{0}}w\dG+\int_0^t\iGamma\widetilde{\theta}w\dG\ds\,,
\end{align}
\end{subequations}
which are martingales with respect to $\trkla{\widetilde{\mathcal{F}}_t}_{t\in\tekla{0,T}}$:
\begin{lemma}\label{lem:martingales}
Let the assumptions of Lemma \ref{lem:convergence} hold true.
Then, the processes $\widetilde{M}^v$ and $\widetilde{N}^w$ defined in \eqref{eq:def:MvNv} with $v\in C\trkla{\overline{\Om}}$ and $w\in C\trkla{\Gamma}$ are martingales with respect to the filtration $\trkla{\widetilde{\mathcal{F}}_t}_{t\in\tekla{0,T}}$.
\end{lemma}
\begin{proof}
We start by showing that the piecewise linear process $\widetilde{M}_j^v$ converges towards $\widetilde{M}^v$ in $L^p\trkla{\widetilde{\Omega}}$ for $t\in\tekla{0,T}$:
\begin{multline}
\expectedt{\abs{\widetilde{M}^v\trkla{t}-\widetilde{M}_j^v\trkla{t}}^p}\\
\leq C\expectedt{\abs{\iO\rkla{\widetilde{\phi}\trkla{t}-\widetilde{\phi}_j\trkla{t}-\widetilde{\phi}\trkla{0}+\widetilde{\phi}_j\trkla{0}}v\dx}^p} +C\expectedt{\abs{\iO\rkla{\widetilde{\phi}_j\trkla{t}-\widetilde{\phi}_j\trkla{0}}\trkla{1-\Ihjop}\gkla{v}\dx}^p}\\
 +C\expectedt{\abs{\iO\trkla{1-\Ihjop}\gkla{\rkla{\widetilde{\phi}_j\trkla{t}-\widetilde{\phi}_j\trkla{0}}\Ihj{v}}\dx}^p}
+C\expectedt{\abs{\int_0^t\iO\nabla\rkla{\widetilde{\mu}-\widetilde{\mu}_j^+}\cdot\nabla v\dx\ds}^p} \\
+ C\expectedt{\abs{\int_0^t\iO\nabla\widetilde{\mu}_j^+\cdot\nabla\trkla{1-\Ihjop}\tgkla{v}\dx\ds}^p}
\rightarrow 0\,,
\end{multline}
due to the convergence results from Theorem \ref{thm:jakubowski}, Lemma \ref{lem:interpolationerror}, the regularity results collected in \eqref{eq:regularity:timecont}, and the standard error estimates for the nodal interpolation operator (see e.g.~\cite{BrennerScott}).
Hence, we can repeat the arguments from the proof of Lemma \ref{lem:wienerprocess} and show that 
\begin{align}
\begin{split}
\expectedt{\rkla{\widetilde{M}^v\trkla{t_2}-\widetilde{M}^v\trkla{t_1}}\mathfrak{f}^{qo}}=&\,\lim_{j\rightarrow\infty} \expectedt{\rkla{\widetilde{M}^v_j\trkla{t_j^n}-\widetilde{M}^v_j\trkla{t_j^m}}\mathfrak{f}^{qo}_j}\\
=&\,\lim_{j\rightarrow\infty}\expectedt{\sum_{a=m+1}^n\iO{\Phi\hj\trkla{\widetilde{\phi}_j\trkla{t_j^{a-1}}}\sinctilde{a}}\dx}=0\,,
\end{split}
\end{align}
due to \ref{item:randomvars}.
Here, we again used grid points $t_j^n$ and $t_j^m$ satisfying $0\leq t_j^n-t_2\leq \tau_j$ and $0\leq t_j^m-t_1\leq \tau_j$ and the abbreviations
\begin{subequations}\label{eq:sigmaalgebra}
\begin{align}
\mathfrak{f}^{qo}:=&\,\prod_{i=1}^q\prod_{\iota=1}^o f_{i\iota}\rkla{\Psi_{i\iota}\rkla{\widetilde{\phi}\trkla{s_\iota},\widetilde{\phi_\Gamma}\trkla{s_\iota},\widetilde{L}\trkla{s_\iota},\widetilde{\mu}\trkla{s_\iota},\widetilde{\theta}\trkla{s_\iota},\widetilde{W}\trkla{s_\iota}}}\\
\mathfrak{f}^{qo}_j:=&\,\prod_{i=1}^q\prod_{\iota=1}^o f_{i\iota}\rkla{\Psi_{i\iota}\rkla{\widetilde{\phi}_j\trkla{s_\iota},\trace{\widetilde{\phi}_j}\trkla{s_\iota},l\tekla{\widetilde{\mu}_j^+}\trkla{s_\iota},\widetilde{\mu}_j^+\trkla{s_\iota},\widetilde{\theta}_j^+\trkla{s_\iota},\widetilde{\bs{\xi}}_j\trkla{s_\iota}}}
\end{align}
\end{subequations}
for times $0\leq s_1 <\ldots<s_q\leq t_1<t_2$ to simplify the notation.
This establishes the martingale property of $\widetilde{M}^v$.
The properties of $\widetilde{N}^w$ follow by similar arguments.
\end{proof}

In the next step, we shall identify the quadratic variation processes $\skla{\!\!\skla{\widetilde{M}^v}\!\!}$ and $\skla{\!\!\skla{\widetilde{N}^w}\!\!}$ of $\widetilde{M}^v$ and $\widetilde{N}^w$, as well as their cross variation processes $\skla{\!\!\skla{\widetilde{M}^v,\widetilde{\beta}_k}\!\!}$ and $\skla{\!\!\skla{\widetilde{N}^w,\widetilde{\beta}_k}\!\!}$.
Natural candidates for the quadratic and cross variation processes are
\begin{subequations}\label{eq:def:variations}
\begin{align}
\skla{\!\!\skla{\widetilde{M}^v}\!\!}\trkla{t}:=&\,\int_0^t\sum_{k\in\mathds{Z}}\rkla{\iO \varrho\trkla{\widetilde{\phi}}\lambda_k\g{k} v\dx}^2\ds\,,\label{eq:quadvar:M}\\
\skla{\!\!\skla{\widetilde{M}^v,\widetilde{\beta}_k}\!\!}\trkla{t}:=&\,\int_0^t\iO\varrho\trkla{\widetilde{\phi}}\lambda_k\g{k}v\dx\ds\,,\label{eq:crossvar:M}\\
\skla{\!\!\skla{\widetilde{N}^w}\!\!}\trkla{t}:=&\,\int_0^t\sum_{k\in\mathds{Z}}\rkla{\iGamma \varrho\trkla{\widetilde{\phi_\Gamma}}\lambda_k\trace{\g{k}} w\dG}^2\ds\,,\label{eq:quadvar:N}\\
\skla{\!\!\skla{\widetilde{N}^w,\widetilde{\beta}_k}\!\!}\trkla{t}:=&\,\int_0^t\iGamma\varrho\trkla{\widetilde{\phi_\Gamma}}\lambda_k\trace{\g{k}}w\dG\ds\,.\label{eq:crossvar:N}
\end{align}
\end{subequations}
\begin{lemma}\label{lem:quadcrossvar}
The processes defined in \eqref{eq:def:variations} are the quadratic and cross variation processes of the martingales $\widetilde{M}^v$ and $\widetilde{N}^w$ defined in \eqref{eq:def:MvNv}.
\end{lemma}
\begin{proof}
We start by showing that $\rkla{\widetilde{M}^v}^2-\skla{\!\!\skla{\widetilde{M}^v}\!\!}$ is a martingale, i.e.~using the abbreviations introduced in \eqref{eq:sigmaalgebra} we shall show that
\begin{align}
\expectedt{\rkla{\rkla{\widetilde{M}^v\trkla{t_2}}^2-\rkla{\widetilde{M}^v\trkla{t_1}}^2 -\skla{\!\!\skla{\widetilde{M}^v}\!\!}\trkla{t_2}+\skla{\!\!\skla{\widetilde{M}^v}\!\!}\trkla{t_1}}\mathfrak{f}^{qo}}=0\,.
\end{align}
For this reason, we introduce a family of approximations $\rkla{\skla{\!\!\skla{\widetilde{M}_j^v}\!\!}}_{j\in\mathds{N}}$ of $\skla{\!\!\skla{\widetilde{M}^v}\!\!}$ via
\begin{align}
\skla{\!\!\skla{\widetilde{M}_j^v}\!\!}\trkla{t_j^m}:=&\,\int_0^{t_j^m}\sum_{k\in\mathds{Z}\hj}\rkla{\iO\Ihj{\varrho\trkla{\widetilde{\phi}_j^-}\lambda_k\g{k}v}\dx}^2\ds
\end{align}
and prove that $\skla{\!\!\skla{\widetilde{M}_j^v}\!\!}\trkla{t_j^m}$ converges in $L^p\trkla{\widetilde{\Omega}}$ towards $\skla{\!\!\skla{\widetilde{M}^v}\!\!}\trkla{t}$ for $t_j^m\searrow t$.
We decompose the discretization error as follows
\begin{multline}
\expectedt{\abs{\skla{\!\!\skla{\widetilde{M}^v}\!\!}\trkla{t}-\skla{\!\!\skla{\widetilde{M}^v_j}\!\!}\trkla{t_j^m}}^p}\\
\leq C\expectedt{\abs{\int_0^t\sum_{k\in\mathds{Z}}\rkla{\iO\varrho\trkla{\widetilde{\phi}}\lambda_k\g{k}v\dx}^2\ds-\int_0^t\sum_{k\in\mathds{Z}\hj}\rkla{\iO\varrho\trkla{\widetilde{\phi}^-_j}\lambda_k\g{k}v\dx}^2}^p} \\
+C\expectedt{\abs{\int_t^{t_j^m}\sum_{k\in\mathds{Z}\hj}\rkla{\iO{\varrho\trkla{\widetilde{\phi}_j^-}\lambda_k\g{k}v}\dx}^2\ds}^p}\\
+C\expectedt{\abs{\int_0^{t_j^m}\sum_{k\in\mathds{Z}\hj}\ekla{\rkla{\iO\varrho\trkla{\widetilde{\phi}_j^-}\lambda_k\g{k}v\dx}^2 -\rkla{\iO\Ihj{\varrho\trkla{\widetilde{\phi}_j^-}\lambda_k\g{k}v}\dx}^2}\ds}^p}\\
=:Q^{\Om}_1+Q^{\Om}_2+Q^{\Om}_3\,.
\end{multline}
From \ref{item:color}, \ref{item:rho}, and Lemma \ref{lem:convergence}, we obtain that $Q^{\Om}_1$ vanishes.
The bounds stated in \ref{item:rho} yield $Q_2^{\Om}\leq C\trkla{v}\tabs{t_j^m-t}^p\rightarrow0$.
To show that also $Q^{\Om}_3$ vanishes, we use \ref{item:rho}, the standard error estimates for the interpolation operator (see e.g.~\cite{BrennerScott}), and Lemma \ref{lem:interpolationerror} to compute
\begin{multline}\label{eq:tmp:q2}
\abs{\iO\varrho\trkla{\widetilde{\phi}_j^-}\lambda_k\g{k}v\dx -\iO\Ihj{\varrho\trkla{\widetilde{\phi}_j^-}\lambda_k\g{k}v}\dx}\\
\leq\abs{\iO\trkla{1-\Ihjop}\gkla{\varrho\trkla{\widetilde{\phi}_j^-}}\lambda_k\g{k} v\dx} +\abs{\iO\Ihj{\varrho\trkla{\widetilde{\phi}_j^-}}\lambda_k\trkla{1-\Ihjop}\gkla{\g{k}v}\dx} \\
+\abs{\iO\trkla{1-\Ihjop}\gkla{\Ihj{\varrho\trkla{\widetilde{\phi}_j^-}}\Ihj{\lambda_k\g{k}v}}\dx}\\
\leq C h_j \norm{\nabla\widetilde{\phi}_j^-}_{L^2\trkla{\Om}}\tabs{\lambda_k}\norm{\g{k}}_{L^\infty\trkla{\Om}}\norm{v}_{L^2\trkla{\Om}} + Ch_j\tabs{\lambda_k}\norm{\g{k}}_{W^{1,\infty}\trkla{\Om}}\norm{v}_{W^{1,\infty}\trkla{\Om}}\\
+C h_j \tabs{\lambda_k}\norm{\g{k}}_{W^{1,\infty}\trkla{\Om}}\norm{v}_{W^{1,\infty}\trkla{\Om}}\,.
\end{multline}
Hence,
\begin{align}
\begin{split}
Q_3^{\Om}\leq &\,C\widetilde{\mathds{E}}\left[\left|\int_0^{t^m_j}\sum_{k\in\mathds{Z}\hj}\abs{\iO\trkla{1+\Ihjop}\gkla{\varrho\trkla{\widetilde{\phi}_j^-}\lambda_k\g{k}v}\dx} h_j\tabs{\lambda_k}\norm{\g{k}}_{W^{2,\infty}\trkla{\Om}}\right.\right.\\
&\quad\times\left.\left.\norm{v}_{W^{1,\infty}\trkla{\Om}}\rkla{1+\norm{\nabla\widetilde{\phi}_j^-}_{L^2\trkla{\Om}}}\ds\right|^p\right]\\
\leq&\, C\trkla{v}\expectedt{\abs{h_j\sum_{k\in\mathds{Z}}\tabs{\lambda_k}^2\norm{\g{k}}_{W^{2,\infty}\trkla{\Om}}^2\int_0^{t_j^m}\rkla{1+\norm{\nabla\widetilde{\phi}_j^-}_{L^2\trkla{\Om}}}\ds}^p}\rightarrow 0\,.
\end{split}
\end{align}
Having established the convergence in $L^p\trkla{\widetilde{\Omega}}$, we have $\widetilde{\Prob}$-almost sure convergence for a subsequence.
Using similar arguments, we also obtain the uniform bounds for higher moments of $\skla{\!\!\skla{\widetilde{M}^v_j}\!\!}$.
Hence, we can reuse the ideas from the proof of Lemma \ref{lem:martingales} and deduce for $0\leq t_j^n-t_2\leq \tau_j$ and $0\leq t_j^m-t_1\leq\tau_j$:
\begin{align}
\begin{split}
&\expectedt{\rkla{\rkla{\widetilde{M}^v\trkla{t_2}}^2-\rkla{\widetilde{M}^v\trkla{t_1}}^2 -\skla{\!\!\skla{\widetilde{M}^v}\!\!}\trkla{t_2}+\skla{\!\!\skla{\widetilde{M}^v}\!\!}\trkla{t_1}}\mathfrak{f}^{qo}}\\
&=\lim_{j\rightarrow\infty}\expectedt{\rkla{\rkla{\widetilde{M}^v_j\trkla{t_j^n}}^2-\rkla{\widetilde{M}^v_j\trkla{t_j^m}}^2 -\skla{\!\!\skla{\widetilde{M}^v_j}\!\!}\trkla{t_j^n}+\skla{\!\!\skla{\widetilde{M}^v_j}\!\!}\trkla{t_j^m}}\mathfrak{f}^{qo}_j}\\
&=\lim_{j\rightarrow\infty}\widetilde{\mathds{E}}\left[\left(\sum_{a=m+1}^n\rkla{\iO\Ihj{\Phi\hj\trkla{\widetilde{\phi}_j\trkla{t_j^{a-1}}}\sinctilde{a}v}\dx}^2 \right.\right.\\
&\qquad\qquad\left.\left.-\sum_{a=m+1}^n\tau_j\sum_{k\in\mathds{Z}\hj}\rkla{\iO\Ihj{\varrho\trkla{\widetilde{\phi}_j\trkla{t_j^{a-1}}}\lambda_k\g{k}v}\dx}^2  \right)\mathfrak{f}_j^{qo}\right]\\
&=0\,.
\end{split}
\end{align}
Here we used that the mutual independence of the stochastic increments and $\expectedt{\abs{\widetilde{\xi}_k^{n,\tau_j}}^2}= 1$ (cf.~\ref{item:randomvars}).\\
We shall now apply similar arguments to show that $\skla{\!\!\skla{\widetilde{M}^v,\widetilde{\beta}_k}\!\!}$ defined in \eqref{eq:crossvar:M} is indeed the cross variation process.
We start by defining suitable time discrete approximations for time points $t_j^m$ via
\begin{align}
\skla{\!\!\skla{\widetilde{M}_j^v,\sum_{a=1}^m\sqrt{\tau_j}\widetilde{\xi}_k^{a,\tau_j}}\!\!}:=\left\{\begin{array}{cc}\int_0^{t_j^m}\iO\lambda_k\Ihj{\varrho\trkla{\widetilde{\phi}_j^-}\g{k}v}\dx\ds&\text{if~}k\in\mathds{Z}\hj\,,\\
0&\text{else,}
\end{array}\right.
\end{align}
and show that this approximation converges towards $\skla{\!\!\skla{\widetilde{M}^v,\widetilde{\beta}_k}\!\!}\trkla{t}$ for $t_j^m\searrow t$.
As $\bigcup_{h>0}\Zh=\mathds{Z}$, we can assume w.l.o.g. that $k\in\mathds{Z}\hj$. 
Hence
\begin{multline}
\expectedt{\abs{\int_0^t\iO\varrho\trkla{\widetilde{\phi}}\lambda_k\g{k} v\dx\ds-\int_0^{t_j^m}\iO\Ihj{\varrho\trkla{\widetilde{\phi}_j^-}\g{k}v}\dx\ds}^p}\\
\leq C\expectedt{\abs{\int_0^t\iO\rkla{\varrho\trkla{\widetilde{\phi}}-\varrho\trkla{\widetilde{\phi}_j^-}}\lambda_k\g{k} v\dx\ds}^p} +C\expectedt{\abs{\int_t^{t_j^m}\!\!\iO{\varrho\trkla{\widetilde{\phi}_j^-}\lambda_k\g{k}v}\dx\ds}^p} \\
+C\expectedt{\abs{\int_0^t\iO\trkla{1-\Ihjop}\gkla{\varrho\trkla{\widetilde{\phi}_j^-}\lambda_k\g{k} v}\dx\ds}^p}\\
=:R_1^{\Om}+R_2^{\Om}+R_3^{\Om}\,.
\end{multline}
As before, $R_1^{\Om}$ vanish due to \ref{item:rho} and Lemma \ref{lem:convergence}, while  $R_2^{\Om}$ vanishes due to $R_2^{\Om}\leq C\trkla{v}\tabs{t_j^m-t}^p$.
To treat $R_3^{\Om}$, we use the estimates in \eqref{eq:tmp:q2}.
Hence, we can again argue for $0\leq t_j^n-t_2\leq \tau_j$ and $0\leq t_j^m-t_1\leq\tau_j$ as follows:
\begin{align}
\begin{split}
\widetilde{\mathds{E}}&\ekla{\rkla{\widetilde{M}^v\trkla{t_2}\widetilde{\beta}_k\trkla{t_2}-\widetilde{M}^v\trkla{t_1}\widetilde{\beta}_k\trkla{t_1} -\skla{\!\!\skla{\widetilde{M}^v,\widetilde{\beta}_k}\!\!}\trkla{t_2}+\skla{\!\!\skla{\widetilde{M}^v,\widetilde{\beta}_k}\!\!}\trkla{t_1} }\mathfrak{f}^{qo}}\\
&=\lim_{j\rightarrow\infty}\widetilde{\mathds{E}}\left[\left(\widetilde{M}_j^v\trkla{t_j^n}\sum_{a=1}^n\sqrt{\tau_j}\widetilde{\xi}_k^{a,\tau_j}-\widetilde{M}_j^v\trkla{t_j^m}\sum_{a=1}^m\sqrt{\tau_j}\widetilde{\xi}_j^{a,\tau_j} \right.\right.\\
&\qquad\qquad\qquad\left.\left.-\int_{t_j^m}^{t_j^n}\iO\lambda_k\Ihj{\varrho\trkla{\widetilde{\phi}_j^-}\g{k}v}\dx\ds \right)\mathfrak{f}^{qo}_j\right]\\
&=\lim_{j\rightarrow\infty}\expectedt{\sum_{b=1}^n\iO\Ihj{\Phi\hj\rkla{\widetilde{\phi}_j\trkla{t_j^{b-1}}}\sinctilde{b} v}\dx \sum_{a=1}^n\sqrt{\tau_j}\widetilde{\xi}_k^{a,\tau_j} \mathfrak{f}^{qo}_j}\\
&\qquad-\lim_{j\rightarrow\infty}\expectedt{\sum_{b=1}^m\iO\Ihj{\Phi\hj\rkla{\widetilde{\phi}_j\trkla{t_j^{b-1}}}\sinctilde{b} v}\dx \sum_{a=1}^m\sqrt{\tau_j}\widetilde{\xi}_k^{a,\tau_j} \mathfrak{f}^{qo}_j}\\
&\qquad-\lim_{j\rightarrow\infty}\expectedt{\int_{t_j^m}^{t_j^n}\iO\lambda_k\Ihj{\varrho\trkla{\widetilde{\phi}_j^-}\g{k} v}\dx\ds \,\mathfrak{f}_j^{qo}}\\
&=\lim_{j\rightarrow\infty} \expectedt{\sum_{b=m+1}^n\sum_{a=1}^n \iO\Ihj{\Phi\hj\rkla{\widetilde{\phi}_j\trkla{t_j^{b-1}}}\sinctilde{b} \sqrt{\tau_j}\widetilde{\xi}_k^{a,\tau_j}v}\dx\,\mathfrak{f}_j^{qo}}\\
&\qquad+\lim_{j\rightarrow\infty} \expectedt{\sum_{b=1}^m\sum_{a=m+1}^n \iO\Ihj{\Phi\hj\rkla{\widetilde{\phi}_j\trkla{t_j^{b-1}}}\sinctilde{b} \sqrt{\tau_j}\widetilde{\xi}_k^{a,\tau_j}v}\dx\,\mathfrak{f}_j^{qo}}\\
&\qquad-\lim_{j\rightarrow\infty}\expectedt{\int_{t_j^m}^{t_j^n}\iO\lambda_k\Ihj{\varrho\trkla{\widetilde{\phi}_j^-}\g{k} v}\dx\ds \,\mathfrak{f}_j^{qo}}\\
&=\lim_{j\rightarrow\infty} \expectedt{\sum_{b=m+1}^n \iO\Ihj{\Phi\hj\rkla{\widetilde{\phi}_j\trkla{t_j^{b-1}}}\sinctilde{b} \sqrt{\tau_j}\widetilde{\xi}_k^{b,\tau_j}v}\dx\,\mathfrak{f}_j^{qo}}\\
&\qquad-\lim_{j\rightarrow\infty}\expectedt{\int_{t_j^m}^{t_j^n}\iO\lambda_k\Ihj{\varrho\trkla{\widetilde{\phi}_j^-}\g{k} v}\dx\ds \,\mathfrak{f}_j^{qo}}\\
&=\lim_{j\rightarrow\infty} \expectedt{\sum_{b=m+1}^n \iO\Ihj{\Phi\hj\rkla{\widetilde{\phi}_j\trkla{t_j^{b-1}}}\lambda_k\g{k}\widetilde{\xi}_k^{b,\tau_j} {\tau_j}\widetilde{\xi}_k^{b,\tau_j}v}\dx\,\mathfrak{f}_j^{qo}}\\
&\qquad-\lim_{j\rightarrow\infty}\expectedt{\int_{t_j^m}^{t_j^n}\iO\lambda_k\Ihj{\varrho\trkla{\widetilde{\phi}_j^-}\g{k} v}\dx\ds \,\mathfrak{f}_j^{qo}}\\
&=\lim_{j\rightarrow\infty}\widetilde{\mathds{E}}\left[\left(\sum_{b=m+1}^n\iO\Ihj{\Phi\hj\rkla{\widetilde{\phi}_j\trkla{t_j^{b-1}}}\lambda_k\g{k}\tau_j v}\dx\right.\right.\\
&\qquad\qquad\qquad\left.\left.-\int_{t_j^m}^{t_j^n}\iO\lambda_k\Ihj{\varrho\trkla{\widetilde{\phi}_j^-}\g{k} v}\dx\ds\right)\mathfrak{f}_j^{qo}\right]\\
&=0\,.
\end{split}
\end{align}
Here, we used \eqref{eq:def:Mvj} and \ref{item:randomvars}.\\ 
We repeat the above arguments to show that $\skla{\!\!\skla{\widetilde{N}^w}\!\!}$ and $\skla{\!\!\skla{\widetilde{N}^w,\widetilde{\beta}_k}\!\!}$ defined in \eqref{eq:quadvar:N} and \eqref{eq:crossvar:N} are indeed the quadratic and cross variation processes of the martingale $\widetilde{N}^w$.
We again introduce a family of approximations $\rkla{\skla{\!\!\skla{\widetilde{N}^w_j}\!\!}}_{j\in\mathds{N}}$ of $\skla{\!\!\skla{\widetilde{N}^w}\!\!}$ via
\begin{align}
\skla{\!\!\skla{\widetilde{N}^w_j}\!\!}\trkla{t_j^m}:=\int_0^{t_j^m}\sum_{k\in\mathds{Z}\hj}\rkla{\iGamma\IhGj{\trace{\varrho\trkla{\widetilde{\phi}_j^-}\lambda_k\g{k}}w}\dG}^2\ds\,.
\end{align}
To establish convergence, we again decompose the error as follows:
\begin{multline}
\expectedt{\abs{\skla{\!\!\skla{\widetilde{N}^w}\!\!}\trkla{t}-\skla{\!\!\skla{\widetilde{N}_j^w}\!\!}\trkla{t_j^m}}^p}\\
\leq C\expectedt{\abs{\int_0^t\sum_{k\in\mathds{Z}}\rkla{\iGamma \varrho\trkla{\widetilde{\phi_\Gamma}}\lambda_k\trace{\g{k}}w\dG}^2\ds -\int_0^t\sum_{k\in\mathds{Z}\hj}\rkla{\iGamma\varrho\trkla{\trace{\widetilde{\phi}_j^-}}\lambda_k\trace{\g{k}}w\dG}^2\ds}^p}\\
+C\expectedt{\abs{\int_t^{t_j^m}\sum_{k\in\mathds{Z}\hj}\rkla{{\trace{\varrho\trkla{\widetilde{\phi}_j^-}\lambda_k\g{k}}w}\dG}^2\ds}^p}\\
+C\expectedt{\abs{\int_0^{t_j^m}\sum_{k\in\mathds{Z}\hj}\ekla{\rkla{\iGamma\varrho\rkla{\trace{\widetilde{\phi}_j^-}}\lambda_k\trace{\g{k}}w\dG}^2-\rkla{\iGamma\IhGj{\trace{\varrho\trkla{\widetilde{\phi}_j^-}\lambda_k\g{k}}w}\dG}^2}\ds}^p}\\
=:Q_1^\Gamma+Q_2^\Gamma+Q_3^\Gamma\,.
\end{multline}
As before, $Q_1^\Gamma$ vanishes due to Lemma \ref{lem:convergence}, \ref{item:color}, and \ref{item:rho}.
The term $Q_2^\Gamma$ can be treated analogously to $Q_2$.
To show that $Q_3^\Gamma$ vanishes, we note that $\widetilde{\phi}_j$ are finite element functions and $\g{k}$ are bounded in $W^{2,\infty}\trkla{\Om}$, i.e.~they are continuous on $\overline{\Om}$.
Together with \eqref{eq:tracespace}, \ref{item:rho}, and the arguments from \eqref{eq:tmp:q2}, this provides
\begin{multline}
\abs{\iGamma \varrho\trkla{\trace{\widetilde{\phi}_j^-}}\lambda_k\trace{\g{k}}w\dG- \iGamma\IhGj{\trace{\rho\trkla{\widetilde{\phi}_j^-}\lambda_k\g{k}}w}\dG}\\
\leq\abs{\iGamma \varrho\trkla{\trace{\widetilde{\phi}_j^-}}\lambda_k\trace{\g{k}}w\dG- \iGamma\IhGj{\rho\trkla{\trace{\widetilde{\phi}_j^-}}}\lambda_k\IhGj{\trace{\g{k}}w}\dG}\\
+\abs{\iGamma\trkla{1-\IhGjop}\gkla{\IhGj{\rho\trkla{\trace{\widetilde{\phi}_j^-}}}\lambda_k\IhGj{\trace{\g{k}}w}}\dG}\\
\leq C h_j\tabs{\lambda_k}\norm{\trace{\g{k}}}_{W^{1,\infty}\trkla{\Gamma}}\norm{w}_{W^{1,\infty}\trkla{\Gamma}}\rkla{1+\norm{\nablaG\widetilde{\phi}_j^-}_{L^2\trkla{\Gamma}}}\,.
\end{multline}
As $\norm{\trace{\g{k}}}_{W^{1,\infty}\trkla{\Gamma}}\leq C\norm{\g{k}}_{W^{2,\infty}\trkla{\Om}}$, we have
\begin{align}
\begin{split}
Q_3^\Gamma\leq&\,C\trkla{w}\expectedt{\abs{h_j\sum_{k\in\mathds{Z}}\tabs{\lambda_k}^2\norm{\g{k}}_{W^{2,\infty}\trkla{\Om}}^2\int_0^t\rkla{1+\norm{\nablaG\widetilde{\phi}_j^-}_{L^2\trkla{\Gamma}}}\ds}^p}\rightarrow 0\,.
\end{split}
\end{align}
As before, we use this convergence property to show that $\rkla{\widetilde{N}^w}^2-\skla{\!\!\skla{\widetilde{N}^w}\!\!}$ is a martingale by showing this property for its approximations and using the convergence.
For $0\leq t_j^n-t_2\leq\tau_j$ and $0\leq t_j^m-t_1\leq \tau_j$ we compute
\begin{align}
\begin{split}
&\expectedt{\rkla{\rkla{\widetilde{N}^w\trkla{t_2}}^2-\rkla{\widetilde{N}^w\trkla{t_1}}^2-\skla{\!\!\skla{\widetilde{N}^w}\!\!}\trkla{t_2}+\skla{\!\!\skla{\widetilde{N}^w}\!\!}\trkla{t_1}}\mathfrak{f}^{qo}}\\
&\quad=\lim_{j\rightarrow\infty}\widetilde{\mathds{E}}\left[\left(\sum_{a=m+1}^n\rkla{\iGamma\IhGj{\trace{\Phi\hj\trkla{\widetilde{\phi}_j\trkla{t_j^{a-1}}}\sinctilde{a}} w}\dG}^2\right.\right.\\
&\qquad\quad\left.\left.-\sum_{a=m+1}^n\tau_j\sum_{k\in\mathds{Z}\hj}\rkla{\iGamma\IhGj{\trace{\varrho\trkla{\widetilde{\phi}_j\trkla{t_j^{a-1}}}\lambda_k\g{k}}w}\dG}^2\right)\mathfrak{f}_j^{qo}\right]\\
&\quad=0
\end{split}
\end{align}
due to \ref{item:randomvars}.
To identify the cross variations of $\widetilde{N}^w$, we define suitable approximations for time points $t_j^m$ via
\begin{align}
\skla{\!\!\skla{\widetilde{N}^w_j,\sum_{a=1}^m\sqrt{\tau_j}\widetilde{\xi}_k^{a,\tau_j}}\!\!}:=\left\{\begin{array}{cc}
\int_0^{t_j^m}\iGamma\IhGj{\trace{\varrho\trkla{\widetilde{\phi}_j^-}\lambda_k\g{k}}w}\dG\ds&\text{if~}k\in\mathds{Z}\hj\,,\\
0&\text{else}\,.
\end{array}\right.
\end{align}
Assuming w.l.o.g.~$k\in\mathds{Z}\hj$ and arguing like before shows
\begin{align}
\expectedt{\abs{\int_0^t\iGamma\varrho\trkla{\widetilde{\phi_\Gamma}}\lambda_k\trace{\g{k}}w\dG\ds-\int_0^{t_j^m}\iGamma\IhGj{\trace{\varrho\trkla{\widetilde{\phi}_j^-}\lambda_k\g{k}}w}\dG\ds}^p}\rightarrow 0
\end{align}
and hence
\begin{multline}
\expectedt{\rkla{\widetilde{N}^w\trkla{t_2}\widetilde{\beta}_k\trkla{t_2}-\widetilde{N}^w\trkla{t_1}\widetilde{\beta}_k\trkla{t_1} -\skla{\!\!\skla{\widetilde{N}^w,\widetilde{\beta}_k}\!\!}\trkla{t_2}+\skla{\!\!\skla{\widetilde{N}^w,\widetilde{\beta}_k}\!\!}\trkla{t_1} }\mathfrak{f}^{qo}}\\
=\lim_{j\rightarrow\infty} \widetilde{\mathds{E}}\left[\left(\widetilde{N}_j^w\trkla{t_j^n}\sum_{a=1}^n\sqrt{\tau_j}\widetilde{\xi}_k^{a,\tau_j}-\widetilde{N}_j^w\trkla{t_j^m}\sum_{a=1}^m\sqrt{\tau_j}\widetilde{\xi}_j^{a,\tau_j} \right.-\right.\\
\left.\left.-\int_{t_j^m}^{t_j^n}\iGamma\lambda_k\IhGj{\trace{\varrho\trkla{\widetilde{\phi}_j^-}\g{k}}w}\dG\ds \right)\mathfrak{f}^{qo}_j\right]=0\,.
\end{multline}
\end{proof}
Having obtained explicit expressions for the quadratic and cross variations of the martingales $\widetilde{M}^v$ and $\widetilde{N}^w$, we can follow the approach introduced in \cite{BrzezniakOndrejat, Ondrejat2010, HofmanovaSeidler} (see also \cite{HofmanovaRoegerRenesse,BreitFeireislHofmanova}) to show that $\widetilde{M}^v$ and $\widetilde{N}^w$ can be written as It\^o integrals with the Wiener process $\widetilde{W}$.
\begin{lemma}
Let the assumptions of Theorem \ref{thm:jakubowski} and Lemma \ref{lem:convergence} hold true.
Then we have for $\widetilde{M}^v$ and $\widetilde{N}^w$
\begin{subequations}
\begin{align}
\widetilde{M}^v\trkla{t}&=\int_0^t\iO\sum_{k\in\mathds{Z}} \varrho\trkla{\widetilde{\phi}} \lambda_k\g{k}v\dx\,\mathrm{d}\widetilde{\beta}_k\,,\label{eq:itoM}\\
\widetilde{N}^w\trkla{t}&=\int_0^t\iGamma\sum_{k\in\mathds{Z}} \varrho\trkla{\widetilde{\phi_\Gamma}} \lambda_k\trace{\g{k}}w\dG\,\mathrm{d}\widetilde{\beta}_k\,.\label{eq:itoN}
\end{align}
\end{subequations}
\end{lemma}
\begin{proof}
As a martingale with vanishing quadratic variation is almost surely constant, it suffices to show that the quadratic variation of 
\begin{align}
\widetilde{M}^v\trkla{t}-\int_0^t\iO\sum_{k\in\mathds{Z}}\varrho\trkla{\widetilde{\phi}}\lambda_k\g{k}v\dx\,\mathrm{d}\widetilde{\beta}_k
\end{align}
vanishes to establish \eqref{eq:itoM}.
We get
\begin{multline}\label{eq:tmp:quadraticvarbinom}
\skla{\!\!\skla{\widetilde{M}^v-\int_0^{\trkla{\cdot}}\iO\sum_{k\in\mathds{Z}}\varrho\trkla{\widetilde{\phi}}\lambda_k\g{k}v\dx\,\mathrm{d}\widetilde{\beta}_k}\!\!}\trkla{t}\\
=\skla{\!\!\skla{\widetilde{M}^v}\!\!}\trkla{t}+\skla{\!\!\skla{\int_0^{\trkla{\cdot}}\iO\sum_{k\in\mathds{Z}}\varrho\trkla{\widetilde{\phi}}\lambda_k\g{k}v\dx\,\mathrm{d}\widetilde{\beta}_k}\!\!}\trkla{t}\\
-2\skla{\!\!\skla{\widetilde{M}^v,\int_0^{\trkla{\cdot}}\iO\sum_{k\in\mathds{Z}}\varrho\trkla{\widetilde{\phi}}\lambda_k\g{k}v\dx\,\mathrm{d}\widetilde{\beta}_k}\!\!}\trkla{t}\,.
\end{multline}
To find a suitable expression for the last term on the right-hand side of \eqref{eq:tmp:quadraticvarbinom}, we apply the cross-variation formula (see e.g.~\cite{KaratzasShreve2004}) to obtain
\begin{align}
\skla{\!\!\skla{\widetilde{M}^v,\int_0^{\trkla{\cdot}}\iO\sum_{k\in\mathds{Z}}\varrho\trkla{\widetilde{\phi}}\lambda_k\g{k}v\dx\,\mathrm{d}\widetilde{\beta}_k}\!\!}\trkla{t} = \int_0^t\iO \sum_{k\in\mathds{Z}}\varrho\trkla{\widetilde{\phi}}\lambda_k\g{k} v\dx\,\mathrm{d}\skla{\!\!\skla{\widetilde{M}^v,\widetilde{\beta}_k}\!\!}\trkla{s}\,.
\end{align}
By \eqref{eq:crossvar:M} and \ref{item:rho}, we deduce that the process $\tekla{0,T}\ni s\mapsto \skla{\!\!\skla{\widetilde{M}^v,\widetilde{\beta}_k}\!\!}\trkla{s}$ is absolutely continuous.
Hence, we have
\begin{align}
\mathrm{d}\skla{\!\!\skla{\widetilde{M}^v,\widetilde{\beta}_k}\!\!}\trkla{s}= \lambda_k\iO\varrho\trkla{\widetilde{\phi}}\g{k} v\dx\ds
\end{align}
and consequently
\begin{align}
\skla{\!\!\skla{\widetilde{M}^v,\int_0^{\trkla{\cdot}}\iO\sum_{k\in\mathds{Z}}\varrho\trkla{\widetilde{\phi}}\lambda_k\g{k}v\dx\,\mathrm{d}\widetilde{\beta}_k}\!\!}\trkla{t}=\int_0^t\sum_{k\in\mathds{Z}}\tabs{\lambda_k}^2\rkla{\iO\varrho\trkla{\widetilde{\phi}}\g{k} v\dx}^2\ds\,.
\end{align}
Together with \eqref{eq:quadvar:M} and
\begin{align}
\skla{\!\!\skla{\int_0^{\trkla{\cdot}}\iO\sum_{k\in\mathds{Z}}\varrho\trkla{\widetilde{\phi}}\lambda_k\g{k}v\dx\,\mathrm{d}\widetilde{\beta}_k}\!\!}\trkla{t}=\int_0^t\sum_{k\in\mathds{Z}}\tabs{\lambda_k}^2\rkla{\iO\varrho\trkla{\widetilde{\phi}}\g{k}v\dx}^2\ds\,,
\end{align}
this provides \eqref{eq:itoM}.
Repeating the above arguments also provides \eqref{eq:itoN}.
\end{proof}
We conclude the proof of Theorem \ref{thm:mainresult} by passing to the limit in \eqref{eq:model:timecont:potentials}.
For any $v\in C^\infty\trkla{\overline{\Om}}$, we choose $\Ihj{v}$ multiplied by a sufficiently regular function from $\widetilde{\Omega}\times\tekla{0,T}$ to $\mathds{R}$ as a test function in \eqref{eq:model:timecont:potentials}, an application of the convergence results collected in Lemma \ref{lem:convergence} yields the following result:
\begin{lemma}
Let the assumptions of Theorem \ref{thm:jakubowski} and Lemma \ref{lem:convergence} hold true.
Then the limit processes $\widetilde{\phi}$, $\widetilde{\phi_\Gamma}$, $\widetilde{\mu}$, and $\widetilde{\theta}$ satisfy
\begin{multline}
\iO \widetilde{\mu}\eta\dx+ \iGamma\widetilde{\theta}\trace{\eta}\dG=\iO\nabla\widetilde{\phi}\cdot\nabla \eta\dx+\iGamma\nablaG\widetilde{\phi_\Gamma}\cdot\nablaG \trace{\eta}\dG\\
+\iO F^\prime\trkla{\widetilde{\phi}} \eta\dx +\iGamma G^\prime\trkla{\widetilde{\phi_\Gamma}}\trace{\eta}\dG
\end{multline}
$\widetilde{\Prob}$-almost surely for almost all $t\in\tekla{0,T}$ and all $\eta\in C^\infty\trkla{\overline{\Om}}$.
\end{lemma}
A density argument shows that the martingale solutions satisfy the formulation stated in Theorem \ref{thm:mainresult}.
Hence, it remains to establish the pathwise uniqueness.

\section{Pathwise uniqueness of martingale solutions}\label{sec:uniqueness}
In this section we shall prove that the martingale solutions established in Section \ref{sec:limit} are pathwise unique and hence complete the proof of Theorem \ref{thm:mainresult}.
As the Allen--Cahn type equation on $\Gamma$ can be interpreted as an $L^2$-gradient flow and the Cahn--Hilliard equation on $\Om$ is a gradient flow with respect to the $H^{-1}$-norm, we shall estimate the difference between two possible solutions using a combination of these norms.
To define the dual norm on $\Om$, we follow the ideas used in \cite{Colli2014} to prove the uniqueness of the deterministic problem and define
\begin{align}
\dom \LapInv:=\gkla{\varphi^*\in \trkla{H^1\trkla{\Om}}^\prime\,:\,\trkla{\varphi^*}_\Om=0}&&\text{and}&&\LapInv\,:\,\dom\LapInv\,\rightarrow\,\gkla{\varphi\in H^1\trkla{\Om}\,:\,\trkla{\varphi}_\Om=0}
\end{align}
by setting for $\varphi^*\in\dom\LapInv$
\begin{align}\label{eq:defLapInv}
\LapInv\varphi^*\in H^1\trkla{\Om}\,,\quad \trkla{\LapInv\varphi^*}_\Om=0\,\quad\text{and}\quad \iO\nabla\LapInv\varphi^*\cdot\nabla z\dx=\skla{\varphi^*,z}\quad\forall z\in H^1\trkla{\Om}\,.
\end{align}
Hence, $\LapInv \varphi^*$ is the solution to the generalized Neumann problem for $-\Delta$ with datum $\varphi^*$ satisfying $\trkla{\varphi^*}_\Om =0$.
From \eqref{eq:defLapInv}, we immediately obtain the identity
\begin{align}
\skla{\psi^*,\LapInv\varphi^*}=\skla{\varphi^*,\LapInv\psi^*}=\iO\nabla\LapInv\psi^*\cdot\nabla\LapInv\varphi^*\dx\qquad \text{for~}\psi^*,\varphi^*\in\dom\LapInv\,.
\end{align}
Using $\LapInv$, we define a norm $\norm{\cdot}_*$ on $\trkla{H^1\trkla{\Om}}^\prime$ that is equivalent to the usual dual norm via
\begin{align}\label{eq:def:dualnorm}
\norm{\varphi^*}_*^2:=\norm{\nabla\LapInv\rkla{\varphi^*-\trkla{\varphi^*}_\Om}}_{L^2\trkla{\Om}}^2+\abs{\trkla{\varphi^*}_\Om}^2\,
\end{align}
and denote the corresponding semi-norm by
\begin{align}
\abs{\varphi^*}_*^2:=\norm{\nabla\LapInv\rkla{\varphi^*-\trkla{\varphi^*}_\Om}}_{L^2\trkla{\Om}}^2\,.
\end{align}

\begin{lemma}\label{lem:normestimate}
Let the assumptions \ref{item:time}, \ref{item:space1}, \ref{item:space2}, \ref{item:potentials}, \ref{item:initial}, \ref{item:rho}, and \ref{item:filtration}-\ref{item:color} hold true.
Then, martingale solutions to \eqref{eq:weakform} are pathwise unique.
\end{lemma}
\begin{proof}
Let $\trkla{\widetilde{\phi}_1,\widetilde{\phi_\Gamma}_1,\widetilde{\mu}_1, \widetilde{\theta}_1}$ and $\trkla{\widetilde{\phi}_2,\widetilde{\phi_\Gamma}_2,\widetilde{\mu}_2, \widetilde{\theta}_2}$ be two martingale solutions to \eqref{eq:weakform} with the same initial data $\phi_0$, the same probability space $\rkla{\widetilde{\Omega},\widetilde{\mathcal{A}},\widetilde{\mathcal{F}},\widetilde{\Prob}}$, and the same $\mathcal{Q}$-Wiener process $\widetilde{W}$
To show that these solutions are identical, we will make use a local monotonicity argument (cf.~\cite{Liu2013} and \cite{rocknerShangZhang2024} for a generalization) and apply a Gronwall type argument.
Hence, we introduce
\begin{align}
\mathfrak{N}\trkla{\varphi^*,\varphi^*_\Gamma}:= \norm{\varphi^*}_*^2+\norm{\varphi^*_\Gamma}_{L^2\trkla{\Gamma}}^2+2\tabs{\Gamma}\trkla{\varphi^*}_\Om^2 +2\trkla{\varphi^*}_\Om\trkla{\varphi^*_\Gamma}_\Gamma\geq \norm{\varphi^*}_*^2+\tfrac12\norm{\varphi^*_\Gamma}_{L^2\trkla{\Gamma}}^2
\end{align}
and apply \Ito's formula to the expression
\begin{align}
\mathfrak{R}\trkla{t,\trkla{\widetilde{\phi}_1,\widetilde{\phi_\Gamma}_1}-\trkla{\widetilde{\phi}_2,\widetilde{\phi_\Gamma}_2}}
:=\tfrac12e^{-\int_0^t \trkla{\widehat{C}+\vartheta\trkla{\widetilde{\phi}_1,\widetilde{\phi}_2,\widetilde{\phi_\Gamma}_1,\widetilde{\phi_\Gamma}_2}}\ds }\mathfrak{N}\trkla{\widetilde{\phi}_1-\widetilde{\phi}_2,\widetilde{\phi_\Gamma}_1-\widetilde{\phi_\Gamma}_2}\,
\end{align}
with a suitable positive constant $\widehat{C}$ and a measurable mapping $\vartheta$ which we will define later.
As both solutions share the same initial conditions, we obtain for $\hat{T}\in\tekla{0,T}$
\begin{align}
\begin{split}
\widetilde{\mathds{E}}&\left[\mathfrak{R}\trkla{\hat{T},\trkla{\widetilde{\phi}_1\trkla{\hat{T}},\widetilde{\phi_\Gamma}_1\trkla{\hat{T}}}-\trkla{\widetilde{\phi}_2\trkla{\hat{T}},\widetilde{\phi_\Gamma}_2\trkla{\hat{T}}}}\right]\\
=&\,\expectedt{\int_0^{\hat{T}}\ekla{-\rkla{\widehat{C}+\vartheta\trkla{\widetilde{\phi}_1,\widetilde{\phi}_2,\widetilde{\phi_\Gamma}_1,\widetilde{\phi_\Gamma}_2}}} \mathfrak{R}\trkla{t,\trkla{\widetilde{\phi}_1,\widetilde{\phi_\Gamma}_1}-\trkla{\widetilde{\phi}_2,\widetilde{\phi_\Gamma}_2}}\dt}\\
&+\expectedt{\int_0^{\hat{T}} e^{-\int_0^t \trkla{\widehat{C}+\vartheta\trkla{\widetilde{\phi}_1,\widetilde{\phi}_2,\widetilde{\phi_\Gamma}_1,\widetilde{\phi_\Gamma}_2}}\ds } \rkla{-\iO\rkla{\widetilde{\phi}_1-\widetilde{\phi}_2-\trkla{\widetilde{\phi}_1-\widetilde{\phi}_2}_\Om}\rkla{\widetilde{\mu}_1-\widetilde{\mu}_2}\dx}\dt}\\
&+\expectedt{\int_0^{\hat{T}} e^{-\int_0^t \trkla{\widehat{C}+\vartheta\trkla{\widetilde{\phi}_1,\widetilde{\phi}_2,\widetilde{\phi_\Gamma}_1,\widetilde{\phi_\Gamma}_2}}\ds } \rkla{-\iGamma \rkla{\widetilde{\phi_\Gamma}_1-\widetilde{\phi_\Gamma}_2-\trkla{\widetilde{\phi}_1-\widetilde{\phi}_2}_\Om}\rkla{\widetilde{\theta}_1-\widetilde{\theta}_2} \dG}\dt}\\
&+\tfrac12\widetilde{\mathds{E}}\left[\int_0^{\hat{T}} e^{-\int_0^t \trkla{\widehat{C}+\vartheta\trkla{\widetilde{\phi}_1,\widetilde{\phi}_2,\widetilde{\phi_\Gamma}_1,\widetilde{\phi_\Gamma}_2}}\ds } \left(\sum_{k\in\mathds{Z}}\norm{\rkla{\varrho\trkla{\widetilde{\phi}_1}-\varrho\trkla{\widetilde{\phi}_2}}\lambda_k\g{k} }_*^2\right.\right.\\
&\qquad\qquad + \sum_{k\in\mathds{Z}}\norm{\trace{\rkla{\varrho\trkla{\widetilde{\phi}_1}-\varrho\trkla{\widetilde{\phi}_2}} \lambda_k\g{k}}}_{L^2\trkla{\Gamma}}^2  + 2\tabs{\Gamma}\sum_{k\in\mathds{Z}}\trkla{\trkla{\varrho\trkla{\widetilde{\phi}_1}-\varrho\trkla{\widetilde{\phi}_2}}\lambda_k\g{k}}_\Om^2\\
&\left.\left.\qquad\qquad+2\sum_{k\in\mathds{Z}}\trkla{\trkla{\varrho\trkla{\widetilde{\phi}_1}-\varrho\trkla{\widetilde{\phi}_2}}\lambda_k\g{k}}_\Om \rkla{\trace{\trkla{\varrho\trkla{\widetilde{\phi}_1}-\varrho\trkla{\widetilde{\phi}_2}}\lambda_k\g{k}}}_\Gamma \right)\dt\right]\\
=:&\,I_1+I_2+I_3+I_4\,.
\end{split}
\end{align}
We start by estimating the $I_2+I_3$:
As $\widetilde{\phi_\Gamma}_1-\widetilde{\phi_\Gamma}_2-\trkla{\widetilde{\phi}_1-\widetilde{\phi}_2}_\Om$ is $\widetilde{\Prob}$-a.s.~the trace of $\widetilde{\phi}_1-\widetilde{\phi}_2 -\trkla{\widetilde{\phi}_1-\widetilde{\phi}_2}_\Om$, we obtain
\begin{align}
\begin{split}
-\iO&\rkla{\widetilde{\phi}_1-\widetilde{\phi}_2-\trkla{\widetilde{\phi}_1-\widetilde{\phi}_2}_\Om}\trkla{\widetilde{\mu}_1-\widetilde{\mu}_2}\dx-\iGamma \rkla{\widetilde{\phi_\Gamma}_1-\widetilde{\phi_\Gamma}_2-\trkla{\widetilde{\phi}_1-\widetilde{\phi}_2}_\Om}\rkla{\widetilde{\theta}_1-\widetilde{\theta}_2} \dG\\
=&\,-\norm{\nabla\widetilde{\phi}_1-\nabla\widetilde{\phi}_2}_{L^2\trkla{\Om}}^2-\norm{\nablaG\widetilde{\phi_\Gamma}_1-\nablaG\widetilde{\phi_\Gamma}_2}_{L^2\trkla{\Gamma}}^2 \\
&-\iO\rkla{F^\prime\trkla{\widetilde{\phi}_1}-F^\prime\trkla{\widetilde{\phi}_2}}\rkla{\widetilde{\phi}_1-\widetilde{\phi}_2-\trkla{\widetilde{\phi}_1-\widetilde{\phi}_2}_\Om}\dx\\
&-\iGamma\rkla{G^\prime\trkla{\widetilde{\phi_\Gamma}_1}-G^\prime\trkla{\widetilde{\phi_\Gamma}_2}}\rkla{\widetilde{\phi_\Gamma}_1-\widetilde{\phi_\Gamma}_2-\trkla{\widetilde{\phi}_1-\widetilde{\phi}_2}_\Om}\dG\\
=:&\,R_1+R_2+R_3+R_4\,.
\end{split}
\end{align}
Similarly to \cite{\citeASAV}, we use Hölder's inequality, \ref{item:potentials}, the standard Sobolev embeddings, and Young's inequality to deduce for $0<\alpha<\!\!<1$
\begin{align}
\begin{split}
\tabs{R_4}\leq&\,C\norm{1+\abs{\widetilde{\phi_\Gamma}_1}^2+\abs{\widetilde{\phi_\Gamma}_2}^2}_{L^3\trkla{\Gamma}}\norm{\widetilde{\phi_\Gamma}_1-\widetilde{\phi_\Gamma}_2}_{L^2\trkla{\Gamma}} \rkla{\norm{\widetilde{\phi_\Gamma}_1-\widetilde{\phi_\Gamma}_2}_{H^1\trkla{\Gamma}} +\tabs{\trkla{\widetilde{\phi}_1-\widetilde{\phi}_2}_\Om}}\\
\leq&\,C_\alpha \rkla{1+ \norm{\widetilde{\phi_\Gamma}_1}_{H^1\trkla{\Gamma}}^4 +\norm{\widetilde{\phi_\Gamma}_2}_{H^1\trkla{\Gamma}}^4}\rkla{\norm{\widetilde{\phi_\Gamma}_1-\widetilde{\phi_\Gamma}_2}_{L^2\trkla{\Gamma}}^2 + \norm{\widetilde{\phi}_1-\widetilde{\phi}_2}_*^2} \\
&+ \alpha\norm{\nablaG\widetilde{\phi_\Gamma}_1-\nablaG\widetilde{\phi_\Gamma}_2}_{L^2\trkla{\Gamma}}^2\,.
\end{split}
\end{align}
From \ref{item:potentials}, Hölder's inequality, Poincar\'e's inequality, and Young's inequality, we obtain
\begin{align}
\begin{split}
\tabs{R_3}\leq&\, C\norm{1+\abs{\widetilde{\phi}_1}^2+\abs{\widetilde{\phi}_2}^2}_{L^3\trkla{\Om}}\norm{\widetilde{\phi}_1-\widetilde{\phi}_2}_{L^2\trkla{\Om}}\norm{\nabla\widetilde{\phi}_1-\nabla\widetilde{\phi}_2}_{L^2\trkla{\Om}}\\
\leq&\,C_\alpha\rkla{1+\norm{\widetilde{\phi}_1}_{H^1\trkla{\Om}}^4+\norm{\widetilde{\phi}_2}_{H^1\trkla{\Om}}^4}\norm{\widetilde{\phi}_1-\widetilde{\phi}_2}_{L^2\trkla{\Om}}^2+\alpha\norm{\nabla\widetilde{\phi}_1-\nabla\widetilde{\phi}_2}_{L^2\trkla{\Om}}^2\,.
\end{split}
\end{align}
Recalling the definition of $\norm{\cdot}_*$ in \eqref{eq:def:dualnorm}, we compute
\begin{multline}\label{eq:Ehrling}
\norm{\widetilde{\phi}_1-\widetilde{\phi}_2}_{L^2\trkla{\Om}}^2 = \norm{\widetilde{\phi}_1-\widetilde{\phi}_2-\trkla{\widetilde{\phi}_1-\widetilde{\phi}_2}_\Om}_{L^2\trkla{\Om}}^2+\trkla{\widetilde{\phi}_1-\widetilde{\phi}_2}_\Om^2\\
\leq \abs{\widetilde{\phi}_1-\widetilde{\phi}_2}_* \norm{\nabla\widetilde{\phi}_1-\nabla\widetilde{\phi}_2}_{L^2\trkla{\Om}}+\trkla{\widetilde{\phi}_1-\widetilde{\phi}_2}_\Om^2\,.
\end{multline}
Hence, by Young's inequality, we obtain
\begin{align}
\tabs{R_3}\leq C_\alpha\rkla{1+\norm{\widetilde{\phi}_1}_{H^1\trkla{\Om}}^8+\norm{\widetilde{\phi}_2}_{H^1\trkla{\Om}}^8}\norm{\widetilde{\phi}_1-\widetilde{\phi}_2}_*^2 +2\alpha\norm{\nabla\widetilde{\phi}_1-\nabla\widetilde{\phi}_2}_{L^2\trkla{\Om}}^2\,.
\end{align}
We continue by deducing estimates for $I_4$. 
Using the Lipschitz continuity of $\varrho$ (cf.~\ref{item:rho}), we obtain using \ref{item:color}, \eqref{eq:Ehrling}, and Young's inequality
\begin{multline}
\tfrac12\sum_{k\in\mathds{Z}}\norm{\rkla{\varrho\trkla{\widetilde{\phi}_1}-\varrho\trkla{\widetilde{\phi}_2}}\lambda_k\g{k}}_*^2\leq C \sum_{k\in\mathds{Z}}\norm{\rkla{\varrho\trkla{\widetilde{\phi}_1}-\varrho\trkla{\widetilde{\phi}_2}}\lambda_k\g{k}}_{L^2\trkla{\Om}}^2\\
\leq C\norm{\widetilde{\phi}_1-\widetilde{\phi}_2}_{L^2\trkla{\Om}}^2\leq C_\alpha\norm{\widetilde{\phi}_1-\widetilde{\phi}_2}_*^2 +\alpha\norm{\nabla\widetilde{\phi}_1-\nabla\widetilde{\phi}_2}_{L^2\trkla{\Om}}^2\,.
\end{multline}
Similar computations provide $\widetilde{\Prob}$-a.s.
\begin{align}
\tabs{\Gamma}\sum_{k\in\mathds{Z}}\trkla{\trkla{\varrho\trkla{\widetilde{\phi}_1}-\varrho\trkla{\widetilde{\phi}_2}}\lambda_k\g{k}}_\Om^2\leq C_\alpha\norm{\widetilde{\phi}_1-\widetilde{\phi}_2}_*^2+\alpha\norm{\nabla\widetilde{\phi}_1-\nabla\widetilde{\phi}_2}_{L^2\trkla{\Om}}^2\,,
\end{align}
\begin{align}
\tfrac12\sum_{k\in\mathds{Z}}\norm{\trace{\rkla{\varrho\trkla{\widetilde{\phi}_1}-\varrho\trkla{\widetilde{\phi}_2}}\lambda_k\g{k}}}_{L^2\trkla{\Gamma}}^2\leq C\norm{\widetilde{\phi_\Gamma}_1-\widetilde{\phi_\Gamma}_2}_{L^2\trkla{\Gamma}}^2\,,
\end{align}
and
\begin{multline}
\sum_{k\in\mathds{Z}}\trkla{\trkla{\varrho\trkla{\widetilde{\phi}_1}-\varrho\trkla{\widetilde{\phi}_2}}\lambda_k\g{k}}_\Om \trkla{\trace{\trkla{\varrho\trkla{\widetilde{\phi}_1}-\varrho\trkla{\widetilde{\phi}_2}}\lambda_k\g{k}}}_\Gamma\\
\leq C_\alpha\norm{\widetilde{\phi}_1-\widetilde{\phi}_2}_*^2+ C\norm{\widetilde{\phi_\Gamma}_1-\widetilde{\phi_\Gamma}_2}_{L^2\trkla{\Gamma}}^2 +\alpha\norm{\nabla\widetilde{\phi}_1-\nabla\widetilde{\phi}_2}_{L^2\trkla{\Gamma}}^2\,.
\end{multline}
Combining the above results, we obtain for $\alpha$ sufficiently small
\begin{align}
\begin{split}
I_2&+I_3+I_4\\
\leq&\, \widetilde{\mathds{E}}\left[\int_0^{\hat{T}}e^{-\int_0^t\trkla{\widehat{C}-\vartheta\trkla{\widetilde{\phi}_1,\widetilde{\phi}_2,\widetilde{\phi_\Gamma}_1,\widetilde{\phi_\Gamma}_2}}\ds} \rkla{\norm{\widetilde{\phi}_1-\widetilde{\phi}_2}_*^2+\norm{\widetilde{\phi_\Gamma}_1-\widetilde{\phi_\Gamma}_2}_{L^2\trkla{\Gamma}}^2} \right. \\
&\left.\qquad\times\rkla{C_1+C_2\rkla{1+\norm{\widetilde{\phi}_1}_{H^1\trkla{\Om}}^8+\norm{\widetilde{\phi}_2}_{H^1\trkla{\Om}}^8 +\norm{\widetilde{\phi_\Gamma}_1}_{H^1\trkla{\Gamma}}^4+\norm{\widetilde{\phi_\Gamma}_2}_{H^1\trkla{\Gamma}}^4}  
}\dt \right]\\
\leq&\,\widetilde{\mathds{E}}\left[\int_0^{\hat{T}}\rkla{C_1+C_2\rkla{1+\norm{\widetilde{\phi}_1}_{H^1\trkla{\Om}}^8+\norm{\widetilde{\phi}_2}_{H^1\trkla{\Om}}^8 +\norm{\widetilde{\phi_\Gamma}_1}_{H^1\trkla{\Gamma}}^4+\norm{\widetilde{\phi_\Gamma}_2}_{H^1\trkla{\Gamma}}^4}  
}\right.\\
&\qquad\left.\times\,\mathfrak{R}\trkla{t,\trkla{\widetilde{\phi}_1,\widetilde{\phi_\Gamma}_1}-\trkla{\widetilde{\phi}_2,\widetilde{\phi_\Gamma}_2}}\dt\right]\,.
\end{split}
\end{align}
Choosing 
\begin{align}
\vartheta\trkla{\widetilde{\phi}_1,\widetilde{\phi}_2,\widetilde{\phi_\Gamma}_1,\widetilde{\phi_\Gamma}_2}=C_2\rkla{\norm{\widetilde{\phi}_1}_{H^1\trkla{\Om}}^8+\norm{\widetilde{\phi}_2}_{H^1\trkla{\Om}}^8+\norm{\widetilde{\phi_\Gamma}_1}_{H^1\trkla{\Gamma}}^4+\norm{\widetilde{\phi_\Gamma}_2}_{H^1\trkla{\Gamma}}^4}
\end{align}
and $\widehat{C}$ sufficiently large, we obtain
\begin{align}
\expectedt{\mathfrak{R}\trkla{\hat{T},\trkla{\widetilde{\phi}_1\trkla{\hat{T}},\widetilde{\phi_\Gamma}_1\trkla{\hat{T}}}-\trkla{\widetilde{\phi}_2\trkla{\hat{T}},\widetilde{\phi_\Gamma}_2\trkla{\hat{T}}}}}\leq 0\,.
\end{align}
Due to the already established regularity of the martingale solutions to \eqref{eq:weakform}, we have
\begin{align}
\int_0^T \vartheta\trkla{\widetilde{\phi}_1,\widetilde{\phi}_2,\widetilde{\phi_\Gamma}_1,\widetilde{\phi_\Gamma}_2} \dt<\infty
\end{align}
$\widetilde{\Prob}$-almost surely.
Together with the continuity $\trkla{\widetilde{\phi}_1,\widetilde{\phi_\Gamma}_1}$ and $\trkla{\widetilde{\phi}_2,\widetilde{\phi_\Gamma}_2}$ in $L^2\trkla{\Om}\times L^2\trkla{\Gamma}$ this implies the pathwise uniqueness of martingale solutions.
\end{proof}

\section{Convergence towards strong solutions}\label{sec:strongsolutions}
In this section, we shall prove Theorem \ref{thm:strongsolutions}.
Hence, we assume that for a given a filtered probability space $\trkla{\Omega,\mathcal{A},\mathcal{F},\Prob}$ with a $\mathcal{Q}$-Wiener process $W$ satisfying \ref{item:W1} and \ref{item:W2}, we have a finite dimensional approximation
\begin{align}\label{eq:finiteQWienerNEU}
\bs{\xi}\h^{m,\tau}=\sum_{k\in\mathds{Z}_h}\rkla{W\trkla{t^m},\g{k}}_{L^2\trkla{\Om}}\g{k}\,,
\end{align}
satisfying \ref{item:filtration}, \ref{item:increment}, \ref{item:randomvars}, and \ref{item:color}.
As this is a specialization of the previously discussed setting, the prior results remain valid.
In particular, we still have the pathwise uniqueness of martingale solutions.
Hence, we can apply a generalization of the Gyöngy--Krylov characterization of convergence in probability (cf.~\cite{GyongyKrylov1996}) to the setting of quasi-Polish spaces that was established in \cite{BreitFeireislHofmanova}.\\
As shown before, our numerical scheme \eqref{eq:modeldisc} has a pathwise unique solution for any $h$ and $\tau$.
Hence, there exists a sequence of stochastic processes defined on $\trkla{\Omega,\mathcal{A},\mathcal{F},\Prob}$ satisfying
\begin{subequations}
\begin{multline}
\iO\Ih{\trkla{\phi\h\tl\trkla{t}-\phi\h\tm}\psi\h}\dx+\trkla{t-t\no}\iO\nabla\mu\h\tp\cdot\nabla\psi\h\dx\\
=\frac{t-t\no}\tau\iO\Ih{\Phi\h\trkla{\phi\h\tm}\trkla{W\trkla{t}-W\trkla{t\no}}\psi\h}\dx\,,
\end{multline}
\begin{multline}
\iGamma\IhG{\trace{\phi\h\tl\trkla{t}-\phi\h\tm}\widehat{\psi}\h}\dG+\trkla{t-t\no}\iGamma\IhG{\theta\h\tp\widehat{\psi}\h}\dG\\
=\frac{t-t\no}{\tau}\iGamma\IhG{\trace{\Phi\h\trkla{\phi\h\tm}\trkla{W\trkla{t}-W\trkla{t\no}}}\widehat{\psi}\h}\dG\,,
\end{multline}
and
\begin{align}
\begin{split}
\iO&\Ih{\mu\h\tp\eta\h}\dx+\iGamma\IhG{\theta\h\tp\trace{\eta\h}}\dG\\
=&\,\iO\nabla\phi\h\tp\cdot\nabla\eta\h\dx+\iGamma\nablaG\trace{\phi\h\tp}\cdot\nablaG\trace{\eta\h}\dG\\
&+\frac{r\h\tp}{\sqrt{E\h^\Om\trkla{\phi\h\tm}}}\iO\Ih{F^\prime\trkla{\phi\h\tm}\eta\h}\dx+\iO\Ih{\Xi_{h,\Om}\tp\eta\h}\dx\\
&+\frac{s\h\tp}{\sqrt{E\h^\Gamma\trkla{\trace{\phi\h\tm}}}}\iGamma\IhG{G^\prime\rkla{\trace{\phi\h\tm}}\trace{\eta\h}}\dG+\iGamma\IhG{\Xi_{h,\Gamma}\tp\trace{\eta\h}}\dG\,
\end{split}
\end{align}
\end{subequations}
for all $\psi\h,\eta\h\in\Uh$ and $\widehat{\psi}\h\in\UhG$.
Here, $\Xi_{h,\Om}\tp$ and $\Xi_{h,\Gamma}\tp$ are the piecewise constant in time interpolation of $\Xi_{h,\Om}\nn$ and $\Xi_{h,\Gamma}\nn$ which were defined in \eqref{eq:deferrorterms}.
As shown in Lemma \ref{lem:potential}, these terms will vanish for $\tau\searrow0$.\\
From this sequence of stochastic processes, we excerpt an arbitrary pair of subsequences which we shall denote by $a_\alpha:=\rkla{\phi_{h_\alpha}^{\tau_\alpha},\trace{\phi_{h_\alpha}^{\tau_\alpha}},r_{h_\alpha}^{\tau_\alpha},s_{h_\alpha}^{\tau_\alpha},\mu_{h_\alpha}^{\tau_\alpha,+}, l\tekla{\mu_{h_\alpha}^{\tau_\alpha,+}},\theta_{h,\alpha}^{\tau_\alpha,+} }_{\alpha\in\mathds{N}}$ and $a_\beta:=\rkla{\phi_{h_\beta}^{\tau_\beta},\trace{\phi_{h_\beta}^{\tau_\beta}},r_{h_\beta}^{\tau_\beta},s_{h_\beta}^{\tau_\beta},\mu_{h_\beta}^{\tau_\beta,+}, l\tekla{\mu_{h_\beta}^{\tau_\beta,+}},\theta_{h,\beta}^{\tau_\beta,+} }_{\beta\in\mathds{N}}$.
For these subsequences, we consider their joint laws $\trkla{\nu_{\alpha,\beta}}_{\alpha,\beta\in\mathds{N}}$ on $\mathcal{Y}\times\mathcal{Y}$ with
\begin{align}
\begin{split}
\mathcal{Y}:=&\,C\trkla{\tekla{0,T};L^s\trkla{\Om}}\times C\trkla{\tekla{0,T};L^r\trkla{\Gamma}}\times L^2\trkla{0,T}_{\weaktop}\times L^2\trkla{0,T}_{\weaktop}\\
&\quad\times L^2\trkla{0,T;H^1\trkla{\Om}}_{\weaktop}\times L^2\trkla{0,T;\trkla{H^1\trkla{\Om}}^\prime}_{\weaktop}\times L^2\trkla{0,T;L^2\trkla{\Gamma}}
\end{split}
\end{align}
and $s\in[1,\tfrac{2d}{d-2})$ and $r\in[1,\tfrac{2\trkla{d-1}}{d-3})$.
In the following we will show that for the sequence of joint laws, there exists a further subsequence which converges weakly to a probability measure $\nu$ such that
\begin{align}\label{eq:limitmeasure}
\nu\trkla{\trkla{a,b}\in\mathcal{Y}\times\mathcal{Y}\,:\,a=b}=1\,.
\end{align}
Then Proposition A.4 in \cite{BreitFeireislHofmanova} provides $\Prob$-almost sure convergence in the topology of $\mathcal{Y}$ for a subsequence $\rkla{\phi\h\tl,\trace{\phi\h\tl},r\h\tl,s\h\tl,\mu\h\tp, l\tekla{\mu\h\tp},\theta\h\tp }_{h,\tau}$.
Together with the pathwise uniqueness, this guarantees convergence for the complete sequence.\\
To establish \eqref{eq:limitmeasure}, we define the extended path space
\begin{align}
\widehat{\mathcal{X}}:=\mathcal{Y}\times\mathcal{Y}\times C\trkla{\tekla{0,T};H^1\trkla{\Om}}\,,
\end{align}
and consider the sequence
\begin{align}
\trkla{z_{\alpha\beta}}_{\alpha,\beta\in\mathds{N}}:=\rkla{a_\alpha,a_\beta,W}_{\alpha,\beta\in\mathds{N}}\,.
\end{align}
Recalling the arguments from Section \ref{sec:compactness}, we note that the joint laws of this sequence are tight on $\widehat{\mathcal{X}}$.
Hence, repeating the arguments of Theorem \ref{thm:jakubowski}, we obtain obtain for a subsequence $\trkla{z_{\alpha_i \beta_i}}_{i\in\mathds{N}}$ the existence of a probability space $\trkla{\widetilde{\Omega},\widetilde{\mathcal{A}},\widetilde{\Prob}}$ and a sequence of random variables
\begin{align}
\trkla{\widetilde{z}_i}_{i\in\mathds{N}}=\trkla{\widetilde{a}_{\alpha_i},\widetilde{a}_{\beta_i},\widetilde{W}_i}_{i\in\mathds{N}}
\end{align}
defined on $\trkla{\widetilde{\Omega},\widetilde{\mathcal{A}},\widetilde{\Prob}}$ that has the same law as $\trkla{z_{\alpha_i \beta_i}}_{i\in\mathds{N}}$ and converges towards random variables
\begin{align}
\rkla{\widetilde{\phi},\widetilde{\phi_\Gamma},\widetilde{r},\widetilde{s},\widetilde{\mu},\widetilde{L},\widetilde{\theta},\widehat{\phi},\widehat{\phi_\Gamma},\widehat{r},\widehat{s},\widehat{\mu},\widehat{L},\widehat{\theta},\widetilde{W}}\,.
\end{align}
Arguing as in Lemma \ref{lem:convergence}, we obtain $\trace{\widetilde{\phi}}=\widetilde{\phi_\Gamma}$, $\widetilde{r}=\sqrt{\iO F\trkla{\widetilde{\phi}}\dx}$, $\widetilde{s}=\sqrt{\iGamma G\trkla{\widetilde{\phi_\Gamma}}\dG}$, $\widetilde{L}=l\tekla{\widetilde{\mu}}$, $\trace{\widehat{\phi}}=\widehat{\phi_\Gamma}$, $\widehat{r}=\sqrt{\iO F\trkla{\widehat{\phi}}\dx}$, $\widehat{s}=\sqrt{\iGamma G\trkla{\widehat{\phi_\Gamma}}\dG}$, and $\widehat{L}=l\tekla{\widehat{\mu}}$ $\widetilde{\Prob}$-almost surely almost everywhere.
Following the arguments of Section \ref{sec:limit}, we can show that $\trkla{\widetilde{\phi},\widetilde{\phi_\Gamma},\widetilde{\mu},\widetilde{\theta},\widetilde{W}}$ and $\trkla{\widehat{\phi},\widehat{\phi_\Gamma},\widehat{\mu},\widehat{\theta},\widetilde{W}}$ are both martingale solutions satisfying \eqref{eq:weakform} with the same initial conditions and the same Wiener process. 
Due to the pathwise uniqueness established in Section \ref{sec:uniqueness} we have $\widetilde{\phi}=\widehat{\phi}$, $\widetilde{\phi_\Gamma}=\widehat{\phi_\Gamma}$, $\widetilde{\mu}=\widehat{\mu}$, $\widetilde{\theta}=\widehat{\theta}$ $\widetilde{\Prob}$-almost surely.
As a consequence, we also obtain $\widetilde{r}=\widehat{r}$, $\widetilde{s}=\widehat{s}$, and $\widetilde{L}=\widehat{L}$.
Hence, by equality of laws, we obtain \eqref{eq:limitmeasure}.\\
As this guarantees $\Prob$-almost sure convergence in the topology of $\mathcal{Y}$ for a subsequence of $\rkla{\phi\h\tl,\trace{\phi\h\tl},r\h\tl,s\h\tl,\mu\h\tp, l\tekla{\mu\h\tp},\theta\h\tp }_{h,\tau}$, we have a convergence result similar to the one stated in Theorem \ref{thm:jakubowski} without introducing a new probability space.
Consequently, we can repeat the arguments of the previous sections to conclude the proof of Theorem \ref{thm:strongsolutions}.

\section{Conclusion}
For small droplets, contact line tension has a major influence on the evolution of the three-phase contact line.
We studied a diffuse interface model consisting of a Cahn--Hilliard equation with Allen--Cahn type dynamic boundary conditions.
The influence of thermal fluctuations is included in this model via multiplicative noise terms in the bulk and on the boundary.
We investigated the well-posedness and the numerical treatment of the arising stochastic PDEs.
We started by proposing a fully discrete, linear finite element scheme based on a recently developed augmented version of the scalar auxiliary variable method.
Based on a priori estimates, we established the existence of pathwise unique, (stochastically) strong solutions by showing convergence of the discrete solution to our proposed scheme.\\
As the investigation of the influence of thermal fluctuations requires simulating multiple individual paths, the numerical exploration of \eqref{eq:model} which is already challenging in the deterministic case, becomes even more expensive.
Hence, we expect the linearity of the scheme to be a key tool opening a pathway for further numerical investigations of the effect of contact line tension on the wetting behavior and contact angle of droplets.
\bibliographystyle{amsplain}
\providecommand{\bysame}{\leavevmode\hbox to3em{\hrulefill}\thinspace}
\providecommand{\MR}{\relax\ifhmode\unskip\space\fi MR }
\providecommand{\MRhref}[2]{%
  \href{http://www.ams.org/mathscinet-getitem?mr=#1}{#2}
}
\providecommand{\href}[2]{#2}

\end{document}